\documentclass[11pt,reqno,a4paper]{amsart}
\usepackage{amsmath,amsthm,verbatim,amscd,amssymb,setspace,hyperref,esint,mathtools}
\usepackage{exscale,color}

\usepackage{mathrsfs}

\usepackage{enumitem}

\renewcommand{\epsilon}{\varepsilon}
\newcommand{\N}{\mathbb{N}}
\newcommand{\Z}{\mathbb{Z}}

\newcommand{\R}{\mathbb{R}}
\newcommand{\C}{\mathbb{C}}
\renewcommand{\P}{\mathbb{P}}

\renewcommand{\Re}{\operatorname{Re}}

\newcounter{mtheorem}
\newtheorem{mtheorem}[mtheorem]{Theorem}

\newcommand{\supp}{\operatorname{supp}}

\renewcommand{\P}{\mathbb{P}}
\newcommand{{\vol}}{\rm vol}
\newcommand{\p}{\partial}
\newcommand{\norm}[1]{\Vert #1 \Vert}

\newcommand{\Ric}{\operatorname{Ric}}

\newcommand{\Rm}{\operatorname{Rm}}

\providecommand{\norm}[1]{\lVert#1\rVert}

\def\tr{\operatorname{tr}}

\def\inj{\operatorname{inj}}

\def\Id{\operatorname{Id}}

\def \Cstarn{(\mathbb{C}^*)^n}

\def \Cstar{\mathbb{C}^*}
\def \t {\mathfrak{t}}
\def \bp {\bar{\partial}}
\def \Pol {P_{-K_M}}

\def\supp{\operatorname{supp}}
\def\tr{\operatorname{tr}}
\def\Ric{\operatorname{Ric}}
\def\tr{\operatorname{tr}}

\def\inj{\operatorname{inj}}

\def\Id{\operatorname{Id}}

\def\diam{\operatorname{diam}}
\def\vol{\operatorname{vol}}

\newtheoremstyle{fancy}{}{}{\itshape}{}{\textbf\bgroup}{.\egroup}{ }{}
\newtheoremstyle{fancy2}{}{}{\rm}{}{\textbf\bgroup}{.\egroup}{ }{}

\theoremstyle{fancy}
\newtheorem{theorem}{Theorem}[section]
\newtheorem{lemma}[theorem]{Lemma}
\newtheorem{corollary}[theorem]{Corollary}

\newtheorem{prop}[theorem]{Proposition}

\theoremstyle{fancy2}
\newtheorem{definition}[theorem]{Definition}
\newtheorem{example}[theorem]{Example}
\newtheorem{remark}[theorem]{Remark}

\newtheorem{claim}[theorem]{Claim}

\setlist{leftmargin=*}

\numberwithin{equation}{section}

\begin{document}
\title{An Aubin continuity path for shrinking gradient K\"ahler-Ricci solitons}
\date{\today}

\author{Charles Cifarelli}
\address{Mathematics Department, Stony Brook University, 100 Nicolls Road, Stony Brook, NY 11794}
\email{charles.cifarelli@stonybrook.edu}
\author{Ronan J.~Conlon}
\address{Department of Mathematical Sciences, The University of Texas at Dallas, Richardson, TX 75080}
\email{ronan.conlon@utdallas.edu}
\author{Alix Deruelle}
\address{Sorbonne Universit\'e and Universit\'e de Paris, CNRS, IMJ-PRG, F-75005 Paris, France}
\email{alix.deruelle@imj-prg.fr}

\date{\today}

\begin{abstract}
Let $D$ be a toric K\"ahler-Einstein Fano manifold. We
show that any toric shrinking gradient K\"ahler-Ricci soliton
on certain toric blowups of $\mathbb{C}\times D$ satisfies a complex Monge-Amp\`ere equation. We then set up an Aubin continuity path to
solve this equation and show that it has a solution at the initial value of the path parameter.
This we do by implementing another continuity method.
\end{abstract}

\maketitle

\markboth{Charles Cifarelli, Ronan J.~Conlon, and Alix Deruelle}{An Aubin path for shrinking gradient K\"ahler-Ricci solitons}

\tableofcontents

\section{Introduction}

\subsection{Overview}\label{overview}
A \emph{Ricci soliton} is a triple $(M,\,g,\,X)$, where $M$ is a Riemannian manifold endowed with a complete Riemannian metric $g$
and a complete vector field $X$, such that
\begin{equation}\label{hot}
\Ric_{g}+\frac{1}{2}\mathcal{L}_{X}g=\lambda g
\end{equation}
for some $\lambda\in\mathbb{R}$. The vector field $X$ is called the
\emph{soliton vector field}. If $X=\nabla^{g} f$ for some smooth real-valued function $f$ on $M$,
then we say that $(M,\,g,\,X)$ is \emph{gradient}. In this case, the soliton equation \eqref{hot}
becomes $$\Ric_{g}+\operatorname{Hess}_{g}(f)=\lambda g,$$
and we call $f$ the \emph{soliton potential}. In the case of gradient Ricci solitons, the completeness of $X$ is guaranteed by the completeness of $g$
\cite{zhang12}.

Let $(M,\,g,\,X)$ be a Ricci soliton. If $g$ is K\"ahler and $X$ is real holomorphic, then we say that $(M,\,g,\,X)$ is a \emph{K\"ahler-Ricci soliton}. Let $\omega$ denote the K\"ahler
form of $g$. If $(M,\,g,\,X)$ is in addition gradient, then \eqref{hot} may be rewritten as
\begin{equation}\label{krseqn}
\rho_{\omega}+i\partial\bar{\partial}f=\lambda\omega,
\end{equation}
where $\rho_{\omega}$ is the Ricci form of $\omega$ and $f$ is the soliton potential.

Finally, a Ricci soliton and a K\"ahler-Ricci soliton are called \emph{steady} if $\lambda=0$, \emph{expanding}
if $\lambda<0$, and \emph{shrinking} if $\lambda>0$ in \eqref{hot}.
One can always normalise $\lambda$, when non-zero, to satisfy $|\lambda|=1$. We henceforth assume that this is the case.

Ricci solitons are interesting both from the point of view of canonical metrics and of the Ricci flow. On one hand, they represent one direction in which the concept of an Einstein manifold can be generalised. On compact manifolds, shrinking Ricci solitons are known to exist in several instances where there are obstructions to the existence of Einstein metrics; see for example \cite{soliton}. By the maximum principle, there are no nontrivial expanding or steady Ricci solitons on compact manifolds. However, there are many examples on noncompact manifolds; see for example \cite{conlon33,cds, futaki3} and the references therein. On the other hand, one can associate to a Ricci soliton a self-similar solution of the Ricci flow, and gradient shrinking Ricci solitons in particular provide models for finite-time Type I singularities of the flow \cite{topping, naber}. From this perspective, it is an important problem to classify such solitons in order to better understand singularity development along the Ricci flow.

In this article, we are concerned with the construction of shrinking gradient K\"ahler-Ricci solitons, models for finite-time Type I singularities of the K\"ahler-Ricci flow. In essence, we set up an Aubin continuity path for a complex Monge-Amp\`ere equation to construct such solitons in a particular geometric setting that allows for control on the data of the equation. We then show that there is a solution to the equation for the initial value of the path parameter. This we do by implementing another continuity path.

\subsection{Main result}\label{sec:main}

In order to state the main result, recall that a complex toric manifold is a smooth $n$-dimensional complex manifold $D$ endowed with an effective holomorphic action of the complex
torus $(\mathbb{C}^{*})^{n}$ with a compact fixed point set. In such a setting, there always exists an orbit $U\subset D$ of the $(\mathbb{C}^{*})^{n}$-action
which is open and dense in $D$. The $(\mathbb{C}^{*})^{n}$-action of course determines the holomorphic action of a real torus
$T^{n}\subset(\mathbb{C}^{*})^{n}$, as is easily seen for the action of the one-dimensional torus
$\mathbb{C}^{*}$ on $\mathbb{P}^{1}$ via $\lambda\cdot[z_{0}:z_{1}]\mapsto[\lambda z_{1}:z_{2}]$.
{This assumption is crucial for obtaining a uniform lower bound on the solution along our continuity path.}
Our main result is stated as follows.
\begin{mtheorem}\label{mainthm}
Let $D^{n-1}$ be a toric K\"ahler-Einstein Fano manifold of complex dimension $n-1$ with K\"ahler form $\omega_{D}$ and Ricci form $\rho_{\omega_{D}}=\omega_{D}$, and consider $\mathbb{P}^{1}\times D$ with the induced product torus action acting by rotation on the $\mathbb{P}^{1}$-factor.
Let $T^{n}$ denote the real torus acting on $\mathbb{P}^{1}\times D$, write $D_{x}:=\{x\}\times D$, and let $\overline{M}$ be a toric Fano manifold obtained as a torus-equivariant (possibly iterated) blowup $\pi:\overline{M}\to\mathbb{P}^{1}\times D$ along smooth torus-invariant subvarieties contained in $D_{0}$.
Let $M:=\overline{M}\setminus\pi^{-1}(D_{\infty})$, $\widehat{M}:=\mathbb{C}\times D$, write $J$ for the complex structure on $M$, $\mathfrak{t}$ for the Lie algebra of $T^{n}$,
and let $z$ denote the holomorphic coordinate on the $\mathbb{C}$-factor of $\widehat{M}$. Then:
\begin{enumerate}
\item There exists a unique complete real holomorphic vector field $JX\in\mathfrak{t}$ such that $X$ is the soliton vector field of any complete toric shrinking gradient K\"ahler-Ricci soliton on $M$.
\end{enumerate}
Assume that the flow-lines of $JX$ are closed. Then:
\begin{enumerate}\setcounter{enumi}{1}
  \item  There exists a complete K\"ahler metric $\omega$ on $M$ invariant under the action of $T$,
   $\lambda>0$ uniquely determined by $X$, and a holomorphic isometry $\nu:(M\setminus K,\,\omega)\to(\widehat{M}\setminus\widehat{K},\,\widehat{\omega}:=\omega_{C}+\omega_{D})$,
where $K\subset M,\,\widehat{K}\subset\widehat{M}$, are compact and $\omega_{C}:=\frac{i}{2}\partial\bar{\partial}|z|^{2\lambda}$,
such that $d\nu(X)=\frac{2}{\lambda}\cdot\operatorname{Re}\left(z\partial_{z}\right)$.
\item There exists a unique torus-invariant function $f\in C^{\infty}(M)$ such that
$-\omega\lrcorner JX=df$. Moreover, $f=\nu^{*}\left(\frac{|z|^{2\lambda}}{2}-1\right)$ and
$\Delta_{\omega}f+f-\frac{X}{2}\cdot f=0$ outside a compact subset of $M$ containing $K$.
  \item Any shrinking K\"ahler-Ricci soliton on $M$ invariant under the action of $T$
  of the form $\omega+i\partial\bar{\partial}\varphi$ for some $\varphi\in C^{\infty}(M)$
with $\omega+i\partial\bar{\partial}\varphi>0$ satisfies the complex Monge-Amp\`ere equation
\begin{equation}\label{cmaa}
(\omega+i\partial\bar{\partial}\varphi)^{n}=e^{F+\frac{X}{2}\cdot\varphi-\varphi}\omega^{n},
\end{equation}
where $F\in C^{\infty}(M)$ is equal to a constant outside a compact subset of $M$ and is determined by the fact that
$$\rho_{\omega}+\frac{1}{2}\mathcal{L}_{X}\omega-\omega=i\partial\bar{\partial}F\qquad\textrm{and}\qquad\int_{M}(e^{F}-1)e^{-f}\omega^{n}=0.$$
Here, $\rho_{\omega}$ denotes the Ricci form of $\omega$.
\item There exists a function $\psi\in C^{\infty}(M)$ invariant under the action of $T$ and with $\omega+i\partial\bar{\partial}\psi>0$ such that
\begin{equation}\label{ronan}
(\omega+i\partial\bar{\partial}\psi)^{n}=e^{F+\frac{X}{2}\cdot\psi}\omega^{n},
\end{equation}
where $\int_{M}\psi\,e^{-f}\omega^{n}=0$ and outside a compact subset, $\psi=c_{1}\log f+c_{2}+\vartheta$ for some constants $c_{1},\,c_{2}\in\mathbb{R}$ and
a smooth real-valued function $\vartheta:M\to\mathbb{R}$ satisfying
\begin{equation*}
|\nabla^{i}\mathcal{L}^{(j)}_{X}\vartheta|_{\omega}=O(f^{-\frac{\beta}{2}})\qquad\textrm{for all $i,\,j\in\mathbb{N}$},\quad \beta\in(0,\lambda^D).
\end{equation*}
Here, $\nabla$ denotes the Levi-Civita connection associated to $\omega$, $\mathcal{L}^{(j)}_{X}=\underbrace{\mathcal{L}_{X}\circ\ldots\circ\mathcal{L}_{X}}_{j-times}$,
and $\lambda^D$ is the first non-zero eigenvalue of $-\Delta_{D}$ acting on $L^2$-functions on $D$.
\end{enumerate}
\end{mtheorem}

Note that since $M$ does not split off any $S^{1}$-factors, toricity implies that $M$ has finite fundamental group \cite{cox}, a necessary condition for the existence of a shrinking gradient K\"ahler-Ricci soliton on $M$ \cite{wyliee}. {Note also that throughout, our convention for the K\"ahler Laplacian $\Delta_{\omega}$ is that with respect to the K\"ahler form $\omega$, $\Delta_{\omega}f=\operatorname{tr}_{\omega}\left(i\partial\bar{\partial}f\right)$ for $f$ a smooth real-valued function, so that the eigenvalues of minus the Laplacian are non-negative on a compact Riemannian manifold.}

Part (i) of the theorem determines the soliton vector field of any complete toric shrinking gradient K\"ahler-Ricci soliton on $M$ and follows immediately
from \cite[Theorem A]{charlie}, where it is asserted that a complete toric shrinking gradient K\"ahler-Ricci soliton is unique up to biholomorphism. The vector field $JX$ is characterised by the fact that it is the point in a specific open convex subset of $\mathfrak{t}$ at which a certain strictly convex functional attains its minimum. More precisely, because $H^{1}(M,\,\mathbb{R})=0$ and $M$ is toric, the action of $T$ is Hamiltonian and there exists a strictly convex functional $\mathcal{F}_{\omega}:\Lambda_{\omega}\to\mathbb{R}_{>0}$, the ``weighted volume functional'' \cite[Definition 5.16]{cds}, defined on an open convex cone $\Lambda_{\omega}\subset\mathfrak{t}$ uniquely determined by the image of $M$ under the moment map defined by the action of $T$ and the choice of $\omega$ \cite[Proposition 1.4]{wu} and well-defined by the non-compact version of the Duistermaat-Heckman formula \cite{wu} (see also \cite[Theorem A.3]{cds}). Because $T$ provides a full-dimensional torus symmetry, the domain $\Lambda_{\omega}$ of $\mathcal{F}_{\omega}$ and $\mathcal{F}_{\omega}$ itself only depend on the torus action \cite{ccd} so that both are independent of the choice of $\omega$. Furthermore, henceforth dropping the subscripts $\omega$, $\mathcal{F}$ is known to be strictly convex \cite[Lemma 5.17(i)]{cds}
and in addition proper \cite[Proposition 3.1]{charlie} on $\Lambda$ in the toric case, and so it must attain a unique minimum on $\Lambda$.
This minimum defines a distinguished point in $\mathfrak{t}$, namely the only vector field in $\mathfrak{t}$ that can admit a complete toric shrinking gradient K\"ahler-Ricci soliton \cite[Theorem 4.6]{charlie}. This is precisely the vector field $JX$ of Theorem \ref{mainthm}(i). Since everything is explicit and is determined by the torus action, one can a priori determine this vector field for a given $M$; see for example \cite[Section A.4]{cds}.

Parts (ii) and (iii) give a reference metric on $M$ that is isometric to a model shrinking gradient K\"ahler-Ricci soliton outside a compact set.
This requires the assumption that the flow-lines of $JX$ are closed. Indeed, this is the case for the soliton vector field on the model.
With respect to this background metric, part (iv) gives a complex Monge-Amp\`ere equation \eqref{cmaa} that any complete toric shrinking gradient K\"ahler-Ricci soliton on $M$ must satisfy with control on the asymptotics of the data $F$ of the equation. By \cite{charlie}, we know that there is at most one such soliton on $M$
and we expect that this equation has a solution, resulting in a complete toric shrinking gradient K\"ahler-Ricci soliton on $M$. Such a soliton should model finite time collapsing of the K\"ahler-Ricci flow in order to be consistent with \cite{tosatti10}. One may attempt to solve \eqref{cmaa} by implementing the Aubin continuity path that was introduced for K\"ahler-Einstein manifolds \cite[Section 7.26]{Aubin}. Specifically in our case, one may consider the path
\begin{equation}\label{ast-t}
\left\{
\begin{array}{rl}
(\omega+i\partial\bar{\partial}\varphi_{t})^{n}=e^{F+\frac{X}{2}\cdot\varphi_{t}-t\varphi_{t}}\omega^{n},&\quad\varphi\in C^{\infty}(M),\quad\mathcal{L}_{JX}\varphi=0,\quad\omega+i\partial\bar{\partial}\varphi>0,\quad t\in[0,\,1],\\
\int_{M}e^{F-f}\omega^{n}=\int_{M}e^{-f}\omega^{n}. &
\end{array} \right.\tag{$\ast_{t}$}
\end{equation}
The main content of Theorem \ref{mainthm} is part (v) where we provide a solution to the equation corresponding to $t=0$. This we do by implementing another continuity path. In the compact case,
this was achieved by Zhu \cite{Zhu-KRS-C1}.

The simplest example of a toric Fano manifold $D$ satisfying the conditions of Theorem \ref{mainthm} is $D=\mathbb{P}^{1}$ with $\pi$ the blowup map. Indeed, these choices result in
$M$ being the blowup of $\mathbb{C}\times\mathbb{P}^{1}$ at one point, a manifold for which the flow-lines of $JX$ close as one can see from  Example \ref{finaleg} or \cite[Example 2.33]{ccd}.
In \cite[Conjecture 1.1]{ccd}, $M$ was identified as a new manifold potentially admitting a (unique) complete shrinking gradient K\"ahler-Ricci soliton with bounded scalar curvature. Thanks to \cite{bamler1}, it is now known that $M$ admits such a soliton. However, the proof of existence in \cite{bamler1}
is strictly dimension dependent and is \emph{indirect} in that the soliton is constructed as a blowup limit of a specific K\"ahler-Ricci flow on the blowup of $\mathbb{P}^{1}\times\mathbb{P}^{1}$ at one point. The principal motivation behind Theorem \ref{mainthm} therefore is that it provides a first step in a \emph{direct} construction of this soliton on $M$, namely via the continuity method,
and is more widely applicable than the methods of \cite{bamler1}. It also serves to provide examples of non-compact manifolds with strictly positive Bakry-Emery tensor.

Equation \eqref{ronan} a priori looks identical to the complex Monge-Amp\`ere equation solved in \cite{conlon33}, where complete steady gradient K\"ahler-Ricci solitons were constructed. Even though
the equations appear the same and the same continuity path is used in both cases, there are several important differences between the two that result in
additional difficulties arising in the solution of \eqref{ronan} in contrast to the equation of \cite{conlon33}. We conclude this section by highlighting some of these differences.
\begin{itemize}
  \item On a closed K\"ahler manifold, the $X$-derivative of any K\"ahler potential is bounded prior to any other bound; see \cite{Zhu-KRS-C1}. This fact does not seem to be amenable to an arbitrary noncompact K\"ahler manifold and  represents one of the major obstacles to adapting Tian and Zhu's work \cite{Tian-Zhu-I} to our current setting. For us, not only is the drift operator $X$ of \eqref{ronan} unbounded, in contrast to \cite{conlon33} where it is bounded, but it also has the opposite sign. This prevents us from adapting the proof of the $C^0$ a priori estimate in \cite{conlon33} to the present situation.
  \item In \cite{Zhu-KRS-C1}, a generalisation of Calabi's conjecture was proved on compact K\"ahler manifolds using a continuity path that shrinks the hypothetical soliton vector field $X$ to zero as the path parameter tends to zero, thereby reducing the existence at the initial value of the path parameter to Yau's solution of the Calabi conjecture \cite{Calabiconj}. In our setting, implementing such a continuity path to solve \eqref{cmaa} does not preserve the weighted volume and indeed the weighted volume diverges at the initial value of the path parameter. This explains why the Aubin continuity path is more suited to solving \eqref{cmaa} which yields \eqref{ronan} at the initial value of the path parameter (in contrast to the Calabi-Yau equation). This is precisely the equation that we provide a solution to in Theorem \ref{mainthm}(v).
  \item In \cite{conlon33}, the corresponding equation was solved using the continuity path with exponentially weighted function spaces. Here, we solve
  \eqref{ronan} in polynomially weighted function spaces. This difference is derived from the fact that in the present situation, the linearised operator contains logarithmically growing functions in its kernel at infinity. This makes the linear theory more delicate than in the previous work \cite{conlon33}.
  \item In obtaining an a priori $C^{0}$-estimate for \eqref{ronan}, the toricity assumption is crucial. This was not the case in
  \cite{conlon33} where no toricity was required. However, a priori weighted $L^p$-estimates on the solution of \eqref{ronan} are obtained \emph{without} requiring toricity. The same also applies to the a priori estimates apart from the one concerning a \textit{lower} bound on the solution. This will all be made clear in the relevant statements throughout.
\item The order in which we obtain the a priori estimates differs to that of \cite{conlon33}. Here we first obtain an a priori \textit{lower} bound on the radial derivative of the solution. This then allows us to derive an a priori \textit{upper} bound on the solution. The next step is to derive an a priori \textit{lower} bound on the solution. At this stage, we follow the same strategy as that of \cite{conlon33} to obtain a priori \textit{local} estimates on the solution.
\item In addition to containing logarithmically growing functions, the kernel of the linearised operator in the present situation contains constants, a fact that makes the a priori weighted estimate of the difference of the solution and of its value at infinity more subtle in a nonlinear setting. To circumvent this issue, we apply the Bochner formula to the $X$-derivative of our solution with respect to the unknown K\"ahler metric.
\item Our geometric setting bears some resemblance to the work \cite{acyl} on asymptotically cylindrical Calabi-Yau metrics. However, in the context of metric measure spaces, our setting is somewhat dissimilar to theirs {in that as metric measure spaces, our spaces have finite volume, whereas their spaces have infinite volume. This forces us to take an alternative approach to obtain (weighted) a priori estimates.}
\end{itemize}

\subsection{Outline of paper}

We begin in Section \ref{sec-srs} by recalling the basics of shrinking Ricci and K\"ahler-Ricci solitons. Some important examples are discussed as well as some technical
lemmas proved. We also recall the definition of a metric measure space in Section \ref{metricmeasure}. In Section \ref{pooly}, we digress and define
polyhedrons and polyhedral cones before moving on to the definition of a Hamiltonian action in Section \ref{hamilton}.
Section \ref{toric-geom} then comprises the background material on toric geometry that we require.

In Section \ref{sec-construction-back-metric},  we construct a background metric with the desired properties, resulting in
the proof of Theorem \ref{mainthm}(ii). Next, in Section \ref{sec-set-up-CMA}, the complex Monge-Amp\`ere equation is set up and the normalisation of the Hamiltonian of $JX$ is obtained, leading to the proof of Theorem \ref{mainthm}(iii)--(iv). Our background metric is isometric to a shrinking gradient K\"ahler-Ricci soliton compatible with $X$ outside a compact set. This is what allows us to set up the complex Monge-Amp\`ere equation with compactly supported data.

From Section \ref{sec-poin-inequ} onwards, the content takes on a more analytic flavour with the proof of Theorem \ref{mainthm}(v) taking up Sections \ref{sec-poin-inequ}--\ref{sec-a-priori-est}. To prove this part of Theorem \ref{mainthm},
we implement the continuity method. The specific continuity path that we consider is outlined at the beginning of Section \ref{sec-a-priori-est} but beforehand, in Section \ref{sec-poin-inequ}, we
prove a Poincar\'e inequality which is the content of Proposition \ref{poincare}. This is essential in deriving the a priori weighted energy estimate for the complex Monge-Amp\`ere equation \eqref{cmaa} with compactly supported data.

In Section \ref{linear-theory-section}, we study the properties of the drift Laplacian of our background metric acting on polynomially weighted function spaces. More
precisely, we introduce polynomially weighted function spaces
whose elements are invariant under the flow of $JX$ in Section \ref{function-spaces-subsection}. We follow this up in Section \ref{linear} by showing that the drift Laplacian of our
background metric is an isomorphism between such spaces. This latter result is the content of Theorem \ref{iso-sch-Laplacian-pol}.
Using it, we then prove Theorem \ref{Imp-Def-Kah-Ste} that serves as the openness part of the continuity argument.
The closedness part involves a priori estimates and these make up Section \ref{sec-a-priori-est}.

As noted previously, the presence of the unbounded vector field $X$ makes the analysis much more involved. An a priori lower bound  for the radial derivative $X\cdot \psi$, where $\psi$ solves \eqref{cmaa}, has to be proved \emph{before} the a priori $C^0$ bound in order to avoid a circular argument; see Section \ref{sec-low-bd-rad-der}. A priori energy estimates are obtained in Section \ref{sec-a-priori-energy} through the use of the so-called Aubin-Tian-Zhu's functionals and result in an a priori upper bound on a solution to the complex Monge-Amp\`ere equation \eqref{cmaa}; cf.~Proposition \ref{prop-bd-abo-uni-psi}. As explained above, the invariance of the solution under the whole action torus is crucial in obtaining
an a priori lower bound on the infimum; cf.~Proposition \ref{prop-bd-bel-uni-psi}. Then and only then an a priori upper bound on the radial derivative of a solution to \eqref{cmaa} is derived; cf.~Proposition \ref{sec-upp-bd-rad-der}. Section \ref{sec-high-der} is devoted to proving a local bootstrapping phenomenon for \eqref{cmaa}. Finally, Section \ref{sec-wei-bd} takes care of establishing a priori weighted estimates at infinity for \eqref{cmaa}, leading to the completion of the proof of Theorem \ref{mainthm}(v) in Section \ref{sec-proof-main-thm}.

\subsection{Acknowledgements}
The authors wish to thank Song Sun and Jeff Viaclovsky for useful discussions, as well as the referees whose comments improved the clarity of the writing in certain places. The first author is supported by the grant Connect Talent ``COCOSYM'' of the r\'{e}gion des Pays de la Loire and the Centre Henri Lebesgue,  programme ANR-11-LABX-0020-0. The second author is supported by NSF grant DMS-1906466 and the third author is supported by grants ANR-17-CE40-0034 of the French National Research Agency ANR (Project CCEM) and ANR-AAPG2020 (Project PARAPLUI).

\section{Preliminaries}

\subsection{Shrinking Ricci solitons}\label{sec-srs}

Recall the definitions given at the beginning of Section \ref{overview}. An important class of examples of such manifolds for us is the following.

\begin{example}\label{example2}
We have a $1$-parameter family $\{\tilde{\omega}_{a}\}_{a>0}$ of (in-complete) shrinking gradient K\"ahler-Ricci soliton on $\mathbb{C}$. Indeed
for each $a>0$ , the K\"ahler form of the shrinking soliton is given by $\tilde{\omega}_{a}:=\frac{i}{2}\partial\bar{\partial}|z|^{2a}$, where $z$ is the holomorphic coordinate on $\mathbb{C}$. The soliton vector field of $\tilde{\omega}_{a}$ is given by $\frac{2}{a}\cdot\operatorname{Re}\left(z\partial_{z}\right)$. Of course when $a=1$, $\tilde{\omega}_{a}$ is complete and we recover the flat shrinking Gaussian soliton $\omega_{\mathbb{C}}$ on $\mathbb{C}$ with soliton vector field $2\cdot\operatorname{Re}\left(z\partial_{z}\right)$.
\end{example}

Any K\"ahler-Einstein manifold trivially defines a shrinking gradient K\"ahler-Ricci soliton (with soliton vector field $X=0$).
We may then take the Cartesian product with Example \ref{example2} to produce many more examples. These examples
provide the model at infinity for the reference metric that we will construct in Theorem \ref{mainthm}(i).

\begin{example}\label{example}
Let $(D,\,\omega_{D})$ be a K\"ahler-Einstein Fano manifold with K\"ahler form $\omega_{D}$. Then for each $a>0$,
the Cartesian product $\widehat{M}:=\mathbb{C}\times D$ endowed with the K\"ahler form
$\widehat{\omega}_{a}:=\tilde{\omega}_{a}+\omega_{D}$ is an example of an (incomplete) shrinking gradient K\"ahler-Ricci soliton.
Here, $\tilde{\omega}_{a}$ is as in Example \ref{example2}. Writing $r:=|z|^{a}$ with $z$ the complex coordinate on the $\mathbb{C}$-factor of $\widehat{M}$, the soliton vector field of this example is given by $\widehat{X}:=r\partial_{r}=\frac{2}{a}\cdot\operatorname{Re}\left(z\partial_{z}\right)$. When $a=1$, the soliton is complete and
up to isometry, we obtain a complete shrinking gradient K\"ahler-Ricci soliton on $\C\times D$ with bounded scalar curvature which is unique if $D$ is moreover toric \cite[Corollary C]{charlie}. We write $\widehat{g}_{a}$ and $\widehat{J}$ for the K\"ahler metric associated to $\widehat{\omega}_{a}$ and product complex structure on $\widehat{M}$ respectively.
\end{example}

The following lemma concerning $(\widehat{M},\,\widehat{\omega}_{a})$ will prove useful throughout.

\begin{lemma}\label{pluri}
With notation as in Example \ref{example}, fix $a>0$ (and hence the function $r$) and let $\widehat{K}\subset\widehat{M}$ be a compact subset such that $\widehat{M}\setminus\widehat{K}$ is connected.
If $u:\widehat{M}\setminus\widehat{K}\to\mathbb{R}$ is a smooth real-valued function defined on $\widehat{M}\setminus\widehat{K}$ that is pluriharmonic (meaning that $\partial\bar{\partial}u=0$) and invariant under the flow of $\widehat{J}\widehat{X}$, then $u=c_{0}\log(r)+c_{1}$ for some $c_{0},\,c_{1}\in\mathbb{R}$.
\end{lemma}

\begin{proof}
{Let $\widehat{X}^{1,\,0}:=\frac{1}{2}(\widehat{X}-i\widehat{J}\widehat{X})$. Then since $\widehat{X}$ is real holomorphic and $\mathcal{L}_{\widehat{J}\widehat{X}}u=0$, we see that
$$\bar{\partial}(\widehat{X}\cdot u)=\partial\bar{\partial}u\lrcorner(\widehat{X}^{1,\,0})=0,$$
i.e., $\widehat{X}\cdot u$ is holomorphic. As a real-valued holomorphic function, $\widehat{X}\cdot u$, which itself is equal to $r\partial_{r}u$, must be equal to a constant, $c_{0}$ say. Thus, being invariant under the flow of $\widehat{J}\widehat{X}$, we can write
$$u=c_{0}\log r+c_{1}(x),$$ where $x\in D$. Let $\Delta_{\mathbb{C}}$
and $\Delta_{D}$ denote the Riemannian Laplacians with respect to the flat metric $g_{\mathbb{C}}$ on $\mathbb{C}$ and the K\"ahler-Einstein metric $\omega_{D}$ on $D$, respectively. Then $u$, being pluriharmonic, implies that $\Delta_{\mathbb{C}}u+\Delta_{D}u=0$, and so}
\begin{equation*}
\begin{split}
0&=(\Delta_{D}+\Delta_{\mathbb{C}})(c_{0}\log(r)+c_{1}(x))\\
&=\Delta_{D}c_{1}(x)+\underbrace{\Delta_{\mathbb{C}}c_{1}(x)}_{=\,0}+c_{0}\underbrace{\Delta_{\mathbb{C}}\log(r)}_{=\,0}\\
&=\Delta_{D}c_{1}(x),\\
\end{split}
\end{equation*}
which infers that $c_{1}(x)=c_{1}$. This leaves us with $u=c_{0}\log(r)+c_{1}$, as desired.
\end{proof}

We conclude this section with a gluing lemma.

\begin{lemma}[Gluing lemma]\label{glue}
With notation as in Example \ref{example}, fix $a>0$ (and hence the function $r$), let $\widehat{K}\subset\widehat{M}$ be a compact subset, and let
$\phi\in C^{\infty}(\widehat{M}\setminus\widehat{K})$ be such that $\phi=O(\log(r))$, $|d\phi|_{\widehat{g}_{a}}=O(1)$, and
$|i\partial\bar{\partial}\phi|_{\widehat{g}_{a}}=O(r^{-a})$. Then for all $R>0$ with $\widehat{K}\subseteq\{r\leq R\}$,
there exists a cut-off function $\chi_{R}:M\to\mathbb{R}$ supported on $M\setminus\{r\leq R\}$ with $\chi_{R}(x)=1$
if $r(x)>2R$ such that
$$|i\partial\bar{\partial}(\chi_{R}\cdot\phi)|_{\widehat{g}_{a}}\leq\frac{C}{R^{\min\{1,\,a\}}}\left(\|(\log(r))^{-1}\cdot\phi\|_{C^{0}(\widehat{M}\setminus\widehat{K})}
+\|d\phi\|_{C^{0}(\widehat{M}\setminus\widehat{K},\,\widehat{g}_{a})}+\|r^{a}\cdot i\partial\bar{\partial}\phi\|_{C^{0}(\widehat{M}\setminus\widehat{K},\,\widehat{g}_{a})}\right)$$
for  some $C>0$ independent of $R$. In particular, $\chi_{R}\cdot\phi=\phi$ on $\{r(x)>2R\}$.
\end{lemma}

\begin{proof}
Let $\chi:\mathbb{R}\rightarrow\mathbb{R}$ be a smooth function satisfying
$\chi(x)=0$ for $x\leq 1$, $\chi(x)=1$ for $x\geq 4$, and $|\chi(x)|\leq 1$ for all $x$, and with it, define a function $\chi_{R}:M\to\mathbb{R}$ by
$$\chi_{R}(x)=\chi\left(\frac{r(x)^{2}}{R^{2}}\right)\qquad\textrm{for $R>0$ as in the statement of the lemma}.$$
Then $\chi_{R}$ is identically zero on $\{x\in\widehat{M}\,|\,r(x)<R\}$ and identically equal to one on the set\linebreak $\{x\in \widehat{M}\,|\,r(x)>2R\}$.
Define $\phi_{R}:=\chi_{R}.\phi$. Then the closed real $(1,\,1)$-form $i\partial\bar{\partial}(\chi_{R}.\phi)$ on $\widehat{M}$ is given by
\begin{equation*}
\begin{split}
i\partial\bar{\partial}(\chi_{R}.\phi)=\chi_{R}(r).i\partial\bar{\partial}\phi
&+\chi'\left(\frac{r^{2}}{R^2}\right).
i\frac{\partial r^{2}}{R}\wedge\frac{\bar{\partial}\phi}{R}+\frac{\phi}{R^{2}}.\chi'\left(\frac{r^{2}}{R^{2}}\right)
.i\partial\bar{\partial}r^{2}\\
&+\chi'\left(\frac{r^{2}}{R^{2}}\right).\frac{i\partial\phi}{R}\wedge\frac{\bar{\partial}r^{2}}{R}+\frac{\phi}{R^{2}}.\chi''\left
(\frac{r^{2}}{R^{2}}
\right).i\frac{\partial r^{2}}{R}\wedge \frac{\bar{\partial}r^{2}}{R}.
\end{split}
\end{equation*}
The assumptions on $\phi$ and its derivatives then imply for example that
\begin{equation*}
|\chi_{R}(x).i\partial\bar{\partial}\phi|_{\widehat{g}_{a}}\leq\sup_{r\,\in\,[R,\,\infty)}|i\partial
\bar{\partial}\phi|_{\widehat{g}_{a}}\leq\left(\sup_{r\,\in\,[R,\,\infty)}r^{-a}\right)
\left(\sup_{r\,\in\,[R,\,\infty)}r^{a}\cdot|i\partial
\bar{\partial}\phi|_{\widehat{g}_{a}}\right)\leq R^{-a}\|r^{a}\cdot i\partial\bar{\partial}\phi\|_{C^{0}(\widehat{M}\setminus\widehat{K},\,\widehat{g}_{a})}
\end{equation*}
and that
\begin{equation*}
\left|\chi'\left(\frac{r^{2}}{R^{2}}\right).
i\frac{\partial r^{2}}{R}\wedge\frac{\bar{\partial}\phi}{R}\right|_{\widehat{g}_{a}}\leq\frac{C}{R^{2}}\left(\sup_{r\,\in\,[R,\,2R]}r\right)\left(\sup_{r\,\in\,[R,\,2R]}\left|
i\partial r\wedge\bar{\partial}\phi\right|_{\widehat{g}_{a}}\right)\leq CR^{-1}\|d\phi\|_{C^{0}(\widehat{M}\setminus\widehat{K},\,\widehat{g}_{a})}.
\end{equation*}
The estimate of the lemma is now clear.
\end{proof}

\subsection{Basics of metric measure spaces}\label{metricmeasure}

We take the following from \cite{fut}.

A smooth metric measure space is a Riemannian manifold endowed with a weighted volume.
\begin{definition}
A \emph{smooth metric measure space} is a triple $(M,\,g,\,e^{-f}dV_{g})$, where $(M,\,g)$ is a complete Riemannian manifold with Riemannian metric $g$,
$dV_{g}$ is the volume form associated to $g$, and $f:M\to\mathbb{R}$ is a smooth real-valued function.
\end{definition}
\noindent A shrinking gradient Ricci soliton $(M,\,g,\,X)$ with $X=\nabla^g f$ naturally defines a smooth metric measure space $(M,\,g,\,e^{-f}dV_{g})$.
On such a space, we define the weighted Laplacian $\Delta_{f}$ by
$$\Delta_{f}u:=\Delta u-g(\nabla^g f,\,\nabla u)$$
on smooth real-valued functions $u\in C^{\infty}(M,\,\mathbb{R})$. There is a natural $L^{2}$-inner product $\langle\cdot\,,\,\cdot\rangle_{L^{2}_{f}}$ on the space $L^{2}_{f}$ of square-integrable smooth real-valued functions on $M$
with respect to the measure $e^{-f}dV_g$ defined by $$\langle u,\,v\rangle_{L_{f}^{2}}:=\int_{M}uv\,e^{-f}dV_{g},\qquad u,\,v\in L_{f}^{2}.$$
As one can easily verify, the operator $\Delta_{f}$ is self-adjoint with respect to $\langle\cdot\,,\,\cdot\rangle_{L_{f}^{2}}$.

\subsection{Polyhedrons and polyhedral cones}\label{pooly}

We take the following from \cite{cox}.

Let $E$ be a real vector space of dimension $n$ and let $E^{*}$ denote the dual. Write $\langle\cdot\,,\,\cdot\rangle$ for the evaluation $E^{*}\times E\to\mathbb{R}$. Furthermore, assume that we are given a \emph{lattice} $\Gamma \subset E$, that is, an additive subgroup $\Gamma \simeq \Z^n$. This gives rise to a dual lattice $\Gamma^* \subset E^*$. For any $\nu\in E$, $c\in\mathbb{R}$, let
$K(\nu,\,c)$ be the (closed) half space $\{x\in E\:|\:\langle\nu,\,x\rangle\geq c\}$ in $E$. Then we have:

\begin{definition}
A \emph{polyhedron} $P$ in $E$ is a finite intersection of half spaces,
i.e., $$P=\bigcap_{i=1}^{r}K(\nu_{i},\,c_{i})\qquad\textrm{for $\nu_{i}\in E^{*},\,c_{i}\in\mathbb{R}$}.$$
It is called a \emph{polyhedral cone} if all $c_{i}=0$, and moreover a \emph{rational polyhedral cone} if all $\nu_i \in \Gamma^*$ and $c_i = 0$. In addition, a polyhedron is called \emph{strongly convex} if it does not contain any affine subspace of $E$.
\end{definition}

The following definition will be useful.

\begin{definition}
A polyhedron $P \subset E^{*}$ is called \emph{Delzant} if its set of vertices is non-empty and each vertex $v \in P$ has the property that there are precisely $n$ edges $\{e_1, \dots e_n\}$ (one-dimensional faces) emanating from $v$ and there exists a basis $\{\varepsilon_1, \dots, \varepsilon_n\}$ of $\Gamma^*$ such that $\varepsilon_i$ lies along the ray $\R (e_i - v) $.
\end{definition}
\noindent Note that any such $P$ is necessarily strongly convex. We also have:
\begin{definition}
The \emph{dual} of a polyhedral cone $C$ is the set $C^{\vee}=\{x\in E^{*}\:|\:\langle x,\,C\rangle\geq0\}$.
\end{definition}

\subsection{Hamiltonian actions}\label{hamilton}

Recall what it means for an action to be Hamiltonian.

\begin{definition}
Let $(M,\,\omega)$ be a symplectic manifold and let $T$ be a real torus acting by symplectomorphisms on $(M,\,\omega)$.
Denote by $\mathfrak{t}$ the Lie algebra of $T$ and by $\mathfrak{t}^{*}$ its dual. Then we say that the action of $T$ is \emph{Hamiltonian}
if there exists a smooth map $\mu_{\omega}:M\to\mathfrak{t}^{*}$ such that for all $\zeta\in\mathfrak{t}$,
\begin{equation*}
-\omega\lrcorner\zeta=du_{\zeta},
\end{equation*}
where $u_{\zeta}(x)=\langle\mu_{\omega}(x),\,\zeta\rangle$ for all $\zeta\in\mathfrak{t}$ and $x\in M$
and $\langle\cdot\,,\cdot\rangle$ denotes the dual pairing between $\mathfrak{t}$ and $\mathfrak{t}^{*}$.
We call $\mu_{\omega}$ the \emph{moment map} of the $T$-action and we call $u_{\zeta}$ the \emph{Hamiltonian (potential)} of $\zeta$.
\end{definition}

\subsection{Toric geometry}\label{toric-geom}

In this section, we collect together some standard facts from toric geometry as well as recall those results from \cite{charlie} that we require.
We begin with the following definition.
\begin{definition}\label{toricmanifold}
A \emph{toric manifold} is an $n$-dimensional complex manifold $M$ endowed
with an effective holomorphic action of the algebraic torus $\Cstarn$ such that the following hold true.
\begin{itemize}
  \item The fixed point set of the $\Cstarn$-action is compact.
  \item There exists a point $p\in M$ with the property that the orbit $\Cstarn \cdot p \subset M$ forms a dense open subset of $M$.
\end{itemize}
\end{definition}
We will often denote the dense orbit simply by $\Cstarn \subset M$ in what follows.
The $\Cstarn$-action of course determines the action of the real torus $T^n \subset \Cstarn$.

\subsubsection{Divisors on toric varieties and fans}

Let $T^n \subset \Cstarn$ be the real torus with Lie algebra $\t$ and denote the dual pairing between $\t$ and the dual space $\mathfrak{t}^{*}$ by $\langle \cdot\,,\cdot\rangle$. There is a natural integer lattice $\Gamma \simeq \Z^n \subset \t$ comprising all $\lambda \in \t$ such that $\operatorname{exp}(\lambda) \in T^n$ is the identity. This then induces a dual lattice $\Gamma^* \subset \t^*$. We have the following combinatorial definition.

\begin{definition}
 A \emph{fan} $\Sigma$ in $\t$ is a finite set of rational polyhedral cones $\sigma$ satisfying:

	\begin{enumerate}
		\item For every $\sigma \in \Sigma$, each face of $\sigma$ also lies in $\Sigma$.
		\item For every pair $\sigma_1, \sigma_2 \in \Sigma$, $\sigma_1 \cap \sigma_2$ is a face of each.
	\end{enumerate}
\end{definition}

To each fan $\Sigma$ in $\t$, one can associate a toric variety $X_\Sigma$. Heuristically, $\Sigma$ contains all the data necessary to produce a partial equivariant compactification of $\Cstarn$, resulting in $X_\Sigma$. More concretely, one obtains $X_\Sigma$ from $\Sigma$ as follows. For each $n$-dimensional cone $\sigma \in \Sigma$, one constructs an affine toric variety $U_\sigma$ which we first explain. We have the dual cone $\sigma^{\vee}$ of $\sigma$. Denote by $S_\sigma$ the semigroup of those lattice points which lie in $\sigma^{\vee}$ under addition. Then one defines the semigroup ring, as a set, as all finite sums of the form
	\begin{equation*}
		\C[S_\sigma] = \left\{ \left.\sum \lambda_s s \, \right| \, s \in S_\sigma \right\},
	\end{equation*}
with the ring structure defined on monomials by $\lambda_{s_1}s_1\cdot \lambda_{s_2}s_2  = (\lambda_{s_1}\lambda_{s_2})(s_1+ s_2)$ and extended in the natural way. The affine variety $U_\sigma$ is then defined to be $\text{Spec}(\C[S_\sigma])$. This automatically comes endowed with a $\Cstarn$-action with a dense open orbit. This construction can also be applied to the lower dimensional cones $\tau \in \Sigma$. If $\sigma_1 \cap \sigma_2 = \tau$, then there is a natural way to map $U_\tau$ into $U_{\sigma_1}$ and $U_{\sigma_2}$ isomorphically. One constructs $X_\Sigma$ by declaring the collection of all $U_\sigma$ to be an open affine cover of $X_{\Sigma}$ with transition functions determined by $U_\tau$. This identification is also reversible.

\begin{prop}[{\cite[Corollary 3.1.8]{cox}}]\label{fann}
Let $M$ be a smooth toric manifold. Then there exists a fan $\Sigma$ such that $M \simeq X_\Sigma$.
\end{prop}

 \begin{prop}[{\cite[Theorem 3.2.6]{cox}, Orbit-Cone Correspondence}]\label{orbitcone} The $k$-dimensional cones $\sigma \in \Sigma$ are in a natural one-to-one correspondence with the $(n-k)$-dimensional orbits $O_\sigma$ of the $\Cstarn$-action on $X_\Sigma$.
 \end{prop}
\noindent In particular, each ray $\sigma \in \Sigma$ determines a unique torus-invariant divisor $D_\sigma$. As a consequence, a torus-invariant Weil divisor $D$ on $X_\Sigma$ naturally determines a polyhedron $P_D \subset \mathfrak{t}^{*}$. Indeed, we can decompose $D$ uniquely as $D = \sum_{i=1}^N a_i D_{\sigma_i}$, where $\{\sigma_i\}_{i}\subset\Sigma$ is the collection of rays. Then by assumption, there exists a unique minimal lattice element $\nu_i \in \sigma_i \cap \Gamma$. $P_{D}$ is then given by

 \begin{equation} \label{eqnB2}
 	P_D = \left\{ x \in \mathfrak{t}^{*} \: | \: \langle \nu_i, x \rangle \geq - a_i \right\} = \bigcap_{i = 1}^N K(\nu_i, -a_i).
 \end{equation}

\subsubsection{K\"ahler metrics on toric varieties}\label{finito}

For a given toric manifold $M$ endowed with a Riemannian metric $g$ invariant under the action of the real torus $T^n \subset \Cstarn$ and K\"ahler with respect to the underlying complex structure
of $M$, the K\"ahler form $\omega$ of $g$ is also invariant under the $T^n$-action. We call such a manifold a \emph{toric K\"ahler manifold}.
In what follows, we always work with a fixed complex structure on $M$.

Hamiltonian K\"ahler metrics have a useful characterisation due to Guillemin.

\begin{prop}[{\cite[Theorem 4.1]{Guil}}]\label{propB6}
	Let $\omega$ be any $T^n$-invariant K\"ahler form on $M$. Then the $T^{n}$-action is Hamiltonian with respect to $\omega$ if and only if the restriction of $\omega$ to the dense orbit $\Cstarn \subset M$ is exact, i.e., there exists a $T^{n}$-invariant potential $\phi$ such that
	\begin{equation*}
		\omega = 2i\p\bp \phi.
	\end{equation*}
\end{prop}

Fix once and for all a $\Z$-basis $(X_1,\ldots,X_n)$ of $\Gamma \subset \t$. This in particular induces a background coordinate system $\xi=(\xi^1, \dots, \xi^n)$ on $\t$.
Using the natural inner product on $\t$ to identify $\t \cong \t^*$, we can also identify $\t^* \cong \R^n$.
For clarity, we will denote the induced coordinates on $\t^*$ by $x=(x^1,\ldots, x^n)$. Let $(z_1, \dots, z_n)$ be the natural
coordinates on $\Cstarn$ as an open subset of $\C^n$. There is a natural diffeomorphism $\text{Log}:
\Cstarn \to \t \times T^n$ which provides a one-to-one correspondence between $T^n$-invariant smooth functions on
 $\Cstarn$ and smooth functions on $\t$. Explicitly,
\begin{equation}\label{diffeoo}
(z_1, \dots, z_n)\xmapsto{\operatorname{Log}}(\log(r_1), \dots, \log(r_n), \theta_1, \dots, \theta_n)=(\xi_{1},\ldots,\xi_{n},\,\theta_{1},\ldots,\theta_{n}),
\end{equation}
where $z_j = r_j e^{i \theta_j}$,\,$r_{j}>0$. Given a function $H(\xi)$ on $\t$, we can extend $H$ trivially to $\t \times T^n$ and pull back by Log to
obtain a $T^n$-invariant function on $\Cstarn$. Clearly, any $T^n$-invariant function on $\Cstarn$ can be written in this form.

Choose any branch of $\log$ and write $w = \log(z)$. Then clearly $w = \xi + i \theta$, where $\xi=(\xi^1,\ldots,\xi^n)$ are real coordinates on $\t$
(or, more precisely, there is a corresponding lift of $\theta$ to the universal cover with respect to which this equality holds),
and so if $\phi$ is $T^n$-invariant and $\omega = 2i \p \bp \phi$, then we have that
\begin{equation}\label{e:T5}
	\omega = 2i\frac{\p^2 \phi}{ \p w^i \p\bar{w}^j} dw_i \wedge d\bar{w}_j = \frac{\p^2 \phi}{ \p \xi^i \p\xi^j} d\xi^i \wedge d\theta^j.
\end{equation}
In this setting, the metric $g$ corresponding to $\omega$ is given on $\t \times T^n$ by
\begin{equation*}
	g = \phi_{ij}(\xi)d\xi^i d\xi^j + \phi_{ij}(\xi)d\theta^i d\theta^j,
\end{equation*}
and the moment map $\mu$ as a map $\mu: \t \times T^n \to \t^*$ is defined by the relation
		\begin{equation*}
		\langle \mu(\xi, \theta), b \rangle = \langle \nabla \phi(\xi), b \rangle
	\end{equation*}
for all $b \in \t$, where $\nabla \phi$ is the Euclidean gradient of $\phi$.
The $T^n$-invariance of $\phi$ implies that it depends only on $\xi$ when considered a function on $\t \times T^n$
via \eqref{diffeoo}. Since $\omega$ is K\"ahler, we see from \eqref{e:T5} that the Hessian of $\phi$ is positive-definite so that $\phi$ itself is strictly convex.
In particular, $\nabla \phi$ is a diffeomorphism onto its image.
Using the identifications mentioned above, we view $\nabla \phi$ as a map from $\t$ into an open subset of $\t^*$.

\subsubsection{K\"ahler-Ricci solitons on toric manifolds}

Next we define what we mean by a shrinking K\"ahler-Ricci soliton in the toric category.
\begin{definition}
A complex $n$-dimensional shrinking K\"ahler-Ricci soliton $(M,\,g,\,X)$ with complex structure $J$ and K\"ahler form $\omega$ is \emph{toric} if $(M,\,\omega)$ is a toric K\"ahler
manifold as in Definition \ref{toricmanifold} and $JX$ lies in the Lie algebra $\t$ of the underlying real torus $T^{n}$ that acts on $M$. In particular, the zero set of $X$ is compact.
\end{definition}

It follows from \cite{wyliee} that $\pi_{1}(M)=0$, hence the induced real $T^n$-action is automatically Hamiltonian with respect to $\omega$.
Working on the dense orbit $\Cstarn \subset M$, the condition that a vector field $JY$ lies in $\t$ is equivalent to saying that in the coordinate system $(\xi^1,\ldots,\xi^n,\,\theta_{1},\ldots,\theta_{n})$
from \eqref{diffeoo}, there is a constant $b_Y=(b_{Y}^{1},\ldots,b_{Y}^{n})\in \R^n$ such that
\begin{equation}\label{eqnY4}
	JY =  b_Y^i \frac{\p}{\p\theta^i}\qquad\textrm{or equivalently,}\qquad Y =   b_Y^i \frac{\p}{\p\xi^i}.
\end{equation}
From Proposition \ref{propB6}, we know that $\mathcal{L}_{X}\omega=2i\p\bp X(\phi)$. In addition,
the function $X(\phi)$ on $\Cstarn$ can be written as $\langle b_X, \nabla \phi \rangle = b_X^j \frac{\p\phi}{\p\xi^j}$,
where $b_{X}\in\mathbb{R}^{n}$ corresponds to the soliton vector field $X$ via \eqref{eqnY4}.
These observations allow us to write the shrinking soliton equation \eqref{krseqn} as a real Monge-Amp\`ere equation for $\phi$ on $\R^n$.

\begin{prop}[{\cite[Proposition 2.6]{charlie}}]
Let $(M,\,g,\,X)$ be a toric shrinking gradient K\"ahler-Ricci soliton with K\"ahler form $\omega$. Then
there exists a unique smooth convex real-valued function $\phi$ defined on the dense orbit $\Cstarn\subset M$ such that $\omega=2i\partial\bar{\partial}\phi$
and
\begin{equation} \label{realMA}
	\det(\phi_{ij})=e^{-2\phi+\langle b_{X},\,\nabla\phi\rangle}.
\end{equation}
\end{prop}

A priori, the function $\phi$ is defined only up to addition of a linear function.
However, \eqref{realMA} provides a normalisation for $\phi$ which in turn provides a normalisation for $\nabla\phi$, the moment map of the action.
The next lemma shows that this normalisation coincides with that for the moment map as defined in \cite[Definition 5.16]{cds}.

\begin{lemma}\label{normal}
Let $(M,\,g,\,X)$ be a toric complete shrinking gradient K\"ahler-Ricci soliton with complex structure $J$ and K\"ahler form $\omega$
with soliton vector field $X=\nabla^{g}f$ for a smooth real-valued function $f:M\to\mathbb{R}$.
Let $\phi$ be given by Proposition \ref{propB6} and normalised by \eqref{realMA}, let $JY\in\mathfrak{t}$, and let $u_{Y}=\langle\nabla\phi,\,b_{Y}\rangle$ be the Hamiltonian potential of $JY$
with $b_{Y}$ as in \eqref{eqnY4} so that $\nabla^{g}u_{Y}=Y$. Then $\mathcal{L}_{JX}u_{Y}=0$ and $\Delta_{\omega}u_{Y}+u_{Y}-\frac{1}{2}Y\cdot f=0$.
\end{lemma}
\noindent To see the equivalence with \cite[Definition 5.16]{cds},
simply replace $Y$ with $JY$ in this latter definition as here we assume that $JY\in\mathfrak{t}$, contrary to the convention in \cite[Definition 5.16]{cds} where it is assumed that
$Y\in\mathfrak{t}$.

Given the normalisation \eqref{realMA}, the next lemma identifies the image of the moment map $\mu=\nabla\phi$.

\begin{lemma}[{\cite[Lemmas 4.4 and 4.5]{charlie}}]\label{note}
Let $(M,\,g,\,X)$ be a complete toric shrinking gradient K\"ahler-Ricci soliton, let $\{D_i\}$ be the prime $\Cstarn$-invariant divisors in $M$, and let $\Sigma \subset \t$ be the fan determined by
Proposition \ref{fann}. Let $\sigma_i\in\Sigma$ be the ray corresponding to $D_i$ with minimal generator $\nu_i \in \Gamma$.
\begin{enumerate}
  \item There is a distinguished Weil divisor representing the anticanonical class $-K_{M}$ given by
  \begin{equation*}
		-K_M = \sum_i D_i
	\end{equation*}
	whose associated polyhedron (cf.~\eqref{eqnB2}) is given by
 	\begin{equation}\label{xmas}
		P_{-K_M} = \left\{ x \: | \: \langle \nu_i, x \rangle \geq -1 \right\}
	\end{equation}
which is strongly convex and has full dimension in $\mathfrak{t}^{*}$. In particular, the origin lies in the interior of $P_{-K_{M}}$.
  \item If $\mu$ is the moment map for the induced real $T^n$-action normalised by \eqref{realMA}, then the image of $\mu$ is precisely $P_{-K_{M}}$.
\end{enumerate}
	\end{lemma}

\subsubsection{The weighted volume functional}\label{weighted}

As a result of Lemma \ref{normal}, we can now define the weighted volume functional.

\begin{definition}[{Weighted volume functional, \cite[Definition 5.16]{cds}}]\label{weightedvol}
Let $(M,\,g,\,X)$ be a complex $n$-dimensional toric shrinking gradient K\"ahler-Ricci soliton with K\"ahler form $\omega=2i\partial\bar{\partial}\phi$
 on the dense orbit with $\phi$ strictly convex with moment map $\mu=\nabla\phi$ normalised by \eqref{realMA}. Assume that the fixed point set of the torus is compact and define the open convex cone $$\Lambda_{\omega}:=\{Y\in\mathfrak{t}\:|\:\textrm{$\langle\mu,\,Y\rangle$ is proper and bounded below}\}\subseteq\mathfrak{t}.$$ Then the \emph{weighted volume functional} $\mathcal{F}_{\omega}:\Lambda_{\omega}\to\mathbb{R}$ is defined by
	\begin{equation*}
		\mathcal{F}_{\omega}(v) = \int_M e^{-\langle \mu,\,v \rangle} \omega^n.
	\end{equation*}	
\end{definition}

As the fixed point set of the torus is compact by definition, $\mathcal{F}_{\omega}$ is well-defined by the non-compact version of the Duistermaat-Heckman formula \cite{wu}
(see also \cite[Theorem A.3]{cds}). It is moreover strictly convex on $\Lambda_{\omega}$ \cite[Lemma 5.17(i)]{cds}, hence has at most one critical point in this set.
This leads to two important lemmas concerning the weighted volume functional in the toric category, the independence of $\Lambda_{\omega}$ and $\mathcal{F}_{\omega}$ from the choice of shrinking soliton $\omega$.

\begin{lemma}[{\cite[Lemma 2.25]{ccd}}]\label{one}
$\Lambda_{\omega}$ is independent of the choice of toric shrinking K\"ahler-Ricci soliton $\omega$ in Definition \ref{weightedvol}.
\end{lemma}

\begin{lemma}[{\cite[Lemma 2.26]{ccd}}]\label{two}
$\mathcal{F}_{\omega}$ is independent of the choice of toric shrinking K\"ahler-Ricci soliton $\omega$ in Definition \ref{weightedvol}. Moreover,
after identifying $\Lambda_{\omega}$ with a subset of $\mathbb{R}^{n}$ via \eqref{eqnY4}, $\mathcal{F}_{\omega}$
is given by $\mathcal{F}_{\omega}(v)=(2\pi)^n \int_{P_{-K_M}} e^{-\langle v,\,x \rangle }\,dx$, where $x=(x^1,\ldots,x^n)$ denotes coordinates on $\mathfrak{t}^{*}$ dual to
the coordinates $(\xi^{1},\ldots,\xi^{n})$ on $\t$ introduced in Section \ref{finito}.
\end{lemma}
\noindent Thus, we henceforth drop the subscript $\omega$ from $\mathcal{F}_{\omega}$ and $\Lambda_{\omega}$ when working in the toric category. The
functional $\mathcal{F}:\Lambda\to\mathbb{R}$ is in addition proper in this category \cite[Proof of Proposition 3.1]{charlie}, hence attains a unique critical point in $\Lambda$.
This critical point characterises the soliton vector field of a complete toric shrinking gradient K\"ahler-Ricci soliton.
\begin{theorem}[{\cite[Theorem 4.6]{charlie}, \cite[Theorem 1.1]{caoo}}]\label{thmB13}
Let $(M,\,g,\,X)$ be a complete toric shrinking gradient K\"ahler-Ricci soliton with complex structure $J$.
Then $JX\in\Lambda$ and $JX$ is the unique critical point of $\mathcal{F}$ in $\Lambda$.
\end{theorem}

Having established in Lemmas \ref{one} and \ref{two} that in the toric category the weighted volume functional $F$ and its domain $\Lambda$ are determined solely by the polytope $P_{-K_{M}}$ which itself, by Lemma \ref{note}, depends only on the torus action on $M$ (i.e., is independent of the choice of shrinking soliton), and having an explicit expression for $\mathcal{F}$ given by Lemma \ref{two}, after using the torus action to identify $P_{-K_{M}}$ via \eqref{xmas}, we can determine explicitly the soliton vector field of a hypothetical toric shrinking gradient K\"ahler-Ricci soliton on $M$. Indeed, in light of Lemma \ref{two}, the unique minimiser $b_{X}\in\mathfrak{t}\simeq\mathbb{R}^{n}$ is characterised by the fact that
\begin{equation}\label{futaki}
0=d_{b_{X}}\mathcal{F}(v)=\int_{P_{-K_{M}}}\langle x,\,v\rangle\,e^{-\langle b_{X},\,x\rangle}dx\qquad\textrm{for all $v\in\mathbb{R}^{n}$.}
\end{equation}

In the setting of Theorem \ref{mainthm}, we can also determine $\Lambda$ explicitly. To this end, with notation as in Theorem \ref{mainthm}, we make the following observation concerning the Lie algebra $\t$ of $T$. By assumption, the restricted map $\pi|_{M}:M\to\widehat{M}:=\mathbb{C}\times D$ is a torus-equivariant biholomorphism on the complement of
$\pi^{-1}(D_{0})\subseteq M$ and $D_{0}\subseteq\widehat{M}$, hence $M\setminus\pi^{-1}(D_{0})$ is $\Cstarn$-equivariantly biholomorphic to $\Cstar\times D \supseteq \Cstarn$. It subsequently follows that $\t$ admits the splitting
\begin{equation*}
	\t \simeq \t_{\Cstar} \oplus \t_{D},
\end{equation*}
where $\t_{\Cstar}$ and $\t_{D}$ denote the Lie algebra of vector fields in $\t$ on $M$ whose image under $d\pi$ vanish along the $D$- and $\Cstar$-factors of $\widehat{M}\setminus D_{0}$ respectively.
With this in mind, we then have:
\begin{lemma}\label{useful}
In the setting and notation of Theorem \ref{mainthm} and with respect to the splitting $\t \simeq \t_{\Cstar} \oplus \t_{D}$, the domain $\Lambda$ of the weighted volume functional $\mathcal{F}$ is the half-space
	\begin{equation*}
		\Lambda = \left\{ \alpha \textnormal{Re}( z \partial_z) + Y \in  \t_{\Cstar} \oplus \t_{D}  \: | \: \alpha > 0\quad\textrm{and}\quad Y \in \t_D \right\}.
	\end{equation*}
\end{lemma}

\begin{proof}
Since $D$ is Fano, by Lemma \ref{note} we know that the anticanonical polyhedron $P_{-K_{\C \times D}}$ for $\C \times D$ is the ``simple product'', i.e.,
	\begin{equation}\label{product-poly}
		P_{-K_{\C \times D}} = \{ (x_1, \dots, x_n) \: | \: x_{1} \geq -1 \quad\textrm{and}\quad (x_2, \dots, x_{n}) \in P_D \}.
	\end{equation}
Moreover, it follows from the definition of $\pi$ that the normal fan $\Sigma_M$ of $P_{-K_M}$ is just a refinement of the normal fan $\Sigma_{\C \times D}$ of $P_{-K_{\C \times D}}$ (see \cite[Definition 3.3.17]{cox}). The set of defining equations for $P_{-K_M}$ is therefore obtained from those defining \eqref{product-poly} by including finitely many linear inequalities. This in particular implies that $P_{-K_M}$ and  $P_{-K_{\C \times D}}$ coincide outside a sufficiently large ball $B \subset \t^*$.

Let $Z \in \t$ and via \eqref{eqnY4}, identify $Z$ with a point $b_{Z}\in \R^n$. Then
the distinguished vector field $\operatorname{Re}(z \partial_z)\in\mathfrak{t}$ is identified with $(1,\,0,\ldots,0)\in\mathbb{R}^{n}$
via the aforesaid splitting of $\mathfrak{t}$ so that $Z =  \alpha \textnormal{Re}( z \partial_z) + Y \in  \t_{\Cstar} \oplus \t_{D}$ is identified with the point $b_Z = (\alpha, b_2, \dots, b_n) \in \R^{n}$ for some $b_{i}\in\mathbb{R},\,i=2,\ldots,n$. Since $P_{-K_M}$ is closed, it follows that the Hamiltonian potential $\mu_Z = \langle \mu,  Z \rangle = \langle x, b_{Z} \rangle$ of $Z$ is proper if and only if $|\langle x, b_{Z} \rangle | \to +\infty$ as $|x| \to +\infty$. Thus, since $D$ is compact so that $P_D$ is bounded, we see that the set of vector fields $Z\in\mathfrak{t}$ for which the Hamiltonian potential $\mu_Z$ is proper is precisely the complement of the inclusion $\t_D \hookrightarrow \t$. In addition, $\mu_Z$ is bounded from below if and only if $\langle x , b \rangle \to + \infty$ as $|x| \to +\infty$ in $P_{-K_M}$. As $|x| \to +\infty$ in $P_{-K_M}$ if and only if $x_1 \to + \infty$, the condition that $\mu_Z$ be bounded from below picks out the desired component of $\mathfrak{t}$ defining $\Lambda$.
\end{proof}

We illustrate an application of Lemma \ref{useful} with the following example.

\begin{example}\label{finaleg}
Let $D=\mathbb{P}^{1}$, let $\pi$ be the blowup map,
and let $([z_{1}:z_{2}],\,w)$ denote coordinates on $\mathbb{P}^{1}\times\mathbb{C}$.
Then there is an action of a real two-dimensional torus $T^{2}$ on
$\mathbb{P}^{1}\times\mathbb{C}$ given by
$$([z_{1}:z_{2}],\,w)\mapsto([e^{ib_{2}}z_{1}:z_{2}],\,e^{ib_{1}}w),$$
where $(b_{1},\,b_{2})\in\mathbb{R}^{2}$ which we identify with the Lie algebra $\mathfrak{t}$ of $T^{2}$.
Moreover, $M$ is the blowup of $\mathbb{P}^{1}\times\mathbb{C}$ at one point which without loss of generality we may assume to be $([0:1],\,0)$. The action of $T^{2}$ on $\mathbb{P}^{1}\times\mathbb{C}$ induces a $T^{2}$-action on $M$ in the obvious way. Lemma \ref{useful} then tells us that the domain $\Lambda$ of the weighted
volume functional $\mathcal{F}$ of $M$ is given by $$\{(b_{1},\,b_{2})\in\mathbb{R}^{2}\,|\,b_{1}>0\quad\textrm{and}\quad b_{2}\in\mathbb{R}\}\subseteq\mathfrak{t}.$$
Using the Duistermaat-Heckman theorem \cite[Theorem A.3]{cds}, one can write $\mathcal{F}$ as
\begin{equation*}
\mathcal{F}(b_{1},\,b_{2})=\frac{e^{b_{1}}}{(b_{1}-b_{2})b_{2}}+\frac{e^{b_{2}}}{(b_{2}-b_{1})b_{1}} -\frac{e^{b_{1}-b_{2}}}{b_{1}b_{2}}.
\end{equation*}
Observe that this is symmetric under the transformation $(b_{1},\,b_{2})\mapsto(b_1,\, b_{1}-b_{2})$, a transformation that preserves $\Lambda$.
The minimum of $\mathcal{F}$ in $\Lambda$ therefore lies along the line $0<b_{1}=2b_{2}$, in which case we have for $b_{2}>0$,
\begin{equation*}
\begin{split}
\mathcal{F}(b_{2})&= \frac{e^{2b_{2}} - {e^{b_{2}}}}{b_2^2}.
\end{split}
\end{equation*}
We then have that
\begin{equation*}
	\mathcal{F}'(b_{2})= b_2^{-3} e^{b_2} \left[ 2(b_2 -1) e^{b_2} - (b_2 - 2) \right].
\end{equation*}
This has a zero for $b_{2}>0$ precisely when
$$2\left(b_{2}-1\right)\mathrm{e}^{b_{2}}=b_{2}-2.$$
Numerical approximations give the unique positive root as $b_{2}\approx 0.64$, in agreement with \cite[Example 2.33]{ccd}.
\end{example}

\subsubsection{The Legendre transform}

Let $M$ be a toric manifold of complex dimension $n$ endowed with a complete K\"ahler form $\omega$ invariant under the induced real $T^{n}$-action and with respect to which
this action is Hamiltonian. Write $\omega=2i\partial\bar{\partial}\phi$ on the dense orbit for $\phi$ strictly convex as in Proposition \ref{propB6}. Then
$\nabla\phi(\mathbb{R}^{n})$ is a Delzant polytope $P$. Recall that we have coordinates $\xi$ on $\mathbb{R}^{n}\simeq\mathfrak{t}$, $x$ on $P$, and $\theta$ on $T$.
Given any smooth and strictly convex function $\psi$ on $\R^n$ such that $\nabla\psi(\R^n)=P$, there exists a unique smooth and strictly convex function $u_\psi(x)$ on $P$ defined by
	\begin{equation*}
		\psi(\xi) + u_\psi(\nabla \psi) = \langle\nabla\psi,\,\xi \rangle.
	\end{equation*}
This process is reversible; that is to say, $\psi$ is the unique function satisfying
		\begin{equation*}
		\psi(\nabla u_\psi) + u_\psi(x) = \langle x ,\,\nabla u_\psi \rangle,
	\end{equation*}
where $\nabla$ now denotes the Euclidean gradient with respect to $x$. The function $u_{\psi}$ is called the \emph{Legendre transform of $\psi$} and is sometimes denoted by $L(\psi)(x)$.
Clearly $L(L(\psi))(\xi)=\psi(\xi)$. The Legendre transform $u$ of $\phi$ is called the \emph{symplectic potential} of $\omega$, as the metric $g$ associated to $\omega$ is given by
$$g=u_{ij}(x)dx^{i}dx^{j}+u^{ij}(x)d\theta^{i}d\theta^{j}.$$

The following will prove useful.

\begin{lemma}[{cf.~\cite[Lemma 2.10]{charlie}}]\label{growth}
Let $\phi$ be any smooth and strictly convex function on an open convex domain $\Omega'\subset\mathbb{R}^{n}$
and let $u=L(\phi)$ be the Legendre transform of $\phi$ defined on $(\nabla\phi)(\Omega')=:\Omega$. If $0\in\Omega$, then there exists a constant $C > 0$ such that
$$\phi(\xi)\geq C^{-1}|\xi|-C.$$ In particular, $\phi$ is proper and bounded from below.
\end{lemma}

If $\phi \in C^\infty(\R^n)$ solves \eqref{realMA}, then the Legendre transform $u=L(\phi)$ satisfies
\begin{equation}\label{realMA2}
		2 \left( \langle \nabla u, x \rangle - u(x) \right) - \log\det(u_{ij}(x)) = \langle b_X, x \rangle\qquad\textrm{on $\Pol$}.
	\end{equation}
To study K\"ahler-Ricci solitons on $M$ via \eqref{realMA2} on $\Pol$, we need to understand when a strictly convex function on a Delzant polytope
defines a symplectic potential, i.e., is induced from a K\"ahler metric on $M$ via the Legendre transform.
To this end, consider a Delzant polytope $P$ obtained as the image of the moment map of a toric K\"ahler manifold.
Let $F_i$, $i = 1, \dots, d$ denote the $(n-1)$-dimensional facets of $P$ with inward-pointing normal vector $\nu_i \in \Gamma$, normalised so that $\nu_i$
is the minimal generator of $\sigma_i = \R_+ \cdot \nu_i$ in $\Gamma$, and let $\ell_i(x) = \langle \nu_i, x \rangle$ so that $\overline{P}$ is defined by the system of
inequalities $\ell_i(x) \geq - a_i$, $i = 1, \dots, N$, $a_i\in \R$. Then there exists a canonical metric $\omega_P$ on $M$ \cite[Proposition 2.7]{charlie}, the symplectic potential
$u_{P}$ of which is given explicitly by the formula \cite{BGL, Guil}
\begin{equation} \label{eqn2-10}
	u_P(x) = \frac{1}{2}\sum_{i=1}^d (\ell_i(x) + a_i) \log\left( \ell_i(x) + a_i \right).
\end{equation}
In particular, the Legendre transform $\phi_{P}$ of $u_{P}$ will define the K\"ahler potential on the dense orbit of a globally defined K\"ahler metric $\omega_{P}$ on $M$ \cite{BGL, Guil}.
More generally, it was observed by Abreu \cite{Ab1} that the Legendre transform $L(u)$ of
a strictly convex function $u$ on $P$ will define the K\"ahler potential on the dense orbit of a globally defined K\"ahler metric $\omega_{P}$ on $M$
if and only if $u$ has the same asymptotic behavior as $u_{P}$ of all orders as $x \to \partial P$. Indeed, we have the following slightly more general statement.

\begin{lemma}[{\cite{Ab1},\cite{ACGT2},\cite[Proposition 2.17]{charlie}}]\label{boundaryy}
A convex function $u$ on $P$ defines a K\"ahler metric $\omega_u$ on $M$ if and only if $u$ has the form
\begin{equation*}
u = u_{P} + v,
\end{equation*}
where $v \in C^\infty(\overline{P})$ extends past $\partial P$ to all orders.
\end{lemma}

In the case that $P=\Pol$, we read from Lemma \ref{note}(ii) that $a_{i}=-1$ for all $i$. Thus, in this case, the canonical metric on $\Pol$ has symplectic potential
		\begin{equation*}	
		u_{\Pol} = \frac{1}{2} \sum_i (\ell_i(x) + 1) \log(\ell_i(x) + 1).
	\end{equation*}

\subsubsection{The $\hat{F}$-functional}

We next define the $\hat{F}$-functional on toric K\"ahler manifolds.

\begin{definition}\label{fhat}
Let $(M,\,\omega)$ be a (possibly non-compact) toric K\"ahler manifold with complex structure $J$ endowed with a real holomorphic vector field $X$ such that
$JX\in\Lambda_{\omega}$. Write $T$ for the torus acting on $M$, identify the dense orbit with $\mathbb{R}^{n}$, let $\xi=(\xi_{1},\ldots,\xi_{n})$ denote coordinates on $\mathbb{R}^{n}$,
let $b_{X}$ be as in \eqref{eqnY4}, and write $\omega=2i\partial\bar{\partial}\phi_{0}$
on the dense orbit as in Proposition \ref{propB6}. Let $P:=(\nabla\phi_{0})(\mathbb{R}^{n})$ denote the image of the moment map associated to $\omega$ and let
$x=(x_{1},\ldots,x_{n})$ denote coordinates on $P$. Let $\varphi\in C^{\infty}(M)$ be a smooth function on $M$ invariant under
the action of $T$ such that $\omega+i\partial\bar{\partial}\varphi>0$ and assume that:
\begin{enumerate}[label=\textnormal{(\alph*)}]
  \item There exists a $C^{1}$-path of smooth functions $(\varphi_{s})_{s\in[0,\,1]}\subset C^{\infty}(M)$ invariant under the action of $T$ such that
$\varphi_{0}=0$, $\varphi_{1}=\varphi$, $\omega+i\partial\bar{\partial}\varphi_{s}>0$,
and $(\nabla\phi_{s})(\mathbb{R}^{n})=P$ for all $s\in[0,\,1]$, where $\phi_{s}:=\phi_{0}+\frac{\varphi_{s}}{2}$.
\item $\int_{0}^{1}\int_{\mathbb{R}^{n}}|\dot{\phi}_{s}|\,e^{-\langle b_{X},\,\nabla\phi_{s}\rangle}\det(\phi_{s,\,ij})\,d\xi\,ds<+\infty$.
\end{enumerate}
Then we define
\begin{equation*}
\begin{split}
\hat{F}(\varphi):=2\int_{P}(L(\phi_{1})-L(\phi_{0}))\,e^{-\langle b_{X},\,x\rangle}dx.
\end{split}
\end{equation*}
\end{definition}

The existence of the path $(\varphi_{s})_{s\in[0,\,1]}$ satisfying conditions (a) and (b) is required so that $\hat{F}(\varphi)$ is well-defined. To see this, first note:
\begin{lemma}\label{converge1}
Under the assumptions of Definition \ref{fhat}, let $u_{s}:=L(\phi_{s})$, $\omega_{s}=\omega+i\partial\bar{\partial}\varphi_{s}$, and
write $f_{s}:=f+\frac{X}{2}\cdot \varphi_{s}$ for the Hamiltonian potential of $JX$ with respect to $\omega_{s}$, where $f$ is the Hamiltonian potential of $JX$ with respect to
$\omega$. Then the following are equivalent.
\begin{enumerate}
\item $\int_{0}^{1}\int_{\mathbb{R}^{n}}|\dot{\phi}_{s}|\,e^{-\langle b_{X},\,\nabla\phi_{s}\rangle}\det(\phi_{s,\,ij})\,d\xi\,ds<+\infty$.
\item $\int_{0}^{1}\int_{P}|\dot{u}_{s}|\,e^{-\langle b_{X},\,x\rangle}\,dx\,ds<+\infty$.
\item $\int_{0}^{1}\int_{M}|\dot{\varphi}_{s}|\,e^{-f_{s}}\omega^{n}_{s}\,ds<+\infty$.
\end{enumerate}
In particular when this is the case, $|\hat{F}(\varphi)|<+\infty$.
\end{lemma}

\begin{proof}
The equivalence of (i) and (iii)  is clear. The equivalence of (i) and (ii) follows from \cite[Lemma 3.7]{charlie}.
Finally, for the last statement, for every $x\in P$, we have that
$$|u_{1}-u_{0}|(x)\leq\int_{0}^{1}|\dot{u}_{s}|(x)\,ds.$$
Then using Fubini's theorem and noting Lemma \ref{converge1}, we estimate that
$$|\hat{F}(\varphi)|\leq2\int_{P}|u_{1}-u_{0}|\,e^{-\langle b_{X},\,x\rangle}dx\leq2\int_{P}\left(\int_{0}^{1}|\dot{u}_{s}|\,ds\right) e^{-\langle b_{X},\,x\rangle}dx
=2\int_{0}^{1}\int_{P}|\dot{u}_{s}|\,e^{-\langle b_{X},\,x\rangle}dx\,ds<+\infty,$$
as desired.
\end{proof}

Under an additional assumption on the path $(\varphi_{s})_{s\in[0,\,1]}$, we recover the well-known expression for the
$\hat{F}$-functional given in \cite[p.702]{ctz}.

\begin{lemma}\label{converge2}
If one (and hence all) of the conditions of Lemma \ref{converge1} hold true and if in addition it holds true that
$\int_{0}^{1}\int_{M}|\dot{\varphi}_{s}|\,e^{-f}\omega^{n}\,ds<+\infty$, then
\begin{equation}\label{rel-F-hat-kahler}
\hat{F}(\varphi)=\int_0^1\int_M\dot{\varphi}_{s}\left(e^{-f}\omega^n-e^{-f_{s}}\omega_{s}^n\right)\wedge ds
-\int_M\varphi\,e^{-f}\omega^n.
\end{equation}
\end{lemma}

\begin{proof}
The extra condition implies in particular that $\int_{M}|\varphi|\,e^{-f}\omega^{n}<+\infty$ since by assumption and Fubini's theorem, $\int_M|\varphi|\,e^{-f}\omega^{n}\leq \int_{0}^{1}\int_{M}|\dot{\varphi}_{s}|\,e^{-f}\omega^{n}\,ds <+\infty$ so that the right-hand side of
\eqref{rel-F-hat-kahler} is at least finite. To show that it is equivalent to the expression for $\hat{F}$ given by Definition \ref{fhat},
using the change of coordinates induced by $\nabla\phi_{s}:\mathbb{R}^{n}\to P$
and the fact that $\dot{\phi}_{s}(\nabla\phi_{s})=-\dot{u}_{s}(x)$ (cf.~\cite[Lemma 3.7]{charlie}), we compute that
\begin{equation*}
\begin{split}
\hat{F}(\varphi)&=2\int_{P}(u_{1}(x)-u_{0}(x))\,e^{-\langle b_{X},\,x\rangle}dx\\
&=2\int_0^1\int_{P}\dot{u}_{s}(x)\,e^{-\langle b_{X},\,x\rangle}dx\wedge ds\\
&=-2\int_0^1\int_{P}\dot{\phi_{s}}(\nabla\phi_{s})\,e^{-\langle b_{X},\,x\rangle}dx\wedge ds\\
&=-2\int_0^1\int_{\mathbb{R}^{n}}\dot{\phi_{s}}\,e^{-\langle b_{X},\,\nabla\phi_{s}\rangle}\det(\phi_{s,\,ij})
d\xi\wedge ds\\
&=-\int_0^1\int_M\dot{\varphi}_{s}\,e^{-f_{s}}\omega_{s}^n\wedge ds\\
&=\int_0^1\int_M\dot{\varphi}_{s}\left(e^{-f}\omega^n-e^{-f_{s}}\omega_{s}^n\right)\wedge ds-\int_M\varphi\,e^{-f}\omega^n,\\
\end{split}
\end{equation*}
resulting in the desired expression. Here we have used Fubini's theorem in the last equality.
\end{proof}

\subsubsection{Integrability and independence of the path}

In light of conditions (a) and (b) of Definition \ref{fhat} required to define the $\hat{F}$-functional, it remains to identify sufficient
conditions for the moment polytope to remain unchanged under a path of K\"ahler metrics and for each
summand in the integral of $\hat{F}$ to be finite. This will be important for achieving an a priori $C^{0}$-bound along our continuity path.

To this end, suppose that $(M,\,\omega)$ is a toric K\"ahler manifold,
i.e., $(M,\,\omega)$ is K\"ahler with K\"ahler form $\omega$ with respect to a complex structure $J$, endowed with
the holomorphic action of a complex torus of the same complex dimension as $(M, \, J)$ whose underlying real torus $T$ induces a Hamiltonian action,
and let $JX\in\mathfrak{t}$. Via \eqref{eqnY4}, we can identify $X$ with an element $b_{X}\in\mathbb{R}^{n}\simeq\mathfrak{t}$.
Using Proposition \ref{propB6}, we can also write $\omega=2i\partial\bar{\partial}\phi_{0}$ on the dense orbit for some strictly convex function
$\phi_{0}:\mathbb{R}^{n}\to\mathbb{R}$. Assume that:
\begin{itemize}
\item $JX\in\Lambda_{\omega}$ so that the Hamiltonian potential $f$ of $JX$ is proper and bounded from below.
  \item There exists a smooth bounded real-valued function $F$ on $M$ so that the Ricci form $\rho_{\omega}$ of $\omega$ satisfies
$\rho_\omega + \frac{1}{2}\mathcal{L}_X\omega - \omega = i \p\bp F$.
\end{itemize}
The equation in the second bullet point reads as
\begin{equation*}
\left(F+\log\det(\phi_{0,\,ij}) - \langle \nabla \phi_0, b_X \rangle + 2\phi_0\right)_{ij}=0\qquad\textrm{on $\mathfrak{t}\simeq\mathbb{R}^{n}$}
\end{equation*}
so that
$$F=-\log\det(\phi_{0,\,ij}) + \langle \nabla \phi_0, b_X \rangle - 2\phi_0+a(\xi)\qquad\textrm{on $\mathbb{R}^{n}$}$$
for some {affine} function $a(\xi)$ defined on $\mathbb{R}^{n}$. By considering
$2\phi_0+a+\langle\nabla a,\,b_X \rangle$, we can therefore assume that
\begin{equation}\label{norm}
F=-\log\det(\phi_{0,\,ij})+\langle\nabla\phi_0,\,b_X \rangle-2\phi_0\qquad\textrm{on $\mathbb{R}^{n}$}.
\end{equation}
The main observation of this section is the following lemma.

\begin{lemma}\label{expression}
Under the above assumptions, let $\varphi \in C^\infty(M)$ be a torus-invariant smooth real-valued function on $M$ such that
$\omega_\varphi:=\omega +  i \p \bp \varphi > 0$ and {$\sup_M|X\cdot\varphi| < \infty$}.
Define $\phi:=\phi_{0}+\frac{1}{2}\varphi$ so that $\omega + i \p\bp \varphi = 2i\p\bp \phi$ on the dense orbit. Then:
\begin{enumerate}[label=\textnormal{(\roman*)}]
\item The image of the moment map $\mu_{\omega_{\varphi}}:M \to \t^*$ with respect to $\omega_\varphi$ defined by the Euclidean gradient $\nabla\phi: \R^n \to \R^n$ is equal to $P_{-K_{M}}$.
In particular, $0\in\operatorname{int}\left(\mu_{\omega_{\varphi}}(M)\right)$.
\item $\int_{P}|L(\phi_0)|\,e^{-\langle b_{X},\,x\rangle}dx<+\infty$.
\end{enumerate}
\end{lemma}

\begin{proof}
\begin{enumerate}
\item To prove (i), we begin by noting that since {$\sup_M|X\cdot\varphi| < \infty$}, the Hamiltonian potential $f_\varphi = f + \frac{X}{2}\cdot \varphi$ of $X$ with respect to $\omega_{\varphi}$
is proper and bounded from below. In particular, the image $(\nabla\phi)(\mathbb{R}^{n})$ of the moment map $\mu_{\omega_{\varphi}}:M \to \t^*$ is equal
to a Delzant polyhedron $P$ \cite[Lemma 2.13]{charlie} that a priori depends on $\varphi$.
Let $u(x):=L(\phi)$ be the Legendre transform of $\phi$. Then
the domain of $u$ is precisely $P$. We need to show that $P=P_{-K_{M}}$.
To this end, let $F$ be as in \eqref{norm}. Then a computation shows that
\begin{equation}\label{help}
\begin{split}
	-\log\det \phi_{ij} &+ \langle \nabla \phi, b_X \rangle - 2 \phi  = F + \log\left( \frac{\omega_\varphi^n}{\omega^n}\right) +\frac{X}{2}\cdot \varphi - \varphi.
\end{split}
\end{equation}
Set $A(x):=\langle b_X,\,x \rangle$ and define
	\begin{equation*}
		\rho_u(x):=2 \left(\langle \nabla u,\,x \rangle - u(x) \right)-\log\det(u_{ij}).
	\end{equation*}
Then via the change of coordinates $x=\nabla\phi(\xi)$, we can rewrite \eqref{help} in terms of $u$ as
\begin{equation}\label{help2}
\begin{split}
	A(x)-\rho_{u}(x)=\left(F + \log\left( \frac{\omega_\varphi^n}{\omega^n}\right) +\frac{X}{2}\cdot \varphi - \varphi\right)(\nabla u(x))\qquad\textrm{on $P$}.
\end{split}
\end{equation}
Observe that the right-hand side of \eqref{help2} admits a continuous extension up to the boundary $\partial P$ of $P$.
Denoting the right-hand side of \eqref{help} by $h$ which is a function $h:M\to\mathbb{R}$, this extension is simply given by
$h\circ\mu_{\omega_{\varphi}}^{-1}$, where $\mu_{\omega_{\varphi}}:M\to\overline{P}$, as the moment map,
has fibers precisely the orbits of the torus action.

We now proceed as in \cite[Lemma 4.5]{charlie} using an argument originally due to Donaldson \cite{donaldsontoric}.
Suppose that $P$ is defined by the linear inequalities $\ell_i(x) \geq -a_i$,
where $\ell_i(x) = \langle \nu_i , x \rangle$. Since the right-hand side of \eqref{help2} as well as $A(x)$ has a continuous extension up to $\partial P$, we see that
the same holds true for $\rho_{u}(x)$. Moreover, as $u$ is the symplectic potential of the K\"ahler form $\omega_{\varphi}$ on $P$, we read from
Lemma \ref{boundaryy} that there exists  a function $v\in C^{\infty}(\overline{P}')$ with $u=u_{P}+v$,
where $u_{P}$ is given as in \eqref{eqn2-10}, i.e.,
	\begin{equation}\label{wtf2}
		u_{P}(x) = \frac{1}{2}\sum (\ell_i(x) + a_i) \log(\ell_i(x) + a_i).
	\end{equation}

Fix any facet $F'$ of $P$. Without loss of generality, we may assume that $F'$ is defined by $\ell_1(x) = -a_1$. Up to a change of basis in $\t^*$,
we may also assume by the Delzant condition that $\ell_1(x) = x_1$. Fix a point $p$ in the interior of $F'$. Then from \eqref{wtf2} we see that in a neighbourhood of $p$, $u$ can be written as
	\begin{equation*}
		u(x) = u_{P}(x) + v(x) = \frac{1}{2}(x_1 + a_1)\log(x_1 + a_1) + v_1
	\end{equation*}
for some smooth function $v_1$ which extends smoothly across $F'$. From this expression,
it follows that in a small half ball $B$ in the interior of $P$ containing $p$,
$\rho_u$ takes the form
		\begin{equation*}
\begin{split}
		\rho_u(x) &= x_1 \log(x_1 + a_1) -(x_1 + a_1)\log(x_1 + a_1) + \log(x_1 + a_1) + v_2\\
&=(1 - a_1)\log(x_1 + a_1) + v_2
\end{split}
	\end{equation*}
for another smooth function $v_2$ that extends smoothly across $F'$ in $B$. Thus, already knowing that $\rho_{u}$ has a continuous extension across $\partial P$, we deduce that
$1-a_1 = 0$, i.e, $a_{1}=1$. Continuing in this manner, we see that $a_{i}=1$ for all $i$, leading us to the conclusion that $P=P_{-K_{M}}$.

\item Let $u_0 = L(\phi_0)$. Then as $u_0$ is a convex function on $P_{-K_{M}}$ whose gradient has image equal to all of $\R^n$ by the invertibility of the Legendre transform,
it is proper and bounded from below by Lemma \ref{growth}. Let $A$ denote the lower bound, let $\rho_u$ be as in part (i), and
write $\rho_0 = \rho_{u_0}$. Then $F$ bounded implies the existence of a constant $C>0$ such that $|\rho_0 - \langle b_{X},\,x\rangle| < C $ on $P_{-K_{M}}$.
Indeed, this is clear from \eqref{norm}. Since $\int_{P_{-K_{M}}}u_0\,e^{-\rho_0}dx < \infty$ by \cite[Lemma 4.7]{charlie}, it follows that $\int_{P_{-K_{M}}}u_0 \,e^{-\langle b_{X},\,x\rangle}dx<+\infty$.
Finiteness of the integral $\int_{P_{-K_{M}}}e^{-\langle b_{X},\,x\rangle}dx$ together with the fact that $u_{0}$ is bounded from below now yields the desired result.
\end{enumerate}
\end{proof}

\section{Proof of Theorem \ref{mainthm}(ii): Construction of a background metric}\label{sec-construction-back-metric}

Given the setup and notation of Theorem \ref{mainthm} and with $X$ determined by Theorem \ref{mainthm}(i), we henceforth assume that the flow-lines of $JX$ close. In
this section, we construct a background metric on $M$ with the properties as stated in Theorem \ref{mainthm}(ii)
with a construction reminiscent of that of \cite[Section 4.2]{acyl}.
To this end, recall for $a>0$ the (incomplete) shrinking gradient K\"ahler-Ricci soliton $(\widehat{M}:=\mathbb{C}\times D,\,\widehat{\omega}_{a}:=\tilde{\omega}_{a}+\omega_{D},\,\frac{2}{a}\cdot\operatorname{Re}(z\partial_{z}))$
of Example \ref{example} with complex structure $\widehat{J}$ endowed with the product holomorphic action of the real $n$-torus $\widehat{T}$,
{with $z$ denoting the holomorphic coordinate on the $\mathbb{C}$-factor of $\widehat{M}$, and $r:=|z|^{a}$.}

\begin{prop}\label{background}
There exists:
\begin{enumerate}[label=\textnormal{(\alph{*})}, ref=(\alph{*})]
\item a complete K\"ahler metric $\omega$ on $M$ invariant under the action of $T$, and
\item a biholomorphism $\nu:M\setminus K\to\widehat{M}\setminus\widehat{K}$, where
$K\subset M$, $\widehat{K}\subset\widehat{M}$, are compact,
\end{enumerate}
and $\lambda>0$ such that
\begin{enumerate}
\item $d\nu(X)=\frac{2}{\lambda}\cdot\operatorname{Re}(z\partial_{z})$,
\item $\omega=\nu^{*}(\tilde{\omega}_{\lambda}+\omega_{D})$, and
\item the Ricci form $\rho_{\omega}$ of $\omega$ satisfies
\begin{equation}\label{ciao}
\rho_{\omega}+\frac{1}{2}\mathcal{L}_{X}\omega-\omega=i\partial\bar{\partial}F_{1}
\end{equation}
for $F_{1}\in C^{\infty}(M)$ compactly supported with $\mathcal{L}_{JX}F_{1}=0$.
\end{enumerate}
\end{prop}

Theorem \ref{mainthm}(ii) immediately follows from this proposition. Indeed, this is easily seen by
writing $\omega_{C}:=\tilde{\omega}_{\lambda}$ (cf.~Example \ref{example2}) and $\widehat{\omega}:=\widehat{\omega}_{\lambda}=\tilde{\omega}_{\lambda}+\omega_{D}$
(cf.~Example \ref{example}). With $\lambda$ fixed in subsequent sections, this is the notation that we adopt to be consistent with that of Theorem \ref{mainthm}.
Property (iii) of this proposition will be used in the next section.

\begin{proof}[Proof of Proposition \ref{background}]
Recall that $\pi:\overline{M}\to \mathbb{P}^{1}\times D$ is a torus-equivariant holomorphic map that restricts to
a holomorphic map $\pi:M\to\widehat{M}:=\mathbb{C}\times D$ by removing the fibre $D_{\infty}$
from $\overline{M}$ and $\mathbb{P}^{1}\times D$ respectively, and that
$z$ denotes the holomorphic coordinate on the $\mathbb{C}$-factor of $\widehat{M}$. We define the map $\nu:M\setminus\pi^{-1}(D_{0})\to\widehat{M}\setminus D_{0}$ of (b)
as the $\mathbb{C}^{*}$-equivariant map one obtains by identifying a $\mathbb{P}^{1}$-fibre in each manifold and for each point in this $\mathbb{P}^{1}$, flowing along the vector field $X^{1,\,0}:=\frac{1}{2}(X-i(JX))$ on $M$ and the holomorphic vector field $z\partial_{z}$ on $\widehat{M}$. Since the flow-lines of $JX$ close by assumption, this map is well-defined.

From the construction, it is clear that $d\nu(X^{1,\,0})=\frac{2}{\lambda}\cdot z\partial_{z}$ for some $\lambda>0$. This defines $\lambda$ and verifies condition (i) of the proposition.
The map $\nu$ then extends to a holomorphic map $\overline{\nu}:\overline{M}\setminus\pi^{-1}(D_{0})\to\widehat{M}\setminus D_{0}$.
On $\mathbb{C}\times D$ we consider the product metric $\widehat{\omega}_{\lambda}$. We write
$w:=\frac{1}{z}$ and $r:=|z|^{\lambda}$. Identifying $M\setminus\pi^{-1}(D_{0})$ and
$\widehat{M}\setminus D_{0}$ via $\nu$, we view these as functions, and $\widehat{\omega}_{\lambda}$ as a K\"ahler form, both on the former.
In this way, $w$ defines a holomorphic coordinate on $\overline{M}\setminus\pi^{-1}(D_{0})$ with the divisor $D$ at infinity defined by $\{w=0\}$.

Using $\nu$, we construct the background metric $\omega$ of (a) in the following way. As $\overline{M}$ is Fano by assumption, there exists a hermitian metric $h$ on $-K_{\overline{M}}$ with strictly positive curvature form $\Theta_{h}$. Moreover, since the normal bundle $N_{D}$ of $D$ in $\overline{M}$ is trivial so that $K_{D}=K_{\overline{M}}|_{D}$ by adjunction, the $\p\bp$-lemma guarantees the existence of a function $u\in C^{\infty}(D)$ such that $i\Theta_{h}|_{D}+i\partial\bar{\partial}u=\omega_{D}$.
Extend $u$ to be constant along the $w$-direction and multiply this extension by a cut-off function depending only on $w$
to further extend $u$ to the whole of $\overline{M}$. We still denote this extension by $u$.
If the support of this cut-off function is contained in a sufficiently small tubular neighbourhood of $D$, then the restriction of $i\Theta_{h}+i\partial\bar{\partial}u$ to any
{of the $D$-fibres of the fibration} will be positive-definite. All negative components of $i\Theta_{h}+i\partial\bar{\partial}u$
on the total space $\overline{M}$ can be compensated for by adding
a sufficiently positive ``bump $2$-form'' of the form $\chi(|w|)dw\wedge d\bar{w}$,
where $\chi$ is a bump function supported in an annulus containing the cut-off region;
such a $2$-form is automatically closed and $(1,\,1)$ on $\overline{M}$, and exact on $M$. This creates a K\"ahler form $\tau_{1}$ on $M$.
One can verify in a sufficiently small neighbourhood of $D$ that
\begin{equation}\label{lovely123}
\tau_{1}-\omega_{D}=O(|w|)\left(dw\wedge d\bar{w}+\sum_{j}dw\wedge d\bar{v}_{\bar{\jmath}}+\sum_{i,\,j}dv_{i}\wedge d\bar{v}_{\bar{\jmath}}+\sum_{i}dv_{i}\wedge d\bar{w}\right)\qquad\textrm{as $w\to0$}
\end{equation}
for $\{v_{1},\ldots,v_{n-1}\}$ local holomorphic coordinates on $D$.

Next, let $\psi:\mathbb{R}\to\mathbb{R}$ be a smooth function satisfying $$\psi'(x),\,\psi''(x)\geq 0\quad\textrm{for all $x\in\mathbb{R}$},$$ and
\begin{equation*}
\psi(x) = \left\{ \begin{array}{ll}
\operatorname{const.} & \textrm{if $x<1$,}\\
x & \textrm{if $x>2$,}
\end{array} \right.
\end{equation*}
and consider the composition $k:=\psi\circ r^{2}$, a real-valued smooth function on $M$.
One computes that
\begin{equation*}
\frac{i}{2}\p\bar{\p}k=\psi''(r^{2})\,\frac{i}{2}\p r^{2}\wedge\bar{\p}r^{2}+\psi'(r^{2})\,\frac{i}{2}\p\bar{\p}r^{2}\geq0,
\end{equation*}
a positive semi-definite form equal to $\frac{i}{2}\p\bar{\p}r^{2}$ on the region of $M$ where $r^{2}>2$. Define the K\"ahler form
$$\tau_{2}:=\tau_{1}+\frac{i}{2}\partial\bar{\partial}k$$ and in the holomorphic coordinates $(z,\,v)$ on $\widehat{M}$, consider the product metric $\widehat{\omega}_{\lambda}$ given by
$$\widehat{\omega}_{\lambda}:=\tilde{\omega}_{\lambda}+\omega_{D}
=i\partial\bar{\partial}\left(\frac{|z|^{2\lambda}}{2}\right)+\omega_{D}=\frac{\lambda^{2}dz\wedge d\bar{z}}{2|z|^{2-2\lambda}}+\omega_{D}.$$
Then from \eqref{lovely123}, it is clear that the difference is given by
\begin{equation*}
\tau_{2}-\widehat{\omega}_{\lambda}=O(|w|)\left(dw\wedge d\bar{w}+\sum_{j}dw\wedge d\bar{v}_{\bar{\jmath}}+\sum_{i,\,j}dv_{i}\wedge d\bar{v}_{\bar{\jmath}}+\sum_{i}dv_{i}\wedge d\bar{w}\right)\qquad\textrm{as $w\to0$},
\end{equation*}
so that in particular,
\begin{equation}\label{metric}
|\tau_{2}-\widehat{\omega}_{\lambda}|_{\widehat{\omega}_{\lambda}}=O(r^{-\frac{1}{\lambda}}).
\end{equation}
We now work with the hermitian metric $H$ on $-K_{\widehat{M}}$ induced by $\widehat{\omega}_{\lambda}$. Via the map $\nu$, this pulls back to the hermitian metric
 $$H=\frac{\lambda^{2}\det((g_{D})_{i\bar{\jmath}})}{2|z|^{2-2\lambda}}$$ on $-K_{M}|_{M\setminus\pi^{-1}(D_{0})}$ with respect to the local trivialisation $\partial_{z}\wedge\partial_{v_{1}}\wedge\ldots\wedge\partial_{v_{n-1}}$. The corresponding curvature form is then given by $$-i\partial\bar{\partial}\log H=\omega_{D}.$$ Hence, as a difference of two curvature forms, there exists a smooth real-valued function $\phi$ defined on $M\setminus\pi^{-1}(D_{0})$ such that
$$(i\Theta_{h}+i\partial\bar{\partial}u)-\omega_{D}=i\partial\bar{\partial}\phi.$$
In particular, outside a large compact subset of $M$, we have that
\begin{equation}\label{ribs}
\tau_{2}-\widehat{\omega}_{\lambda}=i\partial\bar{\partial}\phi.
\end{equation}

We claim that $\phi$ is in fact smooth on $\overline{M}\setminus\pi^{-1}(D_{0})$. To see this, note that as
$$i\partial\bar{\partial}\phi=-i\partial\bar{\partial}\log\left(\frac{e^{-u}|\partial_{z}\wedge\partial_{v_{1}}\
\wedge\ldots\wedge\partial_{v_{n-1}}|^{2}_{h}}{|\partial_{z}\wedge\partial_{v_{1}}\wedge\ldots\wedge\partial_{v_{n-1}}|^{2}_{H}}\right)$$
and
\begin{equation*}
\begin{split}
&\log\left(\frac{e^{-u}|\partial_{z}\wedge\partial_{v_{1}}\wedge\ldots\wedge\partial_{v_{n-1}}|^{2}_{h}}{|\partial_{z}\wedge\partial_{v_{1}}\wedge\ldots\wedge\partial_{v_{n-1}}|^{2}_{H}}\right)
=\log\left(\frac{e^{-u}|w|^{4}|\partial_{w}\wedge\partial_{v_{1}}\wedge\ldots\wedge\partial_{v_{n-1}}|^{2}_{h}}{|\partial_{z}\wedge\partial_{v_{1}}\wedge\ldots\wedge\partial_{v_{n-1}}|^{2}_{H}}\right)\\
&=\log\left(\frac{2e^{-u}|\partial_{w}\wedge\partial_{v_{1}}\wedge\ldots\wedge\partial_{v_{n-1}}|^{2}_{h}|z|^{2-2\lambda}}{\lambda^{2}\det((g_{D})_{i\bar{\jmath}})}\right)+2\log(|w|^{2})\\
&=\underbrace{\log\left(\frac{2e^{-u}|\partial_{w}\wedge\partial_{v_{1}}\wedge\ldots\wedge\partial_{v_{n-1}}|^{2}_{h}}{|\partial_{v_{1}}\wedge\ldots\wedge\partial_{v_{n-1}}|^{2}_{\omega_{D}}}\right)}_{\textrm{extends smoothly over $\{w=0\}$}}
-\underbrace{\left(1-\lambda\right)\log|w|^{2}+2\log(|w|^{2})-\log(\lambda^{2})}_{\textrm{pluriharmonic}},\\
\end{split}
\end{equation*}
$\phi$ may be taken to be
\begin{equation}\label{form}
\phi=-\log\left(\frac{2e^{-u}|\partial_{w}\wedge\partial_{v_{1}}\wedge\ldots\wedge\partial_{v_{n-1}}|^{2}_{h}}{|\partial_{v_{1}}\wedge\ldots\wedge\partial_{v_{n-1}}|^{2}_{\omega_{D}}}\right),
\end{equation}
which, albeit defined in terms of local coordinates, is clearly globally defined on $\overline{M}\setminus\pi^{-1}(D_{0})$.
Thus $\phi=O(1)$ and from \eqref{metric} and \eqref{ribs}, we see that $|i\partial\bar{\partial}\phi|_{\widehat{\omega}_{\lambda}}=O(r^{-\frac{1}{\lambda}})$.
Finally, after a computation, the expression for $\phi$ given in \eqref{form} gives us that $|d\phi|_{\widehat{\omega}_{\lambda}}=O(1)$.
Now, $\widehat{\omega}_{\lambda}$ and $\tau_{2}$ are equivalent outside some large compact subset $\tilde{K}$ of $M$ by \eqref{metric}, and on the
complement of $\tilde{K}$ in $M$, Lemma \ref{glue} implies that for all $R>0$ sufficiently large, $\phi$ admits an extension $\phi_{R}$ to $M$
supported on $M\setminus\tilde{K}$ such that $|i\partial\bar{\partial}\phi_{R}|_{\widehat{\omega}_{\lambda}}\leq CR^{-\min\{\lambda^{-1},\,1\}}$.
Thus, at the expense of increasing $C$ if necessary, we can infer that $|i\partial\bar{\partial}\phi_{R}|_{\tau_{2}}\leq CR^{-\min\{\lambda^{-1},\,1\}}$ globally on $M$.
We fix $R>0$ large enough so that $|i\partial\bar{\partial}\phi_{R}|_{\tau_{2}}<1$ and define a K\"ahler form on $M$ by
\begin{equation*}
\tilde{\omega}:=\tau_{2}-i\p\bar{\p}\phi_{R}.
\end{equation*}
By what we have just said, $\tilde{\omega}$ is positive-definite everywhere on $M$ and equal to $\widehat{\omega}_{\lambda}$ outside a large compact subset, hence is complete.
By averaging over the action of $T$, we may assume that $\mathcal{L}_{JX}\tilde{\omega}=0$ without changing the behaviour at infinity.
We further modify $\tilde{\omega}$ to construct $\omega$ satisfying conditions (a) and (ii) of the proposition.

To this end, we know that since $M$ does not split off any $S^{1}$-factors, $\pi_{1}(M)=0$ by toricity \cite{cox}. In particular, $H^{1}(M,\,\mathbb{R})=0$ so that the action of $T$ on $M$ is Hamiltonian with respect to $\tilde{\omega}$. Consequently, there exists a smooth real-valued function $\tilde{f}$
such that $\frac{1}{2}\mathcal{L}_{X}\tilde{\omega}=i\partial\bar{\partial}\tilde{f}$. By averaging, $\tilde{f}$ can be taken to be invariant under the action of $T$ on $M$. It is also clear that as $\tilde{\omega}=i\Theta_{h}+i\partial\bar{\partial}u_{1}$ for some $u_{1}\in C^{\infty}(M)$
with $i\Theta_{h}$ the curvature form of a hermitian metric on $-K_{M}$, we can write $\rho_{\tilde{\omega}}-\tilde{\omega}=i\partial\bar{\partial}u_{2}$ for another function $u_{2}\in C^{\infty}(M)$, $\rho_{\tilde{\omega}}$ here denoting the Ricci form of $\tilde{\omega}$. Thus, there exists a function $\tilde{G}\in C^{\infty}(M)$ such that
\begin{equation}\label{sexier}
\begin{split}
\rho_{\tilde{\omega}}-\tilde{\omega}+\frac{1}{2}\mathcal{L}_{X}\tilde{\omega}=i\partial\bar{\partial}\tilde{G}.
\end{split}
\end{equation}
After averaging, we may assume that $\tilde{G}$ is invariant under the action of $T$. In particular, henceforth identifying $M$ and $\widehat{M}$ on the complement of compact subsets
containing $D_{0}$ and $\pi^{-1}(D_{0})$ respectively, we can write $\tilde{G}:=\tilde{G}(r,\,x)$, where $r=|z|^{\lambda}$ is as above and $x\in D\subset\widehat{M}$.
As $\tilde{\omega}$ defines a shrinking gradient K\"ahler-Ricci soliton on $M\setminus K$ for some $K\subset M$ compact, we see that $\tilde{G}$ is pluriharmonic on $M\setminus K$. It therefore follows from Lemma \ref{pluri} that
$$\tilde{G}=c_{0}\log(r)$$ for some constant $c_{0}\in\mathbb{R}$. Arguing as above, Lemma \ref{glue} guarantees the existence of
an extension $\varphi$ of $c_{0}\log(r)+\frac{c_{0}}{2}$ to $M$ such that
$\omega:=\tilde{\omega}+i\partial\bar{\partial}\varphi$ defines a K\"ahler metric on $M$. As $\varphi$ is pluriharmonic at infinity, it is
clear that $\omega=\tilde{\omega}=\nu^{*}\widehat{\omega}_{\lambda}$ outside a large enough compact subset of $M$. Averaging over the action of $T$,
we obtain our metric $\omega$ of (a) satisfying condition (ii).

Next, as in \eqref{sexier}, we see that there exists a function $G\in C^{\infty}(M)$ invariant under the action of $T$ such that
\begin{equation}\label{sexiest}
\begin{split}
\rho_{\omega}-\omega+\frac{1}{2}\mathcal{L}_{X}\omega=i\partial\bar{\partial}G.
\end{split}
\end{equation}
Subtracting \eqref{sexier} from \eqref{sexiest} yields the relation
\begin{equation*}
\begin{split}
i\partial\bar{\partial}G&=i\partial\bar{\partial}\tilde{G}+\rho_{\omega}-\rho_{\tilde{\omega}}-i\partial\bar{\partial}\varphi+i\partial\bar{\partial}\left(\frac{X}{2}\cdot\varphi\right)\\
&=i\partial\bar{\partial}\left(\tilde{G}-\log\left(\frac{\omega^{n}}{\tilde{\omega}^{n}}\right)-\varphi+\frac{X}{2}\cdot\varphi\right)\\
\end{split}
\end{equation*}
between $G$ and $\tilde{G}$. Set $$F_{1}:=\tilde{G}-\log\left(\frac{\omega^{n}}{\tilde{\omega}^{n}}\right)-\varphi+\frac{X}{2}\cdot\varphi.$$
Then $i\partial\bar{\partial}F_{1}=i\partial\bar{\partial}G$ so that \eqref{ciao} holds true, and outside a large compact subset of $M$ we have that
$$F_{1}=\tilde{G}-\log\left(\frac{\omega^{n}}{\tilde{\omega}^{n}}\right)-\varphi+\frac{X}{2}\cdot\varphi=c_{0}\log(r)-\varphi(r)+\frac{r}{2}\cdot\varphi'(r)=0,$$
demonstrating that $F_{1}\in C^{\infty}(M)$ and is compactly supported. As $\mathcal{L}_{JX}\tilde{G}=0$, condition (iii) and correspondingly, the
proposition now follow.
\end{proof}

\section{Proof of Theorem \ref{mainthm}(iii) and (iv): Set-up of the complex Monge-Amp\`ere equation}\label{sec-set-up-CMA}

Returning now to the setup and notation of Theorem \ref{mainthm}, we next provide a proof of Theorem \ref{mainthm}(iii) by setting up a complex Monge-Amp\`ere equation
that any shrinking K\"ahler-Ricci soliton on $M$ differing from our background metric by $i\partial\bar{\partial}$ of a potential must satisfy,
followed by a proof of Theorem \ref{mainthm}(iv) where a normalised Hamiltonian potential of $JX$ with respect to $\omega$ is given.
Throughout this section we write $r:=|z|^{\lambda}$, where $z$ is the holomorphic coordinate on the $\mathbb{C}$-factor of $\widehat{M}$ and $\lambda>0$ is as in
Theorem \ref{mainthm}(iii) so that $d\nu(X)=r\partial_{r}$.
Our starting point is:

\begin{prop}\label{equationsetup}
Let $\omega$ be the K\"ahler metric in Proposition \ref{background} and let $J$ denote the complex structure on $M$.
Then there exists $\tilde{\varphi}\in C^{\infty}(M)$ with $\mathcal{L}_{JX}\tilde{\varphi}=0$ and $\omega_{\tilde{\varphi}}:=\omega+i\partial\bar{\partial}\tilde{\varphi}>0$ such that
\begin{equation}\label{sexysoliton}
\rho_{\omega_{\tilde{\varphi}}}+\frac{1}{2}\mathcal{L}_{X}\omega_{\tilde{\varphi}}=\omega_{\tilde{\varphi}}
\end{equation}
if and only if for all $a\in\mathbb{R}$, there exists $\varphi\in C^{\infty}(M)$ with $\mathcal{L}_{JX}\varphi=0$ and $\omega+i\partial\bar{\partial}\varphi>0$
and $F_{2}\in C^{\infty}(M)$ compactly supported with $\mathcal{L}_{JX}F_{2}=0$ satisfying
\begin{equation}\label{credit}
\rho_{\omega}+\frac{1}{2}\mathcal{L}_{X}\omega-\omega=i\partial\bar{\partial}F_{2}
\end{equation}
such that
\begin{equation}\label{e:soliton}
\log\left(\frac{(\omega+i\partial\bar{\partial}\varphi)^{n}}{\omega^{n}}\right)-\frac{X}{2}\cdot\varphi+\varphi=F_{2}+a.
\end{equation}
Here, $\rho_{\omega}$ and $\rho_{\omega_{\tilde{\varphi}}}$ denote the Ricci form of $\omega$ and $\omega_{\tilde{\varphi}}$ respectively.
\end{prop}

\begin{proof}
If $\varphi$ satisfies \eqref{e:soliton}, then by taking $i\partial\bar{\partial}$ of this equation, we see that
$\varphi$ satisfies \eqref{sexysoliton} by virtue of \eqref{ciao}. Conversely, assume that \eqref{sexysoliton} holds. Then we compute:
\begin{equation*}
\begin{split}
0&=\rho_{\omega_{\tilde{\varphi}}}-\omega_{\tilde{\varphi}}+\frac{1}{2}\mathcal{L}_{X}
\omega_{\tilde{\varphi}}\\
&=\rho_{\omega_{\tilde{\varphi}}}-\rho_{\omega}+\rho_{\omega}-\omega_{\tilde{\varphi}}+\frac{1}{2}\mathcal{L}_{X}\omega_{\tilde{\varphi}}\\
&=-i\partial\bar{\p}\log\left(\frac{(\omega+i\partial\bar{\partial}\tilde{\varphi})^{n}}{\omega^{n}}\right)-
i\partial\bar{\partial}\tilde{\varphi}+i\partial\bar{\partial}\left(\frac{X}{2}\cdot\tilde{\varphi}\right)
+\rho_{\omega}-\omega+\frac{1}{2}\mathcal{L}_{X}\omega\\
\end{split}
\end{equation*}
so that
\begin{equation}\label{equation}
\begin{split}
i\partial\bar{\p}\left(\tilde{\varphi}+\log\left(\frac{(\omega+i\partial\bar{\partial}\tilde{\varphi})^{n}}{\omega^{n}}\right)-\frac{X}{2}\cdot\tilde{\varphi}\right)
&=\rho_{\omega}-\omega+\frac{1}{2}\mathcal{L}_{X}\omega.\\
\end{split}
\end{equation}
Now, as we have seen in \eqref{ciao},
$$\rho_{\omega}-\omega+\frac{1}{2}\mathcal{L}_{X}\omega=i\partial\bar{\partial}F_{1}$$
for some $JX$-invariant compactly supported $F_{1}\in C^{\infty}(M)$. Plugging this into \eqref{equation}, we have that for every $a\in\mathbb{R}$,
$$i\partial\bar{\p}\left(\tilde{\varphi}+\log\frac{(\omega+i\partial\bar{\partial}\tilde{\varphi})^{n}}{\omega^{n}}-\frac{X}{2}\cdot\tilde{\varphi}-F_{1}-a\right)=0.$$
$JX$-invariance of the sum in parentheses next implies from Lemma \ref{pluri} that
$$\tilde{\varphi}+\log\left(\frac{(\omega+i\partial\bar{\partial}\tilde{\varphi})^{n}}{\omega^{n}}\right)-\frac{X}{2}\cdot\tilde{\varphi}
=F_{1}+a+H$$
for $H$ a pluriharmonic function equal to $c_{0}\log(r)+c_{1}$ outside a compact subset of $M$ for some $c_{0},\,c_{1}\in\mathbb{R}$. Thus,
\begin{equation*}
\begin{split}
\left(\tilde{\varphi}-H-\frac{c_{0}}{2}\right)&+\log\left(\frac{(\omega+i\partial\bar{\partial}(\tilde{\varphi}-H-\frac{c_{0}}{2}))^{n}}{\omega^{n}}\right)
-\frac{X}{2}\cdot\left(\tilde{\varphi}-H-\frac{c_{0}}{2}\right)\\
&=\left(\tilde{\varphi}+\log\left(\frac{(\omega+i\partial\bar{\partial}\tilde{\varphi})^{n}}{\omega^{n}}\right)-\frac{X}{2}\cdot\tilde{\varphi}\right)-H+\frac{X}{2}\cdot H-\frac{c_{0}}{2}\\
&=(F_{1}+a+H)-H+\frac{X}{2}\cdot H-\frac{c_{0}}{2}\\
&=F_{1}+a+\frac{X}{2}\cdot H-\frac{c_{0}}{2}.
\end{split}
\end{equation*}

Notice that after identifying $X$ with $r\partial_{r}$ via $\nu$, we have that $\frac{X}{2}\cdot H-\frac{c_{0}}{2}=\frac{1}{2}r\partial_{r}(c_{0}\log(r)+c_{1})-\frac{c_{0}}{2}=0$ outside a compact set. Set $\varphi:=\tilde{\varphi}-H-\frac{c_{0}}{2}$ and $F_{2}:=F_{1}+\frac{X}{2}\cdot H-\frac{c_{0}}{2}$. Then $F_{2}\in C^{\infty}(M)$, is compactly supported,
both $\varphi$ and $F_{2}$ are $JX$-invariant, $i\partial\bar{\partial}F_{2}=i\partial\bar{\partial}F_{1}$, and
$$\varphi+\log\left(\frac{(\omega+i\partial\bar{\partial}\varphi)^{n}}{\omega^{n}}\right)-\frac{X}{2}\cdot\varphi=F_{2}+a,$$
as required.
\end{proof}

Theorem \ref{mainthm}(iii) is a consequence of the next lemma.

\begin{lemma}\label{warm}
Let $\lambda$, $\omega$, and $\nu:(M\setminus K,\,\omega)\to(\widehat{M}\setminus\widehat{K},\,\widehat{\omega})$, $K\subset M,\,\widehat{K}\subset\widehat{M}$ compact,
be as in Proposition \ref{background}. Moreover, let $F_{2}\in C^{\infty}(M)$ be as in Proposition \ref{equationsetup} satisfying \eqref{credit}
and recall that $z$ denotes the holomorphic coordinate on the $\mathbb{C}$-factor of $\widehat{M}$. Set $r:=|z|^{\lambda}$.
Then there exists a unique torus-invariant smooth real-valued function $f:M\to\mathbb{R}$ such that
$-\omega\lrcorner JX=df$, $f=\nu^{*}\left(\frac{r^{2}}{2}-1\right)$ on $M\setminus K$, and
\begin{equation}\label{normal2}
\Delta_{\omega}f+f-\frac{X}{2}\cdot f=0\qquad\textrm{outside a compact subset of $M$.}
\end{equation}
In particular, $f\to+\infty$ as $r\to+\infty$, hence is proper.
\end{lemma}

\begin{proof}
Since $M$ does not split off any $S^{1}$-factors and is toric, we know that $\pi_{1}(M)=0$ \cite{cox}. Hence there exists a smooth real-valued function $f\in C^{\infty}(M)$, defined up to a constant, with $-\omega\lrcorner JX=df$. Any such choice of $f$ is invariant under the action of $T$ by virtue of the fact that $\omega\lrcorner JX$ is invariant under this action and $T$ has fixed points so that every element of $\mathfrak{t}$ has at least one zero. Next notice that $-\widehat{\omega}\lrcorner\widehat{J}r\partial_{r}=d\left(\frac{r^{2}}{2}\right)$, where recall $\widehat{J}$ is the complex structure on $\widehat{M}$.
As $\omega=\nu^{*}\widehat{\omega}$ on $M\setminus K$, it is therefore clear that $d\left(f-\frac{r^{2}}{2}\right)=0$ on $M\setminus K$ so that $f$
differs from $\frac{r^{2}}{2}$ by a constant on this set, i.e., $f=\frac{r^{2}}{2}+\operatorname{const.}$ on $M\setminus K$. Normalise $f$ so that this constant is equal to
$-1$. Then $f=\nu^{*}\left(\frac{r^{2}}{2}-1\right)$ on $M\setminus K$. What remains to show is that with this normalisation, \eqref{normal2} holds true.

To this end, using the $JX$-invariance of $F_{2}$ and $f$,
contract \eqref{credit} with $X^{1,\,0}:=\frac{1}{2}(X-iJX)$ and use the Bochner formula to derive that
$$i\bar{\partial}\left(\Delta_{\omega}f-\frac{X}{2}\cdot f+f+\frac{X}{2}\cdot F_{2}\right)=0.$$
As a real-valued holomorphic function, we must have that $\Delta_{\omega}f-\frac{X}{2}\cdot f+f+\frac{X}{2}\cdot F_{2}$ is
constant on $M$. But since $X\cdot F_{2}=0$ outside a compact subset of $M$, by the properties of $f$ and $\omega$ we have that outside a compact subset of $M$,
\begin{equation*}
\begin{split}
\Delta_{\omega}f-\frac{X}{2}\cdot f+f+\frac{X}{2}\cdot F_{2}&=\Delta_{\widehat{\omega}}\left(\frac{r^{2}}{2}-1\right)
-\frac{r}{2}\frac{\partial}{\partial r}\left(\frac{r^{2}}{2}-1\right)+\left(\frac{r^{2}}{2}-1\right)=0.
\end{split}
\end{equation*}
Thus, this constant is zero and we are done.
\end{proof}

Let $c_{0}\in\mathbb{R}$ be such that $e^{c_{0}}\int_{M}e^{F_{2}-f}\omega^{n}=\int_{M}e^{-f}\omega^{n}$ and define $F:=F_{2}+c_{0}$. Then:
\begin{itemize}
  \item $F\in C^{\infty}(M)$ and $F$ is torus-invariant,
  \item $F$ is equal to $c_{0}$ outside a compact subset of $M$, and
  \item $\int_{M}e^{F-f}\omega^{n}=\int_{M}e^{-f}\omega^{n}$.
\end{itemize}
Moreover, from \eqref{credit} we have that
\begin{equation*}
\rho_{\omega}-\frac{1}{2}\mathcal{L}_{X}\omega+\omega=i\partial\bar{\partial}F.
\end{equation*}
By Proposition \ref{equationsetup}, any shrinking K\"ahler-Ricci soliton of the form $\omega+i\partial\bar{\partial}\varphi>0$ on $M$
will solve the complex Monge-Amp\`ere equation
\begin{equation*}
\left\{
\begin{array}{rl}
(\omega+i\partial\bar{\partial}\varphi)^{n}=e^{F+\frac{X}{2}\cdot\varphi-\varphi}\omega^{n}&\quad\textrm{for $\varphi\in C^{\infty}(M)$ and $\varphi$ torus-invariant},\\
\int_{M}e^{F-f}\omega^{n}=\int_{M}e^{-f}\omega^{n}. &
\end{array} \right.
\end{equation*}
This is precisely the statement of Theorem \ref{mainthm}(iv). A strategy to solve this equation is given by considering the Aubin continuity path:
\begin{equation}\label{ast-t}
\left\{
\begin{array}{rl}
(\omega+i\partial\bar{\partial}\varphi_{t})^{n}=e^{F+\frac{X}{2}\cdot\varphi_{t}-t\varphi_{t}}\omega^{n},&\quad\varphi\in C^{\infty}(M),\quad\mathcal{L}_{JX}\varphi=0,\quad\omega+i\partial\bar{\partial}\varphi>0,\quad t\in[0,\,1],\\
\int_{M}e^{F-f}\omega^{n}=\int_{M}e^{-f}\omega^{n}. &
\end{array} \right.\tag{$\ast_{t}$}
\end{equation}
The equation corresponding to $t=0$ is given by
\begin{equation}\label{ast-0}
\left\{
\begin{array}{rl}
(\omega+i\partial\bar{\partial}\psi)^{n}=e^{F+\frac{X}{2}\cdot\psi}\omega^{n},&\qquad\psi\in C^{\infty}(M),\qquad\mathcal{L}_{JX}\psi=0,\qquad\omega+i\partial\bar{\partial}\psi>0,\\
\int_{M}e^{F-f}\omega^{n}=\int_{M}e^{-f}\omega^{n}. &
\end{array} \right.\tag{$\ast_{0}$}
\end{equation}
This equation we will solve by the continuity method, the particular
path of which will be introduced in Section \ref{continuitie}. This will yield the final part of Theorem \ref{mainthm}.
Beforehand however, we prove some analytic results regarding the metric $\omega$ and those metrics that are asymptotic to it, beginning with a Poincar\'e inequality.

\newpage
\section{Poincar\'e inequality}\label{sec-poin-inequ}

In this section, we prove a Poincar\'e inequality for the K\"ahler form $\omega$ of Proposition \ref{background}
using the fact that it holds true on the model shrinking gradient K\"ahler-Ricci soliton
$(\widehat{M}:=\mathbb{C}\times\mathbb{P}^{1},\,\widehat{\omega}:=\tilde{\omega}_{\lambda}+\omega_{D},\,r\partial_{r})$ \cite{milman}, where $r=|z|^{\lambda}$.
This will be used in Proposition \ref{prop-a-priori-ene-est} to establish an a priori weighted $L^{2}$-estimate along the
continuity path that we consider in deriving a solution to \eqref{ast-0}. Recall the Hamiltonian potential
$f$ of $JX$ satisfying \eqref{normal2}.

We work with the Lebesgue and Sobolev spaces $L^{p}(e^{-f}\omega^{n})$ and $W^{1,\,p}(e^{-f}\omega^{n})$ on $M$ respectively, defined in the obvious way for $p>1$, and we denote
$$\fint_{M}u\,e^{-f}\omega^{n}:=\frac{1}{\int_{M}e^{-f}\omega^{n}}\int_{M}u\,e^{-f}\omega^{n}\qquad\textrm{for all $u\in L^{p}(e^{-f}\omega^{n})$.}$$
By H\"older's inequality and the finiteness of $\int_{M}e^{-f}\omega^{n}$, the integral $\fint_{M}u\,e^{-f}\omega^{n}$ is finite.

\begin{prop}[Poincar\'e inequality]\label{poincare}
For all $p>1$, there exists a constant $C(p)>0$ such that
$$\left\|u-\fint_{M}u\,e^{-f}\omega^{n}\right\|_{L^{p}(e^{-f}\omega^{n})}\leq C(p)\|\nabla^{g}u\|_{L^{p}(e^{-f}\omega^{n})}\qquad\textrm{for all $u\in W^{1,\,p}(e^{-f}\omega^{n})\cap C^{1}(M)$}.$$
Here, $g$ is the K\"ahler metric associated to $\omega$.
\end{prop}

\begin{proof}
For sake of a contradiction, suppose that the assertion is not true. Then
there exists a sequence of functions $(u_{k})_{k\,\geq\,1}\subset W^{1,\,p}(e^{-f}\omega^{n})$ with the following properties:
$$\left\{
\begin{array}{rl}
&\|u_{k}\|_{L^{p}(e^{-f}\omega^{n})}=1,\qquad\int_{M}u_{k}\,e^{-f}\omega^{n}=0, \\
&\|\nabla^{g}u_{k}\|_{L^{p}(e^{-f}\omega^{n})}\leq\frac{1}{k}.  \\
\end{array} \right.$$
{Indeed, since $\int_{M}e^{-f}\omega^{n}<\infty$, an application of H\"older's inequality demonstrates that we can normalise the sequence $(u_{k})_{k\,\geq\,1}$ so that the weighted integral of each function in the sequence is zero}. By the Rellich-Kondrachov theorem, there exists a subsequence which we also denote by $(u_{k})_{k\,\geq\,1}$ converging to
some $u_{\infty}\in L_{\operatorname{\operatorname{loc}}}^{p}(M)$ as $k\rightarrow+\infty$. On the other hand, for every compactly supported one-form $\alpha$ on $M$, we have that
$$\int_{M}u_{\infty}\cdot\delta^{g}\alpha\,\omega^{n}=\lim_{k\,\to\,+\infty}\int_{M}u_{k}\cdot\delta^{g}\alpha\,\omega^{n}=-\lim_{k\,\to\,+\infty}\int_{M}g(du_{k},\,\alpha)\,\omega^{n}=0,$$
where $\delta_{g}$ is the co-differential of $d$ with respect to $g$. Thus, $u_{\infty}\in W^{1,\,p}_{\operatorname{\operatorname{loc}}}(M)$ and $du_{\infty}=0$ almost everywhere.
In particular, $u_{\infty}$ is constant.

For $R>0$, let $D_{R}:=f^{-1}((-\infty,\,R])$, a compact subset of $M$ by properness of $f$ (cf.~Lemma \ref{warm}). Then
the fact that $\int_{M}u_{k}\,e^{-f}\omega^{n}=0$ implies that for every $R>0$,
$$\int_{D_{R}}u_{k}\,e^{-f}\omega^{n}=-\int_{M\setminus D_{R}}u_{k}\,e^{-f}\omega^{n}.$$
It then follows from H\"older's inequality that
\begin{equation*}
\begin{split}
\left|\int_{D_{R}}u_{k}\,e^{-f}\omega^{n}\right|&\leq\int_{M\setminus D_{R}}|u_{k}|\,e^{-f}\omega^{n}\\
&\leq\left(\int_{M\setminus D_{R}}|u_{k}|^{p}\,e^{-f}\omega^{n}\right)^{\frac{1}{p}}\left(\int_{M\setminus D_{R}}e^{-f}\omega^{n}\right)^{1-\frac{1}{p}}\\
&\leq\|u_{k}\|_{L^{p}(e^{-f}\omega^{n})}\left(\int_{M\setminus D_{R}}e^{-f}\omega^{n}\right)^{1-\frac{1}{p}}\\
&=\left(\int_{M\setminus D_{R}}e^{-f}\omega^{n}\right)^{1-\frac{1}{p}}.
\end{split}
\end{equation*}
Furthermore, $L^{p}_{\operatorname{\operatorname{loc}}}(M)$-convergence implies that
$$\int_{D_{R}}u_{k}\,e^{-f}\omega^{n}\rightarrow\int_{D_{R}}u_{\infty}\,e^{-f}\omega^{n}=u_{\infty}\operatorname{vol}_{f}(D_{R})\qquad\textrm{as $k\to+\infty$}.$$
This allows us to derive that
$$|u_{\infty}|=\lim_{k\to+\infty}\frac{\left|\int_{D_{R}}u_{k}\,e^{-f}\omega^{n}\right|}{\operatorname{vol}_{f}(D_{R})}
\leq\lim_{k\to+\infty}\frac{\left(\int_{M\setminus D_{R}}e^{-f}\omega^{n}\right)^{1-\frac{1}{p}}}{\operatorname{vol}_{f}(D_{R})}
=\frac{\operatorname{vol}_{f}(M\setminus D_{R})^{1-\frac{1}{p}}}{\operatorname{vol}_{f}(D_{R})}\to0\quad\textrm{as $R\to+\infty$},$$
where $\operatorname{vol}_{f}(A):=\int_{A}e^{-f}\omega^{n}$ for $A\subseteq M$. That is, $u_{\infty}\equiv 0$.

Next, choose $C>0$ such that $f+C>0$ on $M$, something that is possible to do by Lemma \ref{warm}, and let $\eta:\mathbb{R}\rightarrow\mathbb{R}$ be a smooth function satisfying
$\eta(x)=0$ for $x\leq 1$, $\eta(x)=1$ for $x\geq 2$, and $|\eta(x)|\leq 1$ for all $x$. Define $\eta_{R}:M\to\mathbb{R}$ by
$$\eta_{R}(x)=\eta\left(\frac{\sqrt{f(x)+C}}{R}\right)\qquad\textrm{for $R>0$ a positive constant to be chosen
later}.$$ Then with $\frac{1}{p}+\frac{1}{q}=1$, we have that for some positive constant $C(p)>0$ that may vary from line to line,
\begin{equation}\label{estimate}
\begin{split}
1&=\|u_{k}\|^{p}_{L^{p}(e^{-f}\omega^{n})}\leq C(p)\left(\|(1-\eta_{R})u_{k}\|^{p}_{L^{p}(e^{-f}\omega^{n})}+\|\eta_{R}u_{k}\|^{p}_{L^{p}(e^{-f}\omega^{n})}\right)\\
&\leq C(p)\left(\int_{D_{R}}|u_{k}|^{p}\,e^{-f}\omega^{n}+\int_{M}|\eta_{R}u_{k}|^{p}\,e^{-f}\omega^{n}\right)\\
&\leq C(p)\left(\int_{D_{R}}|u_{k}|^{p}\,e^{-f}\omega^{n}+\int_{M}\left|\eta_{R}u_{k}-\fint_{M}\eta_{R}u_{k}\,e^{-f}\omega^{n}\right|^{p}\,e^{-f}\omega^{n}+\left|\fint_{M}\eta_{R}u_{k}\,e^{-f}\omega^{n}\right|^{p}\right)\\
&\leq C(p)\left(\int_{D_{R}}|u_{k}|^{p}\,e^{-f}\omega^{n}+\int_{M}\left|\eta_{R}u_{k}-\fint_{M}\eta_{R}u_{k}\,e^{-f}\omega^{n}\right|^{p}\,e^{-f}\omega^{n}
+\|u_{k}\|^{p}_{L^{p}(e^{-f}\omega^{n})}\|\eta_{R}\|^{p}_{L^{q}(e^{-f}\omega^{n})}\right)\\
&\leq C(p)\left(\int_{D_{R}}|u_{k}|^{p}\,e^{-f}\omega^{n}+\int_{M}\left|\eta_{R}u_{k}-\fint_{M}\eta_{R}u_{k}\,e^{-f}\omega^{n}\right|^{p}
\,e^{-f}\omega^{n}+\operatorname{vol}_{f}\left(M\setminus D_{\frac{R^{2}}{2}}\right)^{\frac{p}{q}}\right).
\end{split}
\end{equation}

Now, for $R>0$ sufficiently large, $\eta_{R}u_{k}$ is supported on the set where $\omega$ is isometric to $\widehat{\omega}$ via the
biholomorphism $\nu$ of Proposition \ref{background}, a manifold on which we know that the assertion already holds true \cite{milman}. Applying this observation to the middle term in the
last line of \eqref{estimate}, we arrive at the fact that for $R>0$ sufficiently large,
\begin{equation*}
\begin{split}
1&\leq C(p)\left(\int_{D_{R}}|u_{k}|^{p}\,e^{-f}\omega^{n}+\|\nabla^{g}(\eta_{R}u_{k})\|^{p}_{L^{p}(e^{-f}\omega^{n})}+\operatorname{vol}_{f}\left(M\setminus D_{\frac{R^{2}}{2}}\right)^{\frac{p}{q}}\right)\\
&\leq C(p)\left(\int_{D_{R}}|u_{k}|^{p}\,e^{-f}\omega^{n}+\|\nabla^{g}\eta_{R}\|^{p}_{L^{\infty}(M)}\|u_{k}\|^{p}_{L^{p}(e^{-f}\omega^{n})}+\|\nabla^{g}u_{k}\|^{p}_{L^{p}(e^{-f}\omega^{n})}
+\operatorname{vol}_{f}\left(M\setminus D_{\frac{R^{2}}{2}}\right)^{\frac{p}{q}}\right)\\
&\leq C(p)\left(\int_{D_{R}}|u_{k}|^{p}\,e^{-f}\omega^{n}+\frac{1}{R^{p}}+\frac{1}{k^{p}}+\operatorname{vol}_{f}\left(M\setminus D_{\frac{R^{2}}{2}}\right)^{\frac{p}{q}}\right).\\
\end{split}
\end{equation*}
As $u_{k}\to 0$ in $L^{p}_{\operatorname{\operatorname{loc}}}(M)$ as $k\to+\infty$, we see upon letting $k\to+\infty$ that for all $R>0$ sufficiently large,
$$1\leq C(p)\left(\frac{1}{R^{p}}+\operatorname{vol}_{f}\left(M\setminus D_{\frac{R^{2}}{2}}\right)^{\frac{p}{q}}\right).$$
Letting $R\to+\infty$ now yields the desired contradiction.
\end{proof}

\newpage
\section{Linear theory}\label{linear-theory-section}

Working again in the setting and notation of Theorem \ref{mainthm}, we set up the linear theory for metrics asymptotic to $\omega$.
Openness along the continuity path that we apply to solve \eqref{ast-0}
will automatically follow. Although Theorem \ref{mainthm} holds true for torus-invariant functions, in order
to remain as broad as possible, we present the linear theory under minimal assumptions,
namely for $JX$-invariant functions.

\subsection{Main setting}\label{setup}
Let $\tilde{g}$ be any $JX$-invariant K\"ahler metric on $M$ with K\"ahler form $\tilde{\omega}$ and Levi-Civita connection $\nabla^{\tilde{g}}$ satisfying
\begin{equation}\label{hyp-basic-ass}
|(\nabla^g)^{i}\mathcal{L}_X^{(j)}(\tilde{\omega}-\omega)|_{g}=O(r^{-\gamma})\qquad\textrm{for all $i,\,j\geq 0$,}
\end{equation}
for some $\gamma\in(0,\lambda^D)$, where $r=|z|^{\lambda}$ and $\lambda^D$ is the first non-zero eigenvalue of $-\Delta_{D}$ acting on $L^2$-functions on $D$. Write $X=\nabla^{\tilde{g}}\tilde{f}$ for some smooth function $\tilde{f}:M\rightarrow \R$, a function defined up to an additive constant
that is guaranteed to exist because as noted previously, $H^{1}(M,\,\mathbb{R})=0$ by toricity. We use $\nu$ to identify $M$ and $\widehat{M}$ so that $X=r\partial_{r}$ outside a compact set.
Since $\nabla^gf=X=\nabla^{\tilde{g}}\tilde{f}$, it follows from \eqref{hyp-basic-ass} that $|f-\tilde{f}|=O(r^{-\gamma+2})$ as $r\to+\infty$.
{Throughout, we denote $\Delta_{\tilde{g},\,X}:=\Delta_{\tilde{g}}-X$.}

We begin by identifying a good barrier function for this particular geometric setup.

\begin{lemma}\label{lemma-sub-sol-barrier}
For all $\delta\in(0,\,1)$, there exists $R(\delta)>0$ such that the function $e^{\delta f}$ is a sub-solution of the following equation:
\begin{equation*}
\begin{split}
\Delta_{\tilde{g},\,X}e^{\delta f}&\leq 0\qquad\text{on $f\geq R(\delta)$}.
\end{split}
\end{equation*}
Moreover, the logarithm and polynomial powers of $f$ {(which equals $\frac{|z|^{2\lambda}}{2}-1$ outside a compact subset of $M$)} satisfy for all $\delta >0$,
\begin{equation*}
\begin{split}
\Delta_{\tilde{g},\,X}f^{-\delta}=2\delta f^{-\delta}+O(f^{-\delta-1})\qquad\textrm{and}\qquad\Delta_{\tilde{g},\,X}\log (f+1)=-2
\qquad\textrm{outside a compact subset of $M$.}
\end{split}
\end{equation*}
\end{lemma}

\begin{proof}
Using \eqref{hyp-basic-ass} {and the fact that $(\Delta_{\tilde{g}}-\Delta_{g})f=2(\Delta_{\tilde{\omega}}-\Delta_{\omega})f=(\tilde{\omega}-\omega)\ast i\partial\bar{\partial}f=O(|\tilde{g}-g|_{\tilde{g}})$, the last equality because the Hessian of $f$ is bounded on $M$}, we compute that
\begin{equation*}
\begin{split}
\Delta_{\tilde{g},\,X} e^{\delta f}&=\left(\delta \Delta_{\tilde{g},\,X}f+\delta^2|\nabla^{\tilde{g}}f|^2_{\tilde{g}}\right)e^{\delta f}\\
&=\delta\left( \Delta_{g,\,X}f+(\Delta_{\tilde{g},\,X}-\Delta_{g,\,X})f+\delta|\nabla^{\tilde{g}}f|^2_{\tilde{g}}\right)e^{\delta f}\\
&=\delta\left(-2f+\delta|\nabla^{\tilde{g}}f|^2_{\tilde{g}}+O(|\tilde{g}-g|_{\tilde{g}})\right)e^{\delta f}\\
&=\delta\left(-2f+\delta |X|^2_g(1+o(1))+o(1)\right)e^{\delta f}\\
&\leq 0
\end{split}
\end{equation*}
outside a sufficiently large compact subset of $M$. Here we have also used the fact that $|X|^2_g=2f+2$ and $\delta \in(0,\,1)$ in the last line.

A similar computation based on the asymptotics of $\tilde{g}$ given by \eqref{hyp-basic-ass} shows that
\begin{equation*}
\begin{split}
\Delta_{\tilde{g},\,X}f^{-\delta}&=(\Delta_{\tilde{g}}-X)(f^{-\delta})\\
&=-\delta f^{-\delta-1}(\Delta_{\tilde{g}}f-X\cdot f)+\delta(\delta+1)f^{-\delta-2}|\nabla^{\tilde{g}}f|_{\tilde{g}}^{2}\\
&=-\delta f^{-\delta-1}(\Delta_{g}f-X\cdot f)-\delta f^{-\delta-1}(\Delta_{\tilde{g}}f-\Delta_{g}f)+\delta(\delta+1)f^{-\delta-2}|\nabla^{\tilde{g}}f|_{\tilde{g}}^{2}\\
&=2\delta f^{-\delta}-\delta f^{-\delta-1}\underbrace{(\Delta_{\tilde{g}}f-\Delta_{g}f)}_{=\,O(|\tilde{g}-g|_{g})}+\delta(\delta+1)f^{-\delta-2}\underbrace{|\nabla^{\tilde{g}}f|_{\tilde{g}}^{2}}_{=\,O(|X|^{2}_{g})\,=\,O(f)}\\
&=2\delta f^{-\delta}+O(f^{-\delta-1}).
\end{split}
\end{equation*}
As $\log(r^2)$ is pluriharmonic outside a compact set, the fact that $X=r\partial_r$ outside a compact set gives us that
\begin{equation*}
\Delta_{\tilde{g},\,X}\log\left(f+1\right)=\Delta_{\tilde{g},\,X}\log\left(r^{2}\right)=-2
\end{equation*}
outside a compact subset of $M$, as claimed.
\end{proof}

\subsection{Function spaces}\label{function-spaces-subsection}

We next define the function spaces within which we will work.
\begin{itemize}
\item For $\beta\in\R$ and $k$ a non-negative integer, define $C_{X,\,\beta}^{2k}(M)$ to be the space of $JX$-invariant continuous functions $u$ on $M$ with $2k$ continuous derivatives such that
\begin{equation*}
\norm{u}_{C^{2k}_{X,\,\beta}} :=\sum_{i+2j\leq2k}\sup_{M}\left|f^{\frac{\beta}{2}}(\nabla^{\tilde{g}})^i\left(\mathcal{L}_{X}^{(j)}u\right)\right|_{\tilde{g}} < \infty.
\end{equation*}
Thanks to \eqref{hyp-basic-ass}, this norm is equivalent to that defined with respect to the background metric $g$, hence we may
use either $\tilde{g}$ or $g$ with our particular choice depending on the context. Similarly, as $f$ and $\tilde{f}$ are equivalent at infinity, these function
spaces can be defined in terms of either of these two potential functions.
Define $C_{X,\,\beta}^{\infty}(M)$ to be the intersection of the spaces $C_{X,\,\beta}^{2k}(M)$ over all $k\in \N_0$.

{Notice in the definition of the above norm that the number of spatial derivatives that appear in each summand is no more than
twice the number of Lie derivative terms that appear. This is because, when solving the Poisson equation for the weighted Laplacian
as defined in \eqref{setup}, the weighted Laplacian can be treated as a second order parabolic operator with the time derivative corresponding to the $X$-derivative.
These heuristics are used in the proof of Theorem \ref{iso-sch-Laplacian-pol} below. } \\

\item Let $\delta(\tilde{g})$ denote the injectivity radius of $\tilde{g}$, write $d_{\tilde{g}}(x,\,y)$ for the distance with respect to $\tilde{g}$ between two points $x,\,y\in M$,
and let $\varphi^{X}_{t}$ denote the flow of $X$ for time $t$. A tensor $T$ on $M$ is said to be in $C_{\beta}^{0,\,2\alpha}(M)$, $\alpha\in\left(0,\,\frac{1}{2}\right)$, if
 \begin{equation*}
 \begin{split}
\left[T\right]_{C^{0,\,2\alpha}_{\beta}}:=&\sup_{\substack{x\,\neq\,y\,\in\,M \\d_{\tilde{g}}(x,y)\,<\,\delta(\tilde{g})}}\left[\min(f(x),f(y))^{\frac{\beta}{2}}\frac{\arrowvert T(x)-P_{x,\,y}T(y)\arrowvert_h}{d_{\tilde{g}}(x,\,y)^{2\alpha}}\right]\\
&+\sup_{\substack{x\,\in\, M \\ t\,\neq\,s\,\geq\,1}}\left[\min(t,s)^{\frac{\beta}{2}}\frac{\arrowvert (\varphi^{X}_t)_{\ast}T(x)-(\widehat{P}_{\varphi^{X}_s(x),\,\varphi^{X}_t(x)}((\varphi^{X}_s)_{\ast}T(x)))\arrowvert_h}{|t-s|^{\alpha}}\right]<+\infty,
\end{split}
\end{equation*}
where $P_{x,\,y}$ denotes parallel transport along the unique geodesic joining $x$ and $y$, and
$\widehat{P}_{\varphi^{X}_s(x),\,\varphi^{X}_t(x)}$ denotes parallel transport along the unique flow-line of $X$ joining $\varphi^{X}_s(x)$ and $\varphi^{X}_t(x)$.

\item For $\beta\in\R$, $k$ a non-negative integer, and $\alpha\in\left(0,\,\frac{1}{2}\right)$, define the H\"older space $C_{X,\,\beta}^{2k,\,2\alpha}(M)$ with polynomial weight $f^{\frac{\beta}{2}}$ to be the set of $u\in C_{X,\,\beta}^{2k}(M)$ for which the norm
\begin{equation*}
\norm{u}_{C_{X,\,\beta}^{2k,\,2\alpha}}:=\norm{u}_{C^{2k}_{X,\,\beta}} +\sum_{i+2j\,=\, 2k}\left[\left(\nabla^{\tilde{g}}\right)^i\left(\mathcal{L}_{X}^{(j)}u\right)\right]_{C^{0,\,2\alpha}_{\beta}}
\end{equation*}
is finite. It is straightforward to check that the space $C^{2k,\,2\alpha}_{X,\,\beta}(M)$ is a Banach space.
The intersection
$\bigcap_{k\,\geq\,0}C^{2k}_{X,\,\beta}(M)$ we denote by $C^{\infty}_{X,\,\beta}(M)$.

\item We now consider a smooth cut-off function $\chi:M\rightarrow[0,\,1]$ which equals $1$ outside a compact set.
The source function space $\mathcal{D}^{2k+2,\,2\alpha}_{X,\,\beta}(M)$ is defined as
\begin{equation*}
\mathcal{D}^{2k+2,\,2\alpha}_{X,\,\beta}(M):=\left(\R\chi\log r\oplus\R\oplus C^{2k+2,\,2\alpha}_{X,\,\beta}(M)\right),
\end{equation*}
endowed with the norm
\begin{equation*}
\begin{split}
\norm{u}_{\mathcal{D}_{X,\,\beta}^{2k+2,\,2\alpha}}&:=|c_1|+|c_2|+\|\tilde{u}\|_{C_{X,\,\beta}^{2k+2,\,2\alpha}},\\
u&:=c_1\chi \log r+c_2+\tilde{u}.
\end{split}
\end{equation*}
The target function space is defined as
 \begin{equation*}
\mathcal{C}^{2k,\,2\alpha}_{X,\,\beta}(M):=\left(\R\oplus C^{2k,\,2\alpha}_{X,\,\beta}(M)\right),
\end{equation*}
endowed with a norm defined in a similar manner as above. We define
\begin{equation*}
\mathcal{C}^{\infty}_{X,\,\beta}(M):=\bigcap_{k\,\geq\,0} \mathcal{C}^{2k,\,2\alpha}_{X,\,\beta}(M).
\end{equation*}
\item Finally, we define the spaces
\begin{equation*}
\begin{split}
\mathcal{M}^{2k+2,\,2\alpha}_{X,\,\beta}(M)&:=\left\{\varphi\in C^2_{\operatorname{\operatorname{loc}}}(M)\,|\,\tilde{\omega}+i\partial\bar{\partial}\varphi>0\right\}\bigcap \mathcal{D}^{2k+2,\,2\alpha}_{X,\,\beta}(M),
\end{split}
\end{equation*}
and we will work with the following convex set of K\"ahler potentials:
\begin{equation*}
\mathcal{M}^{\infty}_{X,\,\beta}(M)=\bigcap_{k\,\geq\,0}\,\mathcal{M}^{2k+2,\,2\alpha}_{X,\,\beta}(M).
\end{equation*}
Notice that for each $k\geq 0$, the spaces $\mathcal{M}^{2k+2,\,2\alpha}_{X,\,\beta}(M)$ depend on the choice of a background metric $\tilde{\omega}$.
However, these spaces are all equivalent as soon as $\tilde{\omega}$ satisfies \eqref{hyp-basic-ass}.
\end{itemize}

\subsection{Preliminaries and Fredholm properties of the linearised operator}\label{linear}
We proceed with the same set-up as in Section \ref{setup}, beginning with the following useful observation.

\begin{lemma}\label{lemma-preserved-int}
Let $(\varphi_t)_{t\in[0,\,1]}$ be a $C^1$-path of smooth functions in $\mathcal{M}^{\infty}_{X,\,\beta}(M)$ for some $\beta>0$ and write
$\tilde{\omega}_{t}:=\tilde{\omega}+i\partial\bar{\partial}\varphi_{t}>0$ and $\tilde{f}_{t}:=\tilde{f}+\frac{X}{2}\cdot\varphi_{t}$ so that $-d\tilde{\omega}_{t}\lrcorner JX=d\tilde{f}_{t}$.
\begin{enumerate}
  \item
 Let $G:\R\rightarrow \R$ be a $C^1$-function such that for some $-\infty<\alpha<1$, $|G(x)|+|G'(x)|\leq e^{\alpha x}$, $x\geq -C$.
Then
\begin{equation*}
\int_MG(\tilde{f}_{t})\,e^{-\tilde{f}_{t}}\tilde{\omega}_{t}^n=\int_MG(\tilde{f}_{0})\,e^{-\tilde{f}_{0}}\tilde{\omega}_{0}^n,\qquad t\in[0,\,1].
\end{equation*}
  \item $\int_{0}^{1}\int_{M}|\dot{\varphi}_{t}|\,e^{-\tilde{f}_{t}}\tilde{\omega}^{n}_{t}\,dt<+\infty$ and $\int_{0}^{1}\int_{M}|\dot{\varphi}_{t}|\,e^{-\tilde{f}}\tilde{\omega}^{n}\,dt<+\infty$.
\end{enumerate}
\end{lemma}

\begin{proof}
\begin{enumerate}
\item By differentiating, one sees that
\begin{equation*}
\begin{split}
\frac{d}{dt}\left(\int_MG(\tilde{f}_{t})\,e^{-\tilde{f}_{t}}\tilde{\omega}_{t}^n\right)&=\int_MG'(\tilde{f}_{t})\frac{X}{2}\cdot \dot{\varphi}_t\,e^{-\tilde{f}_{t}}\tilde{\omega}_{t}^n+\int_MG(\tilde{f}_{t})\left(\Delta_{\tilde{\omega}_{t}}\dot{\varphi}_t-\frac{X}{2}\cdot \dot{\varphi}_t\right)\,e^{-\tilde{f}_{t}}\tilde{\omega}_{t}^n\\
&=\int_MG'(\tilde{f}_{t})\frac{X}{2}\cdot \dot{\varphi}_t\,e^{-\tilde{f}_{t}}\tilde{\omega}_{t}^n-\frac{1}{2}\int_MG'(\tilde{f}_{t})\nabla^{g_{\varphi_t}}\tilde{f}_{t}\cdot \dot{\varphi}_t\,e^{-\tilde{f}_{t}}\tilde{\omega}_{t}^n\\
&=0.
\end{split}
\end{equation*}
Here, we have used integration by parts together with the fact that $X=\nabla^{\tilde{g}_{t}}\tilde{f}_{t}$ for all $t\in[0,\,1]$, where $\tilde{g}_{t}$ denotes the
K\"ahler metric associated to $\tilde{\omega}_{t}$.

\item First note that by definition of the function space, the weighted measures $e^{-\tilde{f}_{t}}\tilde{\omega}^{n}_{t}$ and $e^{-\tilde{f}}\tilde{\omega}^{n}$ are equivalent to each other. Therefore it suffices to verify only that $\int_{0}^{1}\int_{M}|\dot{\varphi}_{t}|\,e^{-\tilde{f}}\tilde{\omega}^{n}\,dt<+\infty$. But
from the definition of the function space and $\tilde{\omega}$, this is trivially satisfied.

\end{enumerate}
\end{proof}

Next, define the following map as in \cite{siepmann}:
\begin{equation*}
\begin{split}
MA_{\tilde{\omega}}:\psi\in&\left\{\varphi\in C^2_{\operatorname{\operatorname{loc}}}(M)\,|\,\tilde{\omega}_{\varphi}:=\tilde{\omega}+i\partial\bar{\partial}\varphi>0\right\}\mapsto\log\left(\frac{\tilde{\omega}_{\psi}^n}{\tilde{\omega}^n}\right)
-\frac{X}{2}\cdot\psi\in\mathbb{R}.
\end{split}
\end{equation*}
For any $\psi\in C_{\operatorname{\operatorname{loc}}}^{2}(M)$, let
$\tilde{g}_{\psi}$ (respectively $\tilde{g}_{t\psi}$) denote the K\"ahler metric associated to the K\"ahler form $\tilde{\omega}_{\psi}$
(resp.~$\tilde{\omega}_{t\psi}$ for any $t\in[0,\,1]$). Brute force computations show that
\begin{equation}
\begin{split}
MA_{\tilde{\omega}}(0)&=0,\nonumber\\
D_{\psi}MA_{\tilde{\omega}}(u)&=\Delta_{\tilde{\omega}_{\psi}}u-\frac{X}{2}\cdot u,\quad u\in C^2_{\operatorname{\operatorname{loc}}}(M),\nonumber\\
\frac{d^2}{dt^2}\left(MA_{\tilde{\omega}}(t\psi)\right)&=\frac{d}{dt}(\Delta_{\tilde{\omega}_{t\psi}}\psi)=-\arrowvert\partial\bar{\partial}\psi\arrowvert^2_{\tilde{g}_{t\psi}}\quad\textrm{for $t\in[0,\,1]$},\label{equ:sec-der}
 \end{split}
 \end{equation}
\begin{equation}\label{equ:taylor-exp}
 \begin{split}
 MA_{\tilde{\omega}}(\psi)&=MA_{\tilde{\omega}}(0)+\left.\frac{d}{dt}\right|_{t\,=\,0}MA_{\tilde{\omega}}(t\psi)+\int_0^1\int_0^{u}\frac{d^2}{dt^2}(MA_{\tilde{\omega}}(t\psi))\,dt\,du\\
 &=\Delta_{\tilde{\omega}}\psi-\frac{X}{2}\cdot\psi-\int_0^1\int_0^{u}\arrowvert \partial\bar{\partial}\psi\arrowvert^2_{\tilde{g}_{t\psi}}\,dt\,du.
 \end{split}
 \end{equation}

The main result of this section is that the drift Laplacian of $\tilde{g}$ is an isomorphism between polynomially weighted function spaces with zero mean value.

\begin{theorem}\label{iso-sch-Laplacian-pol}
Let $\alpha\in\left(0,\,\frac{1}{2}\right)$, $k\in\mathbb{N}$, and $\beta\in(0,\lambda^D)$. Then the drift Laplacian
\begin{equation*}
\begin{split}
\Delta_{\tilde{g},\,X}:\mathcal{D}^{2k+2,\,2\alpha}_{X,\,\beta}(M)\cap\left\{\int_{M}u\,e^{-\tilde{f}}\tilde{\omega}^{n}=0\right\}\rightarrow \mathcal{ C}^{2k,\,2\alpha}_{X,\,\beta}(M)\cap\left\{\int_{M}v\,e^{-\tilde{f}}\tilde{\omega}^{n}=0\right\}
\end{split}
\end{equation*}
is an isomorphism of Banach spaces.
\end{theorem}

\begin{remark}
In the statement of Theorem \ref{iso-sch-Laplacian-pol}, if $D=\P^1$ endowed with its metric of constant sectional curvature $1$,
then $\lambda^{\mathbb{P}^{1}}=2$ and correspondingly $\beta\in(0,\,\lambda^{\mathbb{P}^{1}})=(0,\,2)$. In general, Lichnerowicz's estimate implies that $\lambda^D\geq 2$;
see \cite[Theorem $6.14$]{book:Ballmann} for a proof. The rate $\gamma$ from \eqref{hyp-basic-ass} can
take any value in the interval $(0,\lambda^D)$. In Section \ref{invert-poly}, we will apply Theorem \ref{iso-sch-Laplacian-pol} with $\gamma=\beta$.
\end{remark}

\begin{proof}[Proof of Theorem \ref{iso-sch-Laplacian-pol}]
{First observe that the drift Laplacian $\Delta_{\tilde{g},\,X}$ is symmetric with respect to the
weighted measure $e^{-\tilde{f}}\tilde{\omega}^n$, a measure with finite volume. Set
\begin{equation*}
\begin{split}
H^1_{\tilde{f}}(M)&:=\left\{u\in H^1_{\operatorname{\operatorname{loc}}}(M)\quad\text{$JX$-invariant}\quad\left.\right|\quad  u\in L^2(e^{-\tilde{f}}\tilde{\omega}^n),\, \nabla^{\tilde{g}}u\in L^2(e^{-\tilde{f}}\tilde{\omega}^n)\right\},\\
W^2_{\tilde{f}}(M)&:=\left\{u\in H^1_{\tilde{f}}(M)\,\quad\left.\right|\quad \Delta_{\tilde{\omega},\,X}u\in L^2(e^{-\tilde{f}}\tilde{\omega}^n)\right\},
\end{split}
\end{equation*}
endowed with the obvious norms induced by that of $L^2(e^{-\tilde{f}}\tilde{\omega}^n)$.
It can be shown that the operator $\Delta_{\tilde{g},\,X}$ restricted to compactly supported smooth $JX$-invariant functions admits a unique self-adjoint extension to $W^2_{\tilde{f}}(M)$,
with domain contained in $H_{\tilde{f}}^1(M)$ and with discrete $L^2(e^{-\tilde{f}}\tilde{\omega}^n)$-spectrum; see \cite[Proposition $6.13$]{Der-Com-Egs} and \cite[Theorem $4.6$]{Gri-Boo} in the context of expanding gradient Ricci solitons, but whose proofs can be adapted to the current situation.
Observe also that the kernel (and hence the cokernel) of this operator is the constant functions. By considering any function $F$ in the codomain as an element of the weighted $L^2$-space $L^2(e^{-\tilde{f}}\tilde{\omega}^n)$, we can therefore find a unique weak solution $u\in H^1(e^{-\tilde{f}}\tilde{\omega}^n)$ with zero weighted mean value of the equation}
\begin{equation}\label{love-drift-lap}
\Delta_{\tilde{g},\,X} u=F.
\end{equation}
In addition, we have the estimate
\begin{equation}
\begin{split}\label{weak-apriori-bd-lin-th}
\|u\|_{L^2(e^{-\tilde{f}}\tilde{\omega}^n)}+\|\nabla^{\tilde{g}} u\|_{L^2(e^{-\tilde{f}}\tilde{\omega}^n)}&\leq C\|F\|_{L^2(e^{-\tilde{f}}\tilde{\omega}^n)}\leq C\|F\|_{C^0}
\end{split}
\end{equation}
for some positive constant $C$ independent of $u$ and $F$ that may vary from line to line. This estimate essentially follows from the weighted $L^2$-Poincar\'e inequality with respect to the drift Laplacian $\Delta_{\tilde{g}}-X\cdot$. We improve on the regularity of $u$ through a series of claims.

\begin{claim}\label{claim-first-rough-growth}
There exists a positive constant $C=C(\tilde{\omega},n)$ such that $$|u(x)|\leq Ce^{ \frac{\tilde{f}(x)}{2}}\|F\|_{C^0},\qquad x\in M.$$
\end{claim}

\begin{proof}[Proof of Claim \ref{claim-first-rough-growth}]
By conjugating \eqref{love-drift-lap} with a suitable weight, notice that the function $v:=e^{-\frac{\tilde{f}}{2}}u$ satisfies
\begin{equation*}
\begin{split}
\Delta_{\tilde{g}}v&=e^{-\frac{\tilde{f}}{2}}F+\left(\frac{1}{4}|X|^2_{\tilde{g}}-\frac{1}{2}\Delta_{\tilde{g}}\tilde{f}\right)v.
\end{split}
\end{equation*}
This implies that $|v|$ satisfies the following differential inequality in the weak sense:
\begin{equation}
\Delta_{\tilde{g}}|v|\geq -C|v|-C\|F\|_{C^0}.\label{love-drift-lap-conj}
\end{equation}
Here we have made use of the non-negativity of $|X|^2_{\tilde{g}}$ together with the boundedness of $\Delta_{\tilde{g}}\tilde{f}$ given by \eqref{hyp-basic-ass}.

We perform a local Nash-Moser iteration on \eqref{love-drift-lap-conj} in $B_{\tilde{g}}(x,\,r)$.
More precisely, since $(M^{2n},\,\tilde{g})$ is a Riemannian manifold with Ricci curvature bounded from below,
the results of \cite{Sal-Cos-Uni-Ell} yield the following local Sobolev inequality:
\begin{equation}
\begin{split}\label{sob-inequ-loc}
\left(\frac{1}{\operatorname{vol}_{\tilde{g}} (B_{\tilde{g}}(x,\,r))}\int_{B_{\tilde{g}}(x,\,r)}|\varphi|^{\frac{2n}{n-1}}\,\tilde{\omega}^{n}\right)^{\frac{n-1}{n}}\leq \left(\frac{C(r_0)r^2}{\operatorname{vol}_{\tilde{g}}(B_{\tilde{g}}(x,\,r))}\int_{B_{\tilde{g}}(x,\,r)}|\widetilde{\nabla}\varphi|_{\tilde{g}}^2\,\tilde{\omega}^{n}\right)
\end{split}
\end{equation}
for any $\varphi\in H^1_0(B_{\tilde{g}}(x,\,r))$ and for all $x\in M$ and $0<r<r_0$, where $r_0$ is some fixed positive radius.

A Nash-Moser iteration proceeds in several steps. First, one multiplies \eqref{love-drift-lap}
across by\linebreak  $\eta_{s,\,s'}^2v|v|^{2(p-1)}$ with $p\geq 1$, where $\eta_{s,\,s'}$, with $0<s+s'<r$ and $s,s'>0$,
is a Lipschitz cut-off function with compact support in $B_{\tilde{g}}(x,\,s+s')$ equal to $1$ on $B_{\tilde{g}}(x,\,s)$ and with
$|\widetilde{\nabla}\eta_{s,\,s'}|_{\tilde{g}}\leq\frac{1}{s'}$ almost everywhere. One then integrates by parts and uses
the Sobolev inequality of \eqref{sob-inequ-loc} to obtain a so-called ``reversed H\" older inequality'' which, after iteration, leads to the bound
\begin{equation*}
\begin{split}
\sup_{B_{\tilde{g}}(x,\,\frac{r}{2})}|v|\leq&\, C\left(\|v\|_{L^2(B_{\tilde{g}}(x,\,r))}+\|F\|_{L^{\infty}(B_{\tilde{g}}(x,\,r))}\right)\\
\leq&\,C\left(\|u\|_{L^2(e^{-\tilde{f}}\tilde{\omega}^{n})}+\|F\|_{C^0(M)}\right)\\
\leq&\,C\|F\|_{C^0(M)}\\
\end{split}
\end{equation*}
for $r\leq r_0$, where $C=C(r_0,\,\tilde{\omega},n)$. Here we have made use of \eqref{weak-apriori-bd-lin-th} in the last line.
This estimate yields an a priori local $C^0$-estimate which is uniform on the center of the ball $B_{\tilde{g}}(x,\,\frac{r}{2})$. In particular,
unravelling the definition of the function $v$, one obtains the expected a priori uniform exponential growth, namely
\begin{equation*}
|u(x)|\leq Ce^{\frac{\tilde{f}(x)}{2}}\|F\|_{C^0},\qquad x\in M.
\end{equation*}
\end{proof}

Thanks to Claim \ref{claim-first-rough-growth}, by local Schauder elliptic estimates we actually see that $u$ lies in $C^{2k+2,2\alpha}_{\operatorname{loc}}$ and that we have the estimates
\begin{equation}\label{first-loc-a-priori-est-lin-th}
\|u\|_{C^{2k+2\alpha}(\{\tilde{f}\,<\,R\})}\leq C\|F\|_{C^{2k,2\alpha}(\{\tilde{f}\,<\,2R\})}\leq C\|F\|_{\mathcal{C}^{2k,\,2\alpha}_{X,\,\beta}}
\end{equation}
for some positive constant $C=C(R,\,\tilde{\omega},n)$. We now proceed to prove the expected a priori weighted estimates on $u$ and on its derivatives.

\begin{claim}\label{claim-sec-rough-growth}
There exists a positive constant $A=A(\tilde{\omega},\,n)$ such that
$$|u(x)|\leq A\log\tilde{f}(x)\|F\|_{C^0}\qquad\textrm{{for all $x\in M$ with $\tilde{f}(x)\geq 2$.}}$$
\end{claim}

\begin{proof}[Proof of Claim \ref{claim-sec-rough-growth}]
Let $\varepsilon>0$ and let $\delta \in(0,1)$ be
such that $\lim_{\tilde{f}\rightarrow+\infty}\left(u-\varepsilon e^{\delta \tilde{f}}\right)=-\infty$, parameters that we can choose by
Claim \ref{claim-first-rough-growth}. For $A>0$ a constant to be determined later, we have outside a compact set $\{\tilde{f}\geq R(\delta)\}$ the inequality
\begin{equation*}
\begin{split}
\Delta_{\tilde{g},\,X}\left(u-A\log (f+1)-\varepsilon e^{\delta \tilde{f}}\right)\geq -\|F\|_{C^0}+2A> 0,
\end{split}
\end{equation*}
so long as $A>\frac{1}{2}\|F\|_{C^0}$. Here Lemma \ref{lemma-sub-sol-barrier} has been applied. Appealing to the maximum principle then yields the bound
$$\max_{\{\tilde{f}\geq R(\delta)\}}\left(u-A\log (f+1)-\varepsilon e^{\delta \tilde{f}}\right)=\max_{\{\tilde{f}= R(\delta)\}}\left(u-A\log (f+1)-\varepsilon e^{\delta \tilde{f}}\right).$$
Next, letting $\varepsilon\to0$, we see that
\begin{equation*}
u-A\log (f+1)\leq \max_{\{\tilde{f}= R(\delta)\}}\left(u-A\log (f+1)\right)\leq 0
\end{equation*}
if we set $A:= C\max_{\{\tilde{f}= R(\delta)\}}u\leq C\|F\|_{C^0}$ with $C:=C(\delta,\tilde{\omega},n)$. This we can do thanks to \eqref{first-loc-a-priori-est-lin-th}.

Applying the same argument to $-u$ concludes the proof of the claim.
\end{proof}

Observe that $\tilde{u}:=u+c\chi\log r$, where $F-c\in C^{2k,2\alpha}_{X,\beta}(M)$, satisfies the equation
\begin{equation}\label{heat-eqn-tilde-u}
\Delta_{\tilde{g},\,X}\tilde{u}=F+c\Delta_{\tilde{g},\,X}(\chi\log r)=F-c+\underbrace{c+c\Delta_{\tilde{g},\,X}(\chi\log r)}_{\text{compactly supported}}:=\tilde{F}\in C^{2k,2\alpha}_{X,\beta}(M).
\end{equation}
The next claim estimates the $C^{2k+2,2\alpha}_{\operatorname{loc}}$-norms of $\tilde{u}$ in terms of the data $F$ and of its local $C^0$-norm.
For this purpose, define the corresponding solution to the Ricci flow $g(\tau):=(-\tau)(\phi^X_{\tau})^*g$ for $\tau<0$,
where $\partial_{\tau}\phi^X_{\tau}=\frac{X}{2(-\tau)}\circ\phi^X_{\tau}$ and $\phi^X_{\tau=-1}=\Id_{\C\times D}$. Here,
$\phi^X_{\tau}(z,\theta)=(\frac{z}{\sqrt{-\tau}},\theta)$ for $(z,\theta)\in \C\times D$. In particular,
if $A_{r_1,r_2}:=\{(z,\theta)\in\C\times D\,|\, r_1\leq |z|\leq r_2\}$ for $0\leq r_1<r_2$, then $\phi^X_{\tau}\left(A_{r_1,r_2}\right)=A_{\frac{r_1}{\sqrt{-\tau}},\frac{r_2}{\sqrt{-\tau}}}$.

\begin{claim}\label{claim-a-priori-rough-bd-hih-der}
There exists a radius $r_0>0$ and a positive constant $C$ such that if $r\geq r_0$, then
\begin{equation}\label{est-sch-loc-para}
\|\tilde{u}\|_{C^{2k+2,2\alpha}_{X,0}\left(A_{r(x)-C,r(x)+C}\right)}\leq C\left(\|\tilde{u}\|_{C^{0}\left(A_{\frac{r(x)}{C},Cr(x)}\right)}+\|\tilde{F}\|_{C^{2k,2\alpha}_{X,0}\left(A_{\frac{r(x)}{C},Cr(x)}\right)}\right).
\end{equation}
Moreover,
\begin{equation}\label{rough-est-cov-der-tilde-u}
\left|X\cdot \tilde{u}\right|+\left|\nabla^{\tilde{g}}\tilde{u}\right|_{\tilde{g}}+\left|\nabla^{\tilde{g},2}\tilde{u}\right|_{\tilde{g}}\leq C\log r\|F\|_{\mathcal{C}^{2k,\,2\alpha}_{X,\,\beta}},\qquad r\geq r_0.
\end{equation}
\end{claim}

\begin{proof}[Proof of Claim \ref{claim-a-priori-rough-bd-hih-der}]
Since \eqref{heat-eqn-tilde-u} is expressed in terms of the Riemannian metric $\tilde{g}$,
we define analogously the family of metrics $\tilde{g}(\tau):=(-\tau)(\phi^X_{\tau})^*\tilde{g}$ for $\tau<0$, where $\partial_{\tau}\phi^X_{\tau}=\frac{X}{2(-\tau)}\circ\phi^X_{\tau}$ and $\phi^X_{\tau=-1}=\Id_{\C\times D}$. For $-\tau\in\left[\frac{1}{2},\,2\right]$, the metrics $\tilde{g}(\tau)$
are uniformly equivalent and their covariant derivatives (with respect to $g$) and time derivatives are bounded by \eqref{hyp-basic-ass}. Now, $\tilde{u}(\tau):=(\phi^X_{\tau})^*\tilde{u}$ satisfies
\begin{equation}\label{heat-eqn-disguise}
\partial_{\tau}\tilde{u}=\Delta_{\tilde{g}(\tau)}\tilde{u}+\tilde{F}(\tau),\qquad \tilde{F}(\tau):=-(-\tau)^{-1}(\phi^X_{\tau})^*F.
\end{equation}
Standard parabolic Schauder estimates applied to \eqref{heat-eqn-disguise} on a ball $B_{g}(x,r_0)$, $2r_0<\inj(g)$, then ensure the existence of a uniform positive constant $C$ such that
\begin{equation*}
\|\tilde{u}(\tau)\|_{C^{2k+2,\,2\alpha}(B_{g}(x,r_0)\times\left[-\frac{3}{2},\,-1\right])}\leq C\left(\|\tilde{u}(\tau)\|_{C^0(B_{g}(x,2r_0)\times\left[-2,\,-\frac{1}{2}\right])}
+\|\tilde{F}(\tau)\|_{C^{2k,2\alpha}(B_{g}(x,2r_0)\times\left[-2,\,-\frac{1}{2}\right])}\right).
\end{equation*}
Unravelling the definition of the function $\tilde{u}(\tau)$ and that of the metrics $\tilde{g}(\tau)$ then
yields \eqref{est-sch-loc-para} after observing that $$\bigcup_{\tau\in\left[-2,\,-\frac{1}{2}\right]}\phi^X_{\tau}\left(B_{g}(x,\,2r_0)\right)
\subset A_{\frac{r(x)}{\sqrt{2}}-\sqrt{2}r_0,\sqrt{2}r(x)+2\sqrt{2}r_0}.$$
The final estimate \eqref{rough-est-cov-der-tilde-u} is a straightforward combination of \eqref{est-sch-loc-para} together with the a priori bound from Claim \ref{claim-sec-rough-growth}.
\end{proof}

Now we are in a position to linearise equation \eqref{love-drift-lap} outside a compact set with respect to the background metric. Namely, we write
\begin{equation}\label{linearise-covid-0}
\Delta_{g,\,X}\tilde{u}=\tilde{F}+(\Delta_{g}-\Delta_{\tilde{g}})\tilde{u}:=G,
\end{equation}
where $G$ satisfies pointwise estimate
\begin{equation}\label{data-rough-est}
\begin{split}
G-\tilde{F}&=(g^{-1}-\tilde{g}^{-1})\ast\partial\bar{\partial}u=O(r^{-\gamma})|\partial\bar{\partial}u|_g,
\end{split}
\end{equation}
here $\ast$ denoting any linear combination of contractions of tensors with respect to the metric $g$. Indeed, this estimate holds true by virtue of \eqref{hyp-basic-ass}. We rewrite \eqref{linearise-covid-0} (outside a compact set) as follows:
\begin{equation}\label{linearise-covid}
\Delta_{C}\tilde{u}-X\cdot\tilde{u}+\Delta_{D}\tilde{u}=G.
\end{equation}
Here $\Delta_{C}$ and $\Delta_{D}$ denote the Riemannian Laplacian of the metric $\omega_{C}$ on $\C$ and $\omega_{D}$ on $D$ respectively.
Integrating this equation over $D$ at a sufficiently large height $r$, we find that
\begin{equation}\label{nomad}
\Delta_{C,\,X}\overline{u}(r)=\overline{G}(r),\qquad r\geq r_0,
\end{equation}
where $$\overline{u}(r):=\fint_{D}\tilde{u}(r,\cdot)\,\omega^{n-1}_{D}\qquad\textrm{and}\qquad \overline{G}(r)=\fint_{D}G(r,\cdot)\,\omega_{D}^{n-1},$$
both functions in the $r$-variable only because both are $JX$-invariant by definition. We next derive some estimates on $\overline{u}(r)$.

\begin{claim}\label{cla-c-0-bd-mean-val}
One has
\begin{equation*}
\left|\overline{u}(r)\right|\leq C\|\tilde{F}\|_{C^{2k,2\alpha}_{X,\beta}},\qquad r\geq r_0.
\end{equation*}
Moreover, $\lim_{r\rightarrow+\infty}\overline{u}(r)=:u_{\infty}$ exists, is finite, and
\begin{equation*}
\left|\overline{u}(r)-u_{\infty}\right|\leq C\left(r^{-\beta}\|\tilde{F}\|_{C^{2k,2\alpha}_{X,\beta}}+r^{-\gamma}\sup_{\{f\geq \frac{r^2}{2}\}}|\partial\bar{\partial}u| \right),\qquad r\geq r_0.
\end{equation*}
\end{claim}

\begin{proof}[Proof of Claim \ref{cla-c-0-bd-mean-val}]
Equation \eqref{nomad} can be rewritten as
\begin{equation}\label{antecedent-neat-a-priori-ode-bar-u}
\left|\frac{X\cdot X\cdot \overline{u}(r)}{r^{2}}-X\cdot \overline{u}(r)\right|\leq C\left(r^{-\beta}\|\tilde{F}\|_{C^{2k,2\alpha}_{X,\beta}}+r^{-\gamma}
\sup_{\{f\geq \frac{r^2}{2}\}}|\partial\bar{\partial}u|\right),\qquad r\geq r_0,
\end{equation}
by virtue of \eqref{data-rough-est}.  This is a first order differential inequality for $X\cdot\overline{u}(r)$.
Now, estimate \eqref{rough-est-cov-der-tilde-u} from Claim \ref{claim-a-priori-rough-bd-hih-der} implies a first rough estimate, namely
\begin{equation*}
\left|\frac{X\cdot X\cdot \overline{u}(r)}{r^{2}}-X\cdot \overline{u}(r)\right|\leq Cr^{-\min\{\beta,\gamma\}}\left(1+\log r\right)\|\tilde{F}\|_{C^{2k,2\alpha}_{X,\beta}},\qquad r\geq r_0.\label{first-neat-a-priori-ode-bar-u}
\end{equation*}
Gr\"onwall's inequality then leads to the bound
\begin{equation*}
\begin{split}
|X\cdot \overline{u}(r)|&\leq C\|\tilde{F}\|_{C^{2k,2\alpha}_{X,\beta}}e^{\frac{r^2}{2}}\int_r^{+\infty}s^{-\min\{\beta,\gamma\}}\left(1+\log s\right)\,se^{-\frac{s^2}{2}}ds\\
&\leq C\|\tilde{F}\|_{C^{2k,2\alpha}_{X,\beta}}r^{-\min\{\beta,\gamma\}}\log r,\qquad r\geq r_0,
\end{split}
\end{equation*}
for some uniform positive constant $C$ independent of $r\geq r_0$.
Integrating once more in $r$, Claim \ref{claim-sec-rough-growth} ensures that $\overline{u}(r)$ admits a limit $u_{\infty}$ as $r\to+\infty$ and that for $r\geq r_0$,
\begin{equation*}
\begin{split}
|\overline{u}(r)|&\leq |\overline{u}(r_0)|+C\|\tilde{F}\|_{C^{2k,2\alpha}_{X,\beta}}\int_{r_0}^rs^{-\min\{\beta,\gamma\}-1}\log s\,ds\\
&\leq C\|\tilde{F}\|_{C^{2k,2\alpha}_{X,\beta}}
\end{split}
\end{equation*}
for some positive constant $C$ which is independent of $r$ (and of the data $F$) that may vary from line to line. This concludes the proof of the first part of the claim.

Returning to inequality \eqref{antecedent-neat-a-priori-ode-bar-u}, another application Gr\"onwall's inequality leads to the bound
\begin{equation*}
\begin{split}
|X\cdot \overline{u}(r)|&\leq Ce^{\frac{r^2}{2}}\left(\int_r^{+\infty}s^{-\beta}\,se^{-\frac{s^2}{2}}ds\|\tilde{F}\|_{C^{2k,2\alpha}_{X,\beta}}+\int_r^{+\infty}s^{-\gamma}\,se^{-\frac{s^2}{2}}ds\sup_{\{f\geq \frac{r^2}{2}\}}|\partial\bar{\partial}u|\right)\\
&\leq C\left(r^{-\beta}\|\tilde{F}\|_{C^{2k,2\alpha}_{X,\beta}}+r^{-\gamma}\sup_{\{f\geq \frac{r^2}{2}\}}|\partial\bar{\partial}u|\right), \qquad r\geq r_0.
\end{split}
\end{equation*}
Integrating this inequality once more between $r$ and $r=+\infty$ yields the second part of the claim.
\end{proof}

The next claim concerns the uniform boundedness of the projection of $u$ onto the orthogonal complement of the kernel of $\Delta_{D}$, $D$ being interpreted as embedded in each level set $\{f=\frac{r^2}{2}\}$.
\begin{claim}\label{cla-c-0-orth-mean-val}
Given $\delta\in(0,\,\min\{\beta,\gamma\})$, there exists $r_0=r_0(\delta,\tilde{\omega},n)$ such that
\begin{equation*}
\|\tilde{u}-\overline{u}(r)\|_{L^{2}(D)}\leq C\|\tilde{F}\|_{C^{2k,2\alpha}_{X,\beta}}r^{-\delta},\qquad r\geq r_0.
\end{equation*}
\end{claim}

\begin{proof}[Proof of Claim \ref{cla-c-0-orth-mean-val}]
Recall that by \eqref{linearise-covid} and \eqref{nomad}, $\Delta_{g,\,X}\tilde{u}=G$ so that
\begin{equation}\label{eq-gov-u-mean-value}
\Delta_{g,\,X}(\tilde{u}-\overline{u}(r))=G-\overline{G}(r)
\end{equation}
outside a compact set. Since for any function $v$, we have
$$2v\Delta_{C,\,X}v=\Delta_{C,\,X}(v^{2})-2|\nabla^{\mathbb{C}}v|^{2}_{g_{\mathbb{C}}},$$
multiplying \eqref{eq-gov-u-mean-value} across by $\tilde{u}-\overline{u}(r)$ and integrating over
$D$ gives rise to the lower bound
\begin{equation}
\begin{split}\label{ineq-ene-no-young}
\Delta_{C,\,X}\left(\|\tilde{u}-\overline{u}(r)\|^{2}_{L^{2}(D)}\right)&\geq
\Delta_{C,\,X}\left(\|\tilde{u}-\overline{u}(r)\|^{2}_{L^{2}(D)}\right)-2\int_{\mathbb{P}^{1}}|\nabla^{\mathbb{C}}(\tilde{u}-\overline{u}(r))|_{g_{\mathbb{C}}}^{2}\,\frac{\omega_{D}^{n-1}}{(n-1)!}\\
&=2\int_{D}(\tilde{u}-\overline{u}(r))\Delta_{C,\,X}(\tilde{u}-\overline{u}(r))\,\frac{\omega_{D}^{n-1}}{(n-1)!}\\
&=2\int_{D}(\tilde{u}-\overline{u}(r))(G-\overline{G}(r)-\Delta_{D}(\tilde{u}-\overline{u}(r)))\,\frac{\omega_{D}^{n-1}}{(n-1)!}\\
&=2\|\nabla^{g_D}(\tilde{u}-\overline{u}(r))\|_{L^{2}(D)}^{2}+2\langle G-\overline{G}(r),\,\tilde{u}-\overline{u}(r)\rangle_{L^{2}(D)}\\
&\geq2\lambda^{D}\|\tilde{u}-\overline{u}(r)\|^{2}_{L^{2}(D)}-2\| G-\overline{G}(r)\|_{L^{2}(D)}\|\tilde{u}-\overline{u}(r)\|_{L^{2}(D)},
\end{split}
\end{equation}
where we have made use of the Poincar\'e inequality on $(D,\,g_D)$ in the last line. Young's inequality then implies for $\varepsilon\in (0,\lambda^{D})$ that
\begin{equation*}
\begin{split}
\Delta_{C,\,X}\left(\|\tilde{u}-\overline{u}(r)\|^{2}_{L^{2}(D)}\right)\geq 2(\lambda^{D}-\varepsilon)\|\tilde{u}-\overline{u}(r)\|^{2}_{L^{2}(D)}-C_{\varepsilon}\|G-\overline{G}(r)\|^2_{L^{2}(D)}.
\end{split}
\end{equation*}
Therefore, invoking estimate \eqref{data-rough-est} and Claim \ref{claim-a-priori-rough-bd-hih-der} together with the previous inequality, we find that
\begin{equation*}
\begin{split}
\Delta_{C,\,X}\left(\|\tilde{u}-\overline{u}(r)\|^{2}_{L^{2}(D)}\right)&\geq 2(\lambda^D-\varepsilon)\|\tilde{u}-\overline{u}(r)\|^{2}_{L^{2}(D)}-C_{\varepsilon}\|\tilde{F}\|^2_{C^{2k,2\alpha}_{X,\beta}}r^{-2\min\{\beta,\gamma\}}\log^2 r,\qquad r\geq 2.
\end{split}
\end{equation*}
By Lemma \ref{lemma-sub-sol-barrier} applied to $\tilde{g}:=g$, we see that
$$\Delta_{C,\,X}(r^{-2\delta})=2\delta r^{-2\delta}+O(r^{-2\delta-2}),$$
which, for $A>0$ and $\delta\in(0,\min\{\beta,\gamma\})$, implies that
\begin{equation*}
\begin{split}
\Delta_{C,\,X}\left(\|\tilde{u}-\overline{u}(r)\|^{2}_{L^{2}(D)}-Ar^{-2\delta}\right)
&\geq2(\lambda^D-\varepsilon)\|\tilde{u}-\overline{u}(r)\|^{2}_{L^{2}(D)}-C_{\varepsilon}\|\tilde{F}\|^2_{C^{2k,2\alpha}_{X,\beta}}r^{-2\min\{\beta,\gamma\}}\log^2 r
\\
&\quad-2A\delta r^{-2\delta}-ACr^{-2\delta-2}\\
&\geq2(\lambda^D-\varepsilon)\left(\|\tilde{u}-\overline{u}(r)\|^{2}_{L^{2}(D)}-Ar^{-2\delta}\right)+2A(\lambda^D-\varepsilon-\delta)r^{-2\delta}\\
&\quad-C_{\varepsilon}\|\tilde{F}\|^2_{C^{2k,2\alpha}_{X,\beta}}r^{-2\min\{\beta,\gamma\}}\log^2 r-ACr^{-2\delta-2}\\
&\geq 2(\lambda^D-\varepsilon)\left(\|\tilde{u}-\overline{u}(r)\|^{2}_{L^{2}(D)}-Ar^{-2\delta}\right),
\end{split}
\end{equation*}
provided that $\varepsilon\in(0,\lambda^D-\delta)$, $r\geq r_0=r_0(\delta,n,\tilde{\omega}),$ and $A\geq C\|\tilde{F}\|_{C^{2k,2\alpha}_{X,\beta}}.$

Now, since $\|\tilde{u}-\overline{u}(r)\|_{L^{2}(D)}$ is growing at most logarithmically by Claim \ref{claim-sec-rough-growth}, given $B>0$, we compute that
\begin{equation*}
\begin{split}
\Delta_{C,\,X}\left(\|\tilde{u}-\overline{u}(r)\|^{2}_{L^{2}(D)}-Ar^{-2\delta}-B r\right)
&\geq2(\lambda^D-\varepsilon)\left(\|\tilde{u}-\overline{u}(r)\|^{2}_{L^{2}(D)}-Ar^{-2\delta}-B r\right)
\end{split}
\end{equation*}
if $\varepsilon\in(0,\lambda^D-\delta)$, $r\geq r_0=r_0(\delta,n,\tilde{\omega}),$ and $A\geq C\|\tilde{F}\|_{C^{2k,2\alpha}_{X,\beta}}.$ In particular, the maximum principle applied to the function $\|\tilde{u}-\overline{u}(r)\|^{2}_{L^{2}(D)}-Ar^{-2\delta}-B r$ outside a compact set of the form $r\geq r_0$ leads to the equality
\begin{equation*}
\begin{split}
\max_{\{r\geq r_0\}}\left(\|\tilde{u}-\overline{u}(r)\|^{2}_{L^{2}(D)}-Ar^{-2\delta}-B r\right)=\max\left\{0,\max_{\{r= r_0\}}\left(\|\tilde{u}-\overline{u}(r)\|^{2}_{L^{2}(D)}-Ar^{-2\delta}-B r\right)\right\}.
\end{split}
\end{equation*}
Letting $B\to0$ and setting $A=C\|\tilde{F}\|_{C^{2k,2\alpha}_{X,\beta}}$ with $C$ sufficiently large but uniform in the data $F$ and the radius $r$, one arrives at the expected bound:
 \begin{equation*}
\|\tilde{u}-\overline{u}(r)\|_{L^{2}(D)}\leq C\|\tilde{F}\|_{C^{2k,2\alpha}_{X,\beta}}r^{-\delta},\qquad r\geq r_0.
\end{equation*}
\end{proof}
The next claim proves a quantitative almost sharp weighted $C^0$-estimate on $\tilde{u}-u_{\infty}$ in terms of the data $F$.
\begin{claim}\label{claim-first-rough-dec-pt}
Given $\delta\in(0,\min\{\beta,\gamma\})$, there exists $r_0=r_0(\delta,\tilde{\omega},n)>0$ independent of $F$ (and the solution $u$) such that
\begin{equation*}
\sup_{r\geq r_0}r^{\delta}|\tilde{u}-u_{\infty}|\leq C\|\tilde{F}\|_{C^{2k,2\alpha}_{X,\beta}}.
\end{equation*}
\end{claim}

\begin{proof}[Proof of Claim \ref{claim-first-rough-dec-pt}]
It suffices to prove that for all $\delta\in(0,\min\{\beta,\gamma\})$, there exists $r_0=r_0(\delta,n,\tilde{\omega})>0$ such that
\begin{equation}\label{variant-sufficient}
\sup_{r\geq r_0}r^{\delta}|\tilde{u}-\overline{u}(r)|\leq C\|\tilde{F}\|_{C^{2k,2\alpha}_{X,\beta}}.
\end{equation}
Indeed, the triangle inequality together with Claims \ref{claim-a-priori-rough-bd-hih-der} and \ref{cla-c-0-bd-mean-val} already
yield such a uniform $C^0$-polynomial rate on the difference $\overline{u}(r)-u_{\infty}$.

In order to prove \eqref{variant-sufficient}, we apply a local parabolic Nash-Moser iteration to the following heat equation
with a source term (see for instance \cite[Theorem $6.17$]{Boo-Lieberman}) by recalling that for $\tau<0$, $\tilde{u}(\tau):=(\phi^X_{\tau})^{*}\tilde{u}$ and
$\overline{u}(r,\tau):=(\phi^X_{\tau})^{*}\overline{u}(r)=\overline{u}\left(\frac{r}{\sqrt{-\tau}}\right)$:
\begin{equation*}
\partial_{\tau}\left(\tilde{u}-\overline{u}(\cdot,\cdot)\right)(\tau)=\Delta_{(-\tau)\cdot g_D}\left(\tilde{u}-\overline{u}(\cdot,\cdot)\right)(\tau)+\underbrace{\Delta_{C}\left(\tilde{u}-\overline{u}(\cdot,\cdot)\right)(\tau)-\left(G-\overline{G}\right)(\tau)}_{:=S(\tau)\quad\text{source term}},\qquad (-\tau)\in\left[\frac{1}{2},\,2\right].
\end{equation*}
Here we have used \eqref{heat-eqn-disguise}, \eqref{linearise-covid-0}, and \eqref{nomad}.
Also, the notation $(-\tau)\cdot g_D$ denotes the metric on $D$ rescaled by $(-\tau)$. In particular, there exists $C>0$ such that if $r\geq r_0$,
\begin{equation}
\begin{split}\label{baby-nash-moser}
\sup_{f=\frac{r^2}{2}}|\tilde{u}-\overline{u}(r)|&=\sup_{f=\frac{r^2}{2}}|\tilde{u}(\tau)-\overline{u}(r,\tau)|_{\tau=-1}\\
&\leq C\sup_{(-\tau)\in[1/2,2]}\left(\|\tilde{u}(\tau)-\overline{u}(r,\tau)\|_{L^2(D)}+|S(\tau)|\right)\\
&\leq C\sup_{s\in[r/\sqrt{2},\sqrt{2}r]}\left(\|\tilde{u}(1)-\overline{u}(s,1)\|_{L^2(D)}+|S(1)|\right).
\end{split}
\end{equation}
The source term can be estimated as follows: if $k\geq 1$, $(-\tau)\in\left[\frac{1}{2},\,2\right]$ and $r\geq r_0$, then
\begin{equation*}
\begin{split}
\left|\Delta_{C}\left(\tilde{u}-\overline{u}(\cdot,\cdot)\right)(\tau)-\left(G-\overline{G}\right)(\tau)\right|&\leq C\left(r^{-\beta}\|\tilde{F}\|_{C^{2k,2\alpha}_{X,\beta}}+r^{-\gamma}\sup_{f\geq \frac{r^2}{4}}|\partial\overline{\partial}u|+r^{-2}\sup_{f\geq \frac{r^2}{4}}|u|\right)\\
&\leq C\|\tilde{F}\|_{C^{2k,2\alpha}_{X,\beta}}r^{-\min\{\beta,\gamma\}}(1+\log r),
\end{split}
\end{equation*}
where we have applied Claim \ref{claim-a-priori-rough-bd-hih-der} to $X\cdot \tilde{u}$ and $X\cdot X\cdot \tilde{u}$ in order to estimate $\Delta_{C}\tilde{u}$.

Finally, thanks to \eqref{baby-nash-moser}, Claim \ref{cla-c-0-orth-mean-val} combined with the above estimate on the source term implies that
\begin{equation*}
\begin{split}
\sup_{f=\frac{r^2}{2}}|\tilde{u}-\overline{u}(r)|&\leq C\|\tilde{F}\|_{C^{2k,2\alpha}_{X,\beta}}r^{-\delta}+C\|\tilde{F}\|_{C^{2k,2\alpha}_{X,\beta}}r^{-\min\{\beta,\gamma\}}(1+\log r)\\
&\leq C\|\tilde{F}\|_{C^{2k,2\alpha}_{X,\beta}}r^{-\delta},\qquad r\geq r_0,
\end{split}
\end{equation*}
as claimed.
\end{proof}

The next claim provides a quantitative sharp weighted $C^0$-estimate on $\tilde{u}-u_{\infty}$ in terms of the data.

\begin{claim}\label{sharp-pt-dec-u}
Given $\beta\in(0,\lambda^D)$, there exists $r_0=r_0(\beta,\tilde{\omega},n)>0$ independent of $F$ (and the solution $u$) such that
\begin{equation*}
\sup_{r\geq r_0}r^{\beta}|\tilde{u}-u_{\infty}|\leq C\|\tilde{F}\|_{C^{2k,2\alpha}_{X,\beta}}.
\end{equation*}
\end{claim}

\begin{proof}[Proof of Claim \ref{sharp-pt-dec-u}]
Applying \eqref{est-sch-loc-para} from Claim \ref{claim-a-priori-rough-bd-hih-der} to $\tilde{u}-u_{\infty}$ together with Claim \ref{claim-first-rough-dec-pt} demonstrates that for $\delta\in(0,\min\{\beta,\gamma\})$,
\begin{equation*}
|X\cdot \tilde{u}|(x)+|\nabla^{\tilde{g}}\tilde{u}|(x)+|\nabla^{\tilde{g},2}\tilde{u}|(x)\leq C\|\tilde{F}\|_{C^{2k,2\alpha}_{X,\beta}}r^{-\delta},\qquad r\geq r_0.
\end{equation*}
 Recalling \eqref{data-rough-est}, the previous estimate implies in turn the following one:
\begin{equation}\label{data-sharp-est}
|G-\tilde{F}|\leq C\|\tilde{F}\|_{C^{2k,2\alpha}_{X,\beta}}r^{-\gamma-\delta},\qquad r\geq r_0.
\end{equation}
On one hand, thanks to Claim \ref{cla-c-0-bd-mean-val}, one obtains an improved decay on $\overline{u}(r)-u_{\infty}$, namely
\begin{equation*}
|\overline{u}(r)-u_{\infty}|\leq C\|\tilde{F}\|_{C^{2k,2\alpha}_{X,\beta}}\left(r^{-\min\{\beta,\gamma+\delta\}}\right),\qquad r\geq r_0.
\end{equation*}
On the other hand, \eqref{data-sharp-est} can then be inserted into the proof of Claim \ref{cla-c-0-orth-mean-val}
to establish an improved $L^2(D)$-decay on $\tilde{u}-\overline{u}(r)$. Indeed, from inequality \eqref{ineq-ene-no-young} we deduce that for $r\geq r_0$,
\begin{equation*}
\begin{split}
\Delta_{C,\,X}\left(\|\tilde{u}-\overline{u}(r)\|^{2}_{L^{2}(D)}\right)&\geq 2\lambda^{D}\|\tilde{u}-\overline{u}(r)\|^{2}_{L^{2}(D)}-C\|\tilde{u}-\overline{u}(r)\|_{L^{2}(D)}\|\tilde{F}\|_{C^{2k,2\alpha}_{X,\beta}}r^{-\min\{\beta,\gamma+\delta\}}\\
&\geq 2(\lambda^D-\varepsilon)\|\tilde{u}-\overline{u}(r)\|^{2}_{L^{2}(D)}-C_{\varepsilon}\|\tilde{F}\|_{C^{2k,2\alpha}_{X,\beta}}^2r^{-2\min\{\beta,\gamma+\delta\}}
\end{split}
\end{equation*}
for any $\varepsilon\in(0,\lambda^D)$. Using a barrier function of the form $r^{-2\delta'}$ with $0<\delta'\leq\min\{\beta,\gamma+\delta\}<\lambda^D$
and by choosing $\varepsilon>0$ carefully, one arrives at an improved $L^2(D)$-decay of the form above, specifically
\begin{equation*}
\|\tilde{u}-\overline{u}(r)\|_{L^2(D)}\leq C\|\tilde{F}\|_{C^{2k,2\alpha}_{X,\beta}}r^{-\delta'},\qquad r\geq r_0.
\end{equation*}
The proof of Claim \ref{claim-first-rough-dec-pt} can now be adapted to give a
corresponding improved pointwise decay. By applying this reasoning a finite number of times, one arrives at the desired sharp decay on $\tilde{u}-u_{\infty}$.
\end{proof}

Theorem \ref{iso-sch-Laplacian-pol} now follows by combining Claim \ref{claim-a-priori-rough-bd-hih-der} (after multiplying by the weight $r^{\beta}$) and Claim \ref{sharp-pt-dec-u}.
\end{proof}

\subsection{Small perturbations along the continuity path}\label{invert-poly}

In this section we show, using the implicit function theorem, that the
invertibility of the drift Laplacian given by Theorem \ref{iso-sch-Laplacian-pol}
allows for small perturbations in polynomially weighted function spaces
of solutions to the complex Monge-Amp\`ere equation that we wish to solve.
This forms the openness part of the continuity method as will be explained later in Section \ref{continuitie}.

In notation reminiscent of that of \cite[Chapter $5$]{Tian-Can-Met-Boo}, we consider
the space $\left(\mathcal{C}^{2,\,2\alpha}_{X,\,\beta}(M)\right)_{\tilde{\omega},0}$ of functions $F\in \mathcal{C}^{2,\,2\alpha}_{X,\,\beta}(M)$ with
\begin{equation*}
\int_M\left(e^{F}-1\right)\,e^{-\tilde{f}}\tilde{\omega}^n=0.
\end{equation*}
This function space is a hypersurface of the Banach space $\mathcal{C}^{2,\,2\alpha}_{X,\,\beta}(M)$. Notice that the tangent space at a function $F_0$ is the set of functions $u\in \mathcal{C}^{2,\,2\alpha}_{X,\,\beta}(M)$ with
\begin{equation*}
\int_Mu\,e^{F_0-\tilde{f}}\tilde{\omega}^n=0.
\end{equation*}
We have:
\begin{theorem}\label{Imp-Def-Kah-Ste}
Let $F_0\in\left(\mathcal{C}^{2,\,2\alpha}_{X,\,\beta}(M)\right)_{\tilde{\omega},0}\cap \mathcal{C}^{\infty}_{X,\,\beta}(M)$ for some $\beta\in(0,\lambda^D)$ and let $\psi_0\in\mathcal{M}^{\infty}_{X,\,\beta}(M)$ be a solution of the complex Monge-Amp\`ere equation
\begin{equation*}
\log\left(\frac{\tilde{\omega}^n_{\psi_0}}{\tilde{\omega}^n}\right)-\frac{X}{2}\cdot\psi_0=F_0.
\end{equation*}
Then for any $\alpha\in\left(0,\,\frac{1}{2}\right)$, there exists a neighbourhood $U_{F_0}\subset\left(C^{2,\,2\alpha}_{X,\,\beta}(M)\right)_{\tilde{\omega},0}$ of $F_0$ such that for all $F\in U_{F_0}$, there exists a unique function $\psi\in\mathcal{M}^{4,\,2\alpha}_{X,\,\beta}(M)$ such that
\begin{equation}
\log\left(\frac{\tilde{\omega}^n_{\psi}}{\tilde{\omega}^n}\right)-\frac{X}{2}\cdot\psi=F.\label{MA-neigh-small-per-pol}
\end{equation}
Moreover, if $F\in U_{F_0}$ lies in $\mathcal{C}^{\infty}_{X,\,\beta}(M)$ then the unique solution $\psi\in\mathcal{M}^{4,\,2\alpha}_{X,\,\beta}(M)$ to \eqref{MA-neigh-small-per-pol} lies in $\mathcal{M}^{\infty}_{X,\,\beta}(M).$
\end{theorem}

\begin{remark}
{Consideration of only finite regularity of the difference $\omega-\tilde{\omega}$
(which lowers the assumptions on the regularity of the coefficients of the drift Laplacian $\Delta_{\tilde{g},\,X}$)
and of the data $(\psi_0,\,F_0)$ would lead to a more refined version of Theorem \ref{Imp-Def-Kah-Ste}.}
\end{remark}

\begin{proof}[Proof of Theorem \ref{Imp-Def-Kah-Ste}]
In order to apply the implicit function theorem for Banach spaces, we must reformulate
the statement of Theorem \ref{Imp-Def-Kah-Ste} in terms of
the map $MA_{\tilde{\omega}}$ introduced formally at the beginning of Section \ref{linear}. To this end, consider the mapping
\begin{equation*}
\begin{split}
MA_{\tilde{\omega}}:\psi\in &\,\mathcal{M}^{4,\,2\alpha}_{X,\,\beta}(M)\\
&\mapsto \log\left(\frac{\tilde{\omega}_{\psi}^n}{\tilde{\omega}^n}\right)-\frac{X}{2}\cdot\psi\in\left(\mathcal{C}^{2,\,2\alpha}_{X,\,\beta}(M)\right)_{\tilde{\omega},0},\qquad \alpha\in\left(0,\,\frac{1}{2}\right).
\end{split}
\end{equation*}
Notice that the function spaces above can be defined by either using the metric $\tilde{g}$ or $\tilde{g}_{t\psi_0}$ for any $t\in[0,\,1]$.
To see that $MA_{\tilde{\omega}}$ is well-defined, apply the Taylor expansion \eqref{equ:taylor-exp} to the background metric $\tilde{\omega}$
to obtain
\begin{equation}\label{reform-MA-op}
\begin{split}
MA_{\tilde{\omega}}(\psi)&=\log\left(\frac{\tilde{\omega}_{\psi}^n}{\tilde{\omega}^n}\right)-\frac{X}{2}\cdot\psi\\
&=\Delta_{\tilde{\omega}}\psi-\frac{X}{2}\cdot\psi-\int_0^1\int_0^{u}\arrowvert \partial\bar{\partial}\psi\arrowvert^2_{\tilde{g}_{t\psi}}\,dt\,du.
\end{split}
\end{equation}
Then by the very definition of $\mathcal{D}^{4,\,2\alpha}_{X,\,\beta}(M)$,
the first two terms of the last line of \eqref{reform-MA-op} lie in $\mathcal{C}^{2,\,2\alpha}_{X,\,\beta}(M)$.

Now, if $S$ and $T$ are tensors in $C^{2k,\,2\alpha}_{X,\,\gamma_1}(M)$ and $C^{2k,\,2\alpha}_{X,\gamma_2}(M)$ respectively with $\gamma_i\geq 0$, $i=1,2$,
then observe that $S\ast T$ lies in $C^{2k,\,2\alpha}_{X,\,\gamma_1+\gamma_2}(M)$,
where $\ast$ denotes any linear combination of contractions of tensors with respect to the metric $\tilde{g}$. Moreover,
\begin{equation}\label{mult-inequ-holder}
\|S\ast T\|_{C^{2k,\,2\alpha}_{X,\,\gamma_1+\gamma_2}}\leq C(k,\,\alpha)\|S\|_{C^{2k,\,2\alpha}_{X,\,\gamma_1}}\cdot \|T\|_{C^{2k,\,2\alpha}_{X,\,\gamma_2}}.
\end{equation}
Next notice that $$\arrowvert i\partial \bar{\partial}\psi\arrowvert^2_{\tilde{g}_{t\psi}}=\tilde{g}_{t\psi}^{-1}\ast \tilde{g}_{t\psi}^{-1}
\ast (\nabla^{\tilde{g}})^{\,2}\psi\ast(\nabla^{\tilde{g}})^{\,2}\psi$$ and that $$\tilde{g}_{t\psi}^{-1}
-\tilde{g}^{-1}\in C^{2,\,2\alpha}_{X,\,\beta}(M).$$
Thus, applying \eqref{mult-inequ-holder} twice to $S=T=(\nabla^{\tilde{g}})^{2}\psi$ and to the inverse $\tilde{g}_{t\psi}^{-1}$ with weights
$\gamma_1=\gamma_{2}=\beta$ and $k=1$, one finds that $\arrowvert i\partial\bar{\partial}\psi\arrowvert^2_{\tilde{g}_{t\psi}}\in C^{2,\,2\alpha}_{X,\,2\beta}(M)\subset C^{2,\,2\alpha}_{X,\,\beta}(M)$ for each
$t\in[0,\,1]$ and that
\begin{equation*}
\left\|\int_0^1\int_0^{u}\arrowvert i\partial\bar{\partial}\psi\arrowvert^2_{\tilde{g}_{t\psi}}\,dt\,du
\right\|_{C^{2,\,2\alpha}_{X,\,\beta}}\leq C\left(k,\alpha,\tilde{g}\right)\|\psi\|_{\mathcal{D}^{4,\,2\alpha}_{X,\,\beta}},\\
\end{equation*}
as long as $\|\psi\|_{\mathcal{D}^{4,\,2\alpha}_{X,\,\beta}}\leq 1$. Finally, the $JX$-invariance of the right-hand side of \eqref{reform-MA-op} is clear and Lemma \ref{lemma-preserved-int}(i) ensures that the function $$\exp MA_{\tilde{\omega}}(\psi)-1$$ has zero mean value with respect to the weighted measure $e^{-\tilde{f}}\tilde{\omega}^n$. Indeed, Lemma \ref{lemma-preserved-int}(i) applied to the linear path $\tilde{\omega}_{\tau}:=\tilde{\omega}+i\partial\overline{\partial}(\tau\psi)$ for $\tau\in[0,1]$ gives us that
\begin{equation*}
\int_M\left(\exp MA_{\tilde{\omega}}(\psi)-1\right)\,e^{-\tilde{f}}\tilde{\omega}^n=\int_Me^{-\tilde{f}_{\psi}}\tilde{\omega}_{\psi}^n-\int_Me^{-\tilde{f}}\tilde{\omega}^n=0.
\end{equation*}

By \eqref{equ:sec-der}, we have that
\begin{equation*}
\begin{split}
D_{\psi_0}MA_{\tilde{\omega}}:\psi\in &\,\mathcal{M}^{4,\,2\alpha}_{X,\,\beta}(M)\cap\left\{\int_{M}u\,e^{-\tilde{f}_{\psi_0}}\tilde{\omega}_{\psi_0}^{n}=0\right\}\\
&\mapsto \Delta_{\tilde{\omega}_{\psi_0}}\psi-\frac{X}{2}\cdot\psi\in T_{F_0}\left(\mathcal{C}^{2,\,2\alpha}_{X,\,\beta}(M)\right)_{\tilde{\omega},0},
\end{split}
\end{equation*}
where the tangent space of $\left(\mathcal{C}^{2,\,2\alpha}_{X,\,\beta}(M)\right)_{\tilde{\omega},0}$ at $F_0$ is equal to the set of functions $u\in \mathcal{C}^{2,\,2\alpha}_{X,\,\beta}(M)$ with $0$ mean value with respect to the weighted measure $e^{-\tilde{f}_{\psi_0}}\tilde{\omega}_{\psi_0}^n$.
Therefore, after applying Theorem \ref{iso-sch-Laplacian-pol} to the background metric $\tilde{\omega}_{\psi_0}$ in place of $\tilde{\omega}$, we conclude that
$D_{\psi_0}MA_{\tilde{\omega}}$ is an isomorphism of Banach spaces. The result now follows by applying the implicit function theorem
to the map $MA_{\tilde{\omega}}$ in a neighbourhood of $\psi_0\in\mathcal{M}^{4,\,2\alpha}_{X,\,\beta}(M)\cap\left\{\int_{M}u\,e^{-\tilde{f}_{\psi_0}}\tilde{\omega}_{\psi_0}^{n}=0\right\}$.\\

The proof of the regularity at infinity of the solution $\psi$ in case the data $F\in \mathcal{C}^{\infty}_{X,\beta}(M)$ follows by a standard bootstrapping and will therefore be omitted; see Propositions \ref{prop-C4-est} and \ref{high-order-est-prop} for the non-linear setting.
\end{proof}

\section{Proof of Theorem \ref{mainthm}(v): A priori estimates}\label{sec-a-priori-est}

\subsection{The continuity path}\label{continuitie}

Recall the setup and notation of Theorem \ref{mainthm}: $J$ denotes the complex structure on $M$, $z$ the holomorphic coordinate
on the $\mathbb{C}$-component of $\widehat{M}$, and we write $r=|z|^{\lambda}$, treating both $r$ and $z$ as functions on $M$ via $\nu$.
It is clear then that $X=r\partial_{r}$ on $M\setminus K$.

Recall from \eqref{cmaa} that the complex Monge-Amp\`ere equation we wish to solve is
\begin{equation}\label{ast-0}
\left\{
\begin{array}{rl}
(\omega+i\partial\bar{\partial}\psi)^{n}=e^{F+\frac{X}{2}\cdot\psi}\omega^{n},&\qquad\psi\in C^{\infty}(M),
\qquad\mathcal{L}_{JX}\psi=0,\qquad\omega+i\partial\bar{\partial}\psi>0,\\
\int_{M}e^{F-f}\omega^{n}=\int_{M}e^{-f}\omega^{n}, &
\end{array} \right.\tag{$\ast_{0}$}
\end{equation}
where $F:M\to\mathbb{R}$ is a $JX$-invariant smooth function equal to a constant $c_{0}$ outside a compact subset $V$ of $M$ and
$f:M\to\mathbb{R}$ is the Hamiltonian potential of $X$ with respect to $\omega$, i.e., $-\omega\lrcorner JX=df$, normalised so that
\begin{equation*}
\Delta_{\omega}f-f+\frac{X}{2}\cdot f=0
\end{equation*}
outside a compact set. Define $F_{s}:=\log(1+s(e^{F}-1))$. In this section, we prove Theorem \ref{mainthm}(v) by providing a solution to \eqref{ast-0} by
implementing the continuity path
\begin{equation}\label{star-s}
\left\{
\begin{array}{rl}
(\omega+i\partial\bar{\partial}\psi_{s})^{n}=e^{F_{s}+\frac{X}{2}\cdot\psi_{s}}\omega^{n},&\qquad\psi_{s}\in\mathcal{M}^{\infty}_{X,\,\beta}(M),
\qquad\mathcal{L}_{JX}\psi_{s}=0,\qquad s\in[0,\,1],\\
\int_{M}e^{F-f}\omega^{n}=\int_{M}e^{-f}\omega^{n}, &\qquad\int_{M}\psi_{s}\,e^{-f}\omega^{n}=0.
\end{array} \right.\tag{$\star_{s}$}
\end{equation}

When $s=0$, $(\star_{0})$ admits the trivial solution, namely $\psi_{0}\equiv0$. When $s=1$, $(\star_{1})$
corresponds to \eqref{ast-0}, that is, the equation that we wish to solve. Via the a priori estimates to follow, we will
show that the set $s\in[0,\,1]$ for which \eqref{star-s} has a solution is closed. As we have just seen, this set is
non-empty. Openness of this set follows from the isomorphism properties of the drift Laplacian given
by Theorem \ref{Imp-Def-Kah-Ste}. Connectedness of $[0,\,1]$ then implies that \eqref{star-s} has a solution for $s=1$, resulting
in the desired solution of \eqref{ast-0}.

\subsection{The continuity path re-parametrised}

To obtain certain localisation results and in turn, a priori estimates for \eqref{star-s}, we need to consider a reformulation of \eqref{star-s} in the following way.
Identify $(M\setminus K,\,\omega)$ and $(\widehat{M}\setminus\widehat{K},\,\widehat{\omega})$ using $\nu$, where $K\subset M$, $\widehat{K}\subset\widehat{M}$ are compact,
and define $F_{s}:=\log(1+s(e^{F}-1))$. Then there exists a compact subset $K\subset V\subset M$ such that
for all $s\in[0,\,1]$, $F_{s}$ is equal to a constant $c_{s}$ on $M\setminus V$. Explicitly, $c_{s}=\log(1+s(e^{c_{0}}-1))$.
Note that $c_{s}$ varies continuously as a function of $s$ and that \eqref{star-s} {takes the form}
$$(\omega+i\partial\bar{\partial}\psi_{s})^{n}=e^{F_{s}+\frac{X}{2}\cdot\psi_{s}}\omega^{n}.$$

Let $\eta_{s}:=-2c_{s}\log(r)$, a real-valued function defined on $M\setminus K$. Then, with $g$ denoting the K\"ahler metric associated to $\omega$, it is clear that
$$\|(\log(r))^{-1}\cdot\eta_{s}\|_{C^{0}(M\setminus K)}
+\|d\eta_{s}\|_{C^{0}(M\setminus K,\,g)}+\|r\cdot i\partial\bar{\partial}\eta_{s}\|_{C^{0}(M\setminus K,\,g)}\leq2|c_{s}|\left(1+\sup_{M\setminus K}r^{-1}\right)\leq C(K),$$
and so Lemma \ref{glue} infers the existence of a bump function $\chi:M\to\mathbb{R}$ supported on $M\setminus V$ and a compact subset $W\supset V$, both independent of $s$, such that $\chi=1$ on $M\setminus W$ and such that for all $s\in[0,\,1]$,
$\omega_{s}:=\omega+i\partial\bar{\partial}\left(\chi\cdot\eta_{s}\right)>0$ on $M$.
Define $\Phi_{s}:=\chi\cdot\eta_{s}$. Then $\omega_{s}=\omega+i\partial\bar{\partial}\Phi_{s}$
and since $\Phi_{s}=-2c_{s}\log r$ on $M\setminus W$, that is, a pluriharmonic function, $\omega_{s}$ is isometric to $\omega$
on this set. Furthermore, we find that
\begin{equation*}
\begin{split}
\log\left(\frac{(\omega_{s}+i\partial\bar{\partial}(\psi_{s}-\Phi_{s}))^{n}}{\omega_{s}^{n}}\right)&-\frac{X}{2}\cdot(\psi_{s}-\Phi_{s})=
\log\left(\frac{(\omega+i\partial\bar{\partial}\psi_{s})^{n}}{(\omega+i\partial\bar{\partial}\Phi_{s})^{n}}\right)-\frac{X}{2}\cdot(\psi_{s}-\Phi_{s})\\
&=\log\left(\frac{(\omega+i\partial\bar{\partial}\psi_{s})^{n}}{\omega^{n}}\right)-\frac{X}{2}\cdot\psi_{s}
-\log\left(\frac{(\omega+i\partial\bar{\partial}\Phi_{s})^{n}}{\omega^{n}}\right)+\frac{X}{2}\cdot\Phi_{s}\\
&=F_s-\left(\log\left(\frac{(\omega+i\partial\bar{\partial}\Phi_{s})^{n}}{\omega^{n}}\right)-\frac{X}{2}\cdot\Phi_{s}\right)=:G_{s},
\end{split}
\end{equation*}
with $G_{s}$ vanishing on $M\setminus W$. Set $\vartheta_{s}:=\psi_{s}-\Phi_{s}$. Then $\vartheta_{s}\in \mathbb{R}\oplus C^{\infty}_{X,\,\beta}(M)$ and
we can rewrite \eqref{star-s} in terms of $\vartheta_{s}$ as
\begin{equation}\label{starstar-s}
\log\left(\frac{(\omega_{s}+i\partial\bar{\partial}\vartheta_{s})^{n}}{\omega_{s}^{n}}\right)-\frac{X}{2}\cdot\vartheta_{s}=G_{s},
\quad\vartheta_{s}\in\mathbb{R}\,\oplus\, C^{\infty}_{X,\,\beta}(M),\quad\mathcal{L}_{JX}\vartheta_{s}=0,\quad\omega_{s}+i\partial\bar{\partial}\vartheta_{s}>0,\, s\in[0,\,1],
\tag{$\star\star_{s}$}
\end{equation}
with the support of $G_{s}$ contained in $W$ and $\omega_{s}=\omega$ on $M\setminus W$. We derive a priori estimates for
\eqref{starstar-s}, the advantage over \eqref{star-s} being that it allows for a localisation of the infimum and supremum of $|\vartheta_{s}|$,
essentially because the unbounded log term has been absorbed into the background metric $\omega_{s}$ in \eqref{starstar-s}. As we have control on $\Phi_{s}$, the a priori estimates we derive
for $\vartheta_{s}$ will translate into the desired a priori estimates for $\psi_{s}$, thereby allowing us to complete the closedness part of the continuity method for
\eqref{star-s}.

Define $\sigma_{s}:=\omega_{s}+i\partial\bar{\partial}\vartheta_{s}$. Then in terms of the Ricci forms $\rho_{\sigma_{s}}$ and $\rho_{\omega_{s}}$ of
$\sigma_{s}$ and $\omega_{s}$ respectively, \eqref{starstar-s} yields
\begin{equation}\label{wtf}
\rho_{\sigma_{s}}+\frac{1}{2}\mathcal{L}_{X}\sigma_{s}=\rho_{\omega_{s}}+\frac{1}{2}\mathcal{L}_{X}\omega_{s}-i\partial\bar{\partial}G_{s}.
\end{equation}
We will write $h_{s}$ for the K\"ahler metric associated to $\sigma_{s}$.

We will need the following lemma regarding the Hamiltonian potential $f_{\omega_{s}}$ of $X$ with respect to $\omega_{s}$.
\begin{lemma}\label{normal-fss}
Let $f_{\omega_{s}}:=f+\frac{X}{2}\cdot\Phi_{s}$. Then $-\omega_{s}\lrcorner JX=df_{\omega_{s}}$ and there
exists a compact subset $U\subset M$ containing $W$ such that for all $s\in[0,\,1]$,
there exists $H_{s}\in C^{\infty}(M)$ varying smoothly in $s$ and equal to $-c_{s}$ on $M\setminus U$
so that
\begin{equation}\label{normal-fs}
\Delta_{\omega_{s}}f_{\omega_{s}}-\frac{X}{2}\cdot f_{\omega_{s}}+f_{\omega_{s}}=H_{s}.
\end{equation}
\end{lemma}

\begin{proof}
The first assertion is clear. Regarding the normalisation condition \eqref{normal-fs}, a computation shows that for the Ricci forms $\rho_{\omega}$ and $\rho_{\omega_{s}}$ of $\omega$ and $\omega_{s}$ respectively,
\begin{equation*}
\begin{split}
\rho_{\omega_{s}}+\frac{1}{2}\mathcal{L}_{X}\omega_{s}-\omega_{s}&=\rho_{\omega}+\frac{1}{2}\mathcal{L}_{X}\omega-\omega-i\partial\bar{\partial}
\left(\log\left(\frac{\omega_{s}^{n}}{\omega^{n}}\right)-\frac{X}{2}\cdot\Phi_{s}+ \Phi_{s}\right)\\
&=i\partial\bar{\partial}(F_{2}+G_{s}-F_s-\Phi_{s}),
\end{split}
\end{equation*}
where we have used \eqref{credit}. Write $Q_{s}:=F_{2}+G_{s}-F_s-\Phi_{s}$. Then $Q_{s}$ is $JX$-invariant and
it is easy to see that $Q_{s}$ is equal to $2c_{s}\log(r)-c_{s}$ outside a compact subset $U\supseteq W$ of $M$ independent of $s$.
Contracting the identity
\begin{equation*}
\rho_{\omega_{s}}+\frac{1}{2}\mathcal{L}_{X}\omega_{s}-\omega_{s}=i\partial\bar{\partial}Q_{s}
\end{equation*}
with $X^{1,\,0}:=\frac{1}{2}(X-iJX)$ and arguing as in Lemma \ref{warm} using the $JX$-invariance of the functions involved, we find that
$$\Delta_{\omega_{s}}f_{\omega_{s}}-\frac{X}{2}\cdot f_{\omega_{s}}+f_{\omega_{s}}+\frac{X}{2}\cdot Q_{s}$$
is constant on $M$. But since on $M\setminus W$, $\omega_{s}=\omega$, $f_{\omega_{s}}=f-c_{s}$,
and $\frac{X}{2}\cdot Q_{s}=c_{s}$, this constant must be zero. Hence the result follows with $H_{s}:=-\frac{X}{2}\cdot Q_{s}$.
\end{proof}

This allows for a normalisation for the Hamiltonian potential $f_{\sigma_{s}}:=f_{\omega_{s}}+\frac{X}{2}\cdot\vartheta_{s}$
of $X$ with respect to $\sigma_{s}$.

\begin{lemma}\label{lemma-tr-star-star}
Let $f_{\sigma_{s}}:=f_{\omega_{s}}+\frac{X}{2}\cdot\vartheta_{s}$ and let $U$ be as in Lemma \ref{normal-fss}. Then $-\sigma_{s}\lrcorner JX=df_{\sigma_{s}}$ and
for all $s\in[0,\,1]$, there exists a compactly supported function $P_{s}\in C^{\infty}(M)$ varying smoothly in $s$ with
$\operatorname{supp}P_{s}\subseteq U$ such that
$$\Delta_{\sigma_{s}}f_{\sigma_{s}}-\frac{X}{2}\cdot f_{\sigma_{s}}=-f+P_{s}.$$
\end{lemma}

\begin{proof}
Again, the first assertion is clear. As for \eqref{normal-fs}, we have that
\begin{equation*}
\begin{split}
\frac{X}{2}\cdot\log\left(\frac{\sigma_{s}^{n}}{\omega_{s}^{n}}\right)&=\frac{1}{2}\tr_{\sigma_{s}}
\mathcal{L}_{X}\sigma_{s}-\frac{1}{2}\tr_{\omega_{s}}\mathcal{L}_{X}\omega_{s}\\
&=\tr_{\sigma_{s}}(i\partial\bar{\partial}f_{\sigma_{s}})-\tr_{\omega}(i\partial\bar{\partial}f_{\omega_{s}})\\
&=\Delta_{\sigma_{s}}f_{\sigma_{s}}-\Delta_{\omega_{s}}f_{\omega_{s}}.
\end{split}
\end{equation*}
Thus, contracting both sides of \eqref{starstar-s} with $\frac{X}{2}$, we obtain
$$\Delta_{\sigma_{s}}f_{\sigma_{s}}-\Delta_{\omega_{s}}f_{\omega_{s}}=\frac{X}{2}\cdot G_{s}+\frac{X}{2}\cdot\left(f_{\omega_{s}}+\frac{X}{2}\cdot\vartheta_{s}\right)-\frac{X}{2}\cdot f_{\omega_{s}},$$
i.e.,
$$\Delta_{\sigma_{s}}f_{\sigma_{s}}-\frac{X}{2}\cdot f_{\sigma_{s}}=\Delta_{\omega_{s}}f_{\omega_{s}}-\frac{X}{2}\cdot f_{\omega_{s}}+\frac{X}{2}\cdot G_{s}.$$
Hence we derive from \eqref{normal-fs} that
$$\Delta_{\sigma_{s}}f_{\sigma_{s}}-\frac{X}{2}\cdot f_{\sigma_{s}}=H_{s}+\frac{X}{2}\cdot G_{s}-f_{\omega_{s}}.$$
With $P_{s}:=H_{s}+\frac{X}{2}\cdot G_{s}-\frac{X}{2}\cdot\Phi_{s}$, the result is now clear.
\end{proof}

\subsection{Summary of notation}

For clarity, in this section we provide a summary of our notation regarding the various K\"ahler forms in play.
\begin{itemize}
 \item $F$ is the data in \eqref{ast-0} equal to a constant $c_{0}$ outside a compact set.
  \item $\omega$ is the background K\"ahler form given in \eqref{ast-0} isometric to $\omega_{C}+\omega_{D}$ outside a fixed compact subset $K\subset M$.
  \item $g$ is the K\"ahler metric associated to $\omega$.
  \item $f$ is the Hamiltonian potential of $JX$ with respect to $\omega$ given in Theorem \ref{mainthm}(iii). It is equal to $\frac{|z|^{2\lambda}}{2}-1$ outside the compact subset $K\subset M$ and normalised so that
  $$\Delta_{\omega}f-f+\frac{X}{2}\cdot f=0$$ outside a compact set.
 \item $c_{s}:=\log(1+s(e^{c_{0}}-1))$.
  \item $F_{s}$ is the data in \eqref{star-s} equal to $c_{s}$ outside a fixed compact subset $V \subset M$ with $V\supset K$.
\item $\psi_{s}$ is the solution to the original continuity path \eqref{star-s}.
\item $\Phi_{s}=-2\chi\cdot c_{s}\log r$, where $0\leq\chi\leq1$ is a bump function identically equal to $1$ outside a fixed compact subset $W \supset V\supset K$ of $M$. In particular, notice that
$\Phi_{s}=-c_{s}\log(2(f+1))$ on $M\setminus W$.
   \item $\omega_{s}:=\omega+i\partial\bar{\partial}\Phi_{s}$ is the $1$-parameter family of background metrics isometric to $\omega$ outside a compact set independent of $s$
  appearing in \eqref{starstar-s}.
\item $g_{s}$ is the K\"ahler metric associated to $\omega_{s}$.
    \item $f_{s}:=f+\frac{X}{2}\cdot\Phi_{s}$ is the Hamiltonian potential of $JX$ with respect to $\omega_{s}$.
\item $\vartheta_{s}=\psi_s - \Phi_s$ is the solution of the re-parametrised continuity path \eqref{starstar-s}.
    \item $\sigma_{s}:=\omega_{s}+i\partial\bar{\partial}\vartheta_{s}$ is the associated K\"ahler metric.
        \item $f_{\sigma_{s}}$ is the Hamiltonian potential of $JX$ with respect to $\sigma_{s}$. It is normalised by the equation
    $$\Delta_{\sigma_{s}}f_{\sigma_{s}}-\frac{X}{2}\cdot f_{\sigma_{s}}=-f+P_{s},$$
where $P_{s}$ is compactly supported.
\item $h_{s}$ is the K\"ahler metric associated to $\sigma_{s}$.
\end{itemize}

\subsection{A priori lower bound on the radial derivative}\label{sec-low-bd-rad-der}

The fact that the data $G_{s}$ of \eqref{starstar-s} is compactly supported allows us to localise the extrema of $X\cdot\vartheta_{s}$ using the maximum principle. This leads to a uniform lower bound on $X\cdot\vartheta_{s}$ and in particular on $X\cdot\psi_{s}$.

\begin{lemma}[Localising the supremum and infimum of the radial derivative]\label{lemma-loc-crit-pts-rad-der}
Let $(\vartheta_s)_{0\,\leq\, s\,\leq\, 1}$  be a path of solutions in $\mathbb{R}\oplus C^{\infty}_{X,\,\beta}(M)$ to \eqref{starstar-s}.
{Then $\sup_M X\cdot\vartheta_{s}=\max\{0\,,\,\max_{W}X\cdot\vartheta_{s}\}$
and $\inf_M X\cdot\vartheta_{s}=\min\{0\,,\,\min_{W}X\cdot\vartheta_{s}\}$.}
\end{lemma}

\begin{proof}
First, using $\nu$ to identify $(M,\,\omega)$ and $(\widehat{M},\,\widehat{\omega})$ on $M\setminus W$, notice that
\begin{equation*}
\begin{split}
\frac{X}{2}\cdot\left(\log\left(\frac{\sigma_{s}^{n}}{\omega_{s}^{n}}\right)\right)&=\operatorname{tr}_{\sigma_{s}}\mathcal{L}_{\frac{X}{2}}\sigma_{s}-\operatorname{tr}_{\omega_{s}}\mathcal{L}_{\frac{X}{2}}\omega_{s}\\
&=\operatorname{tr}_{\sigma_{s}}\mathcal{L}_{\frac{X}{2}}(\omega_{s}+i\partial\bar{\partial}\vartheta_{s})-\operatorname{tr}_{\omega}\mathcal{L}_{\frac{X}{2}}\omega\\
&=\operatorname{tr}_{\sigma_{s}}\omega_{C}+\frac{1}{2}\Delta_{\sigma_{s}}(X\cdot\vartheta_{s})-\operatorname{tr}_{\omega}\omega_{C}\\
&=\operatorname{tr}_{\sigma_{s}}\omega_{C}+\frac{1}{2}\Delta_{\sigma_{s}}(X\cdot\vartheta_{s})-1.\\
\end{split}
\end{equation*}
Thus, upon differentiating \eqref{starstar-s} along $X$, we obtain on $M\setminus W$ the equation
\begin{equation}\label{banana1}
\Delta_{\sigma_{s},\,X}\left(\frac{X\cdot\vartheta_{s}}{2}\right):=
\Delta_{\sigma_{s}}\left(\frac{X\cdot\vartheta_{s}}{2}\right)-\frac{X}{2}\cdot\left(\frac{X\cdot\vartheta_{s}}{2}\right)
=1-\operatorname{tr}_{\sigma_{s}}\omega_{C}.
\end{equation}
Now on $M\setminus V$, we have that
\begin{equation*}
\begin{split}
\operatorname{tr}_{\sigma_{s}}\omega_{C}&=\frac{n\sigma_{s}^{n-1}\wedge\omega_{C}}{\sigma_{s}^{n}}\\
&=\frac{ne^{-\frac{X\cdot\vartheta_{s}}{2}}\sigma_{s}^{n-1}\wedge\omega_{C}}{\omega^{n}},
\end{split}
\end{equation*}
hence
\begin{equation}\label{banana2}
\begin{split}
1-\operatorname{tr}_{\sigma_{s}}\omega_{C}&=e^{-\frac{X\cdot\vartheta_{s}}{2}}\left(e^{\frac{X\cdot\vartheta_{s}}{2}}-\frac{n\sigma_{s}^{n-1}\wedge\omega_{C}}{\omega^{n}}\right)\\
&=e^{-\frac{X\cdot\vartheta_{s}}{2}}\left(\frac{\sigma_{s}^{n}-n\sigma_{s}^{n-1}\wedge\omega_{C}}{\omega^{n}}\right).\\
\end{split}
\end{equation}
For $k=1,\ldots,n$, we have for dimensional reasons that
$$\omega^{k}=(\omega_{D}+\omega_{C})^{k}=\omega_{D}^{k}+k\omega_{D}^{k-1}\wedge\omega_{C}.$$
Thus,
\begin{equation*}
\begin{split}
\sigma^{n}_{s}&=(\omega+i\partial\bar{\partial}\vartheta_{s})^{n}\\
&=\sum_{k\,=\,0}^{n}{n \choose k}\omega^{k}\wedge(i\partial\bar{\partial}\vartheta_{s})^{n-k}\\
&=(i\partial\bar{\partial}\vartheta_{s})^{n}+\sum_{k\,=\,1}^{n}{n \choose k}\omega^{k}\wedge(i\partial\bar{\partial}\vartheta_{s})^{n-k}\\
&=(i\partial\bar{\partial}\vartheta_{s})^{n}+\sum_{k\,=\,1}^{n}{n \choose k}(\omega_{D}^{k}+k\omega_{D}^{k-1}\wedge\omega_{C})\wedge(i\partial\bar{\partial}\vartheta_{s})^{n-k}\\
&=(i\partial\bar{\partial}\vartheta_{s})^{n}+\sum_{k\,=\,1}^{n}{n \choose k}\omega_{D}^{k}\wedge(i\partial\bar{\partial}\vartheta_{s})^{n-k}
+\sum_{k\,=\,1}^{n}k{n \choose k}\omega_{D}^{k-1}\wedge(i\partial\bar{\partial}\vartheta_{s})^{n-k}\wedge\omega_{C},\\
&=\sum_{k\,=\,0}^{n}{n \choose k}\omega_{D}^{k}\wedge(i\partial\bar{\partial}\vartheta_{s})^{n-k}
+\sum_{k\,=\,1}^{n}k{n \choose k}\omega_{D}^{k-1}\wedge(i\partial\bar{\partial}\vartheta_{s})^{n-k}\wedge\omega_{C}\\
\end{split}
\end{equation*}
and
\begin{equation*}
\begin{split}
n\sigma^{n-1}_{s}\wedge\omega_{C}&=n\sum_{j\,=\,0}^{n-1}{n-1 \choose j}\omega^{j}\wedge(i\partial\bar{\partial}\vartheta_{s})^{n-1-j}\wedge\omega_{C}\\
&=ni\partial\bar{\partial}\vartheta_{s}^{n-1}\wedge\omega_{C}+n\sum_{j\,=\,1}^{n-1}{n-1 \choose j}\omega^{j}\wedge(i\partial\bar{\partial}\vartheta_{s})^{n-1-j}\wedge\omega_{C}\\
&=ni\partial\bar{\partial}\vartheta_{s}^{n-1}\wedge\omega_{C}+n\sum_{j\,=\,1}^{n-1}{n-1 \choose j}(\omega_{D}^{j}+j\omega_{D}^{j-1}\wedge\omega_{C})\wedge(i\partial\bar{\partial}\vartheta_{s})^{n-1-j}\wedge\omega_{C}\\
&=ni\partial\bar{\partial}\vartheta_{s}^{n-1}\wedge\omega_{C}+n\sum_{j\,=\,1}^{n-1}{n-1 \choose j}\omega_{D}^{j}\wedge(i\partial\bar{\partial}\vartheta_{s})^{n-1-j}\wedge\omega_{C}\\
&=ni\partial\bar{\partial}\vartheta_{s}^{n-1}\wedge\omega_{C}+n\sum_{k\,=\,2}^{n}{n-1 \choose k-1}\omega_{D}^{k-1}\wedge(i\partial\bar{\partial}\vartheta_{s})^{n-k}\wedge\omega_{C}\\
&=n\sum_{k\,=\,1}^{n}{n-1 \choose k-1}\omega_{D}^{k-1}\wedge(i\partial\bar{\partial}\vartheta_{s})^{n-k}\wedge\omega_{C}.\\
\end{split}
\end{equation*}
Consequently,
\begin{equation*}
\begin{split}
\sigma^{n}_{s}-n\sigma^{n-1}_{s}\wedge\omega_{C}&=(i\partial\bar{\partial}\vartheta_{s})^{n}+\sum_{k\,=\,1}^{n}{n \choose k}\omega_{D}^{k}\wedge(i\partial\bar{\partial}\vartheta_{s})^{n-k}+\sum_{k\,=\,1}^{n}k{n \choose k}\omega_{D}^{k-1}\wedge(i\partial\bar{\partial}\vartheta_{s})^{n-k}\wedge\omega_{C}\\
&\qquad-n\sum_{k\,=\,1}^{n}{n-1 \choose k-1}\omega_{D}^{k-1}\wedge(i\partial\bar{\partial}\vartheta_{s})^{n-k}\wedge\omega_{C}\\
&=(i\partial\bar{\partial}\vartheta_{s})^{n}+\sum_{k\,=\,1}^{n}{n \choose k}\omega_{D}^{k}\wedge(i\partial\bar{\partial}\vartheta_{s})^{n-k}\\
&\qquad+\sum_{k\,=\,1}^{n}\underbrace{\left[k{n \choose k}-n{n-1 \choose k-1}\right]}_{=\,0}\omega_{D}^{k-1}\wedge(i\partial\bar{\partial}\vartheta_{s})^{n-k}\wedge\omega_{C}\\
&=\sum_{k\,=\,0}^{n}{n \choose k}\omega_{D}^{k}\wedge(i\partial\bar{\partial}\vartheta_{s})^{n-k}\\
&=(\omega_{D}+i\partial\bar{\partial}\vartheta_{s})^{n}.
\end{split}
\end{equation*}
Combining \eqref{banana1} and \eqref{banana2}, we find that
\begin{equation}\label{lovely-eqn-der-rad}
\Delta_{\sigma_{s},\,X}\left(\frac{X\cdot\vartheta_{s}}{2}\right)=\underbrace{e^{-\frac{X\cdot\vartheta_{s}}{2}}\frac{(\omega_{D}+i\partial\bar{\partial}\vartheta_{s})^{n}}{\omega^{n}}}_{\textrm{first order
operator acting on $X\cdot\vartheta_s$}}.
\end{equation}
Indeed, the right-hand side of \eqref{lovely-eqn-der-rad} can be written schematically as:
\begin{equation}\label{lovely-eqn-der-rad-bis}
\frac{(\omega_{D}+i\partial\bar{\partial}\vartheta_{s})^{n}}{\omega^{n}}=\frac{1}{r^2}\left(X\cdot (X\cdot \vartheta_{s}) \alpha_1+\nabla^{g_D}(X\cdot \vartheta_{s})\ast \nabla^{g_D}(X\cdot \vartheta_{s})\ast \alpha_2\right),
\end{equation}
where $\alpha_1$ and $\alpha_2$ are tensors on $M\setminus V$ depending polynomially on $i\partial\bar{\partial}\vartheta_{s}$ and where $\ast$ denotes any linear combination of tensors with respect to the background metric $\omega$. This can be seen, for example, by noting that on $M\setminus V$,
\begin{equation*}
\begin{split}
	\frac{(\omega_{D}+i\partial\bar{\partial}\vartheta_{s})^{n}}{\omega^{n}} &=
\frac{(i\partial\bar{\partial}\vartheta_{s})^{n}}{\omega^{n}}+
 \sum_{k=1}^{n-1}{n \choose k}\frac{\omega_{D}^{n-k}\wedge(i\partial\bar{\partial}\vartheta_{s})^{k}}{\omega^{n}},
\end{split}
\end{equation*}
together with an application of the following claim.

\begin{claim}\label{commute}
Let $Y$ and $Z$ be real holomorphic vector fields such that $[Y,\,Z]=0$. Then for any smooth real-valued function $v$ on $M$ with $\mathcal{L}_{JY}v=\mathcal{L}_{JZ}v=0$, we have
$\frac{i}{2}\partial\bar{\partial}v(Y,\,Z)=\frac{i}{2}\partial\bar{\partial}v(JY,\,JZ)=0$ and $Z\cdot(Y\cdot v)=Y\cdot(Z\cdot v)=2i\partial\bar{\partial}v(Z,\,JY)$.
\end{claim}

\begin{proof}[Proof of Claim \ref{commute}]
The first equality follows from the fact that
$$2i\partial\bar{\partial}v(Y,\,Z)=2i\partial\bar{\partial}v(JY,\,JZ)=dd^{c}v(JY,\,JZ)=JY\cdot(d^{c}v(JZ))-JZ\cdot(d^{c}v(JY))-d^{c}v([JY,\,JZ]).$$ As for the
second, the vanishing of $[Y,\,Z]$ implies that $Z\cdot(Y\cdot v)=Y\cdot(Z\cdot v)$, whereas
with $Y^{1,\,0}:=\frac{1}{2}(Y-iJY)$ and $Z^{1,\,0}:=\frac{1}{2}(Z-iJZ)$, the invariance of $v$ and the fact that $JY\cdot(Z\cdot v)=0$ implies that
\begin{equation*}
\frac{1}{4}Y\cdot(Z\cdot v)=Y^{1,\,0}\cdot(Z^{1,\,0}\cdot v)=\overline{Y^{1,\,0}}\cdot (Z^{1,\,0}\cdot v)=\partial\bar{\partial}v(Z^{1,\,0},\,\overline{Y^{1,\,0}})
=\frac{i}{2}\partial\bar{\partial}v(Z,\,JY)-\frac{1}{2}\underbrace{\partial\bar{\partial}v(JY,\,JZ)}_{=\,0}.
\end{equation*}
\end{proof}

The strong maximum principle combined with the fact that $X\cdot\vartheta_{s}\to0$ at infinity now implies the result.
\end{proof}

From this, we can derive a lower bound on $X\cdot\vartheta_{s}$, and hence on $X\cdot\psi_{s}$.
\begin{prop}\label{lowerbound}
There exists a positive constant $C$ such that for all $s\in[0,\,1]$, $X\cdot\vartheta_s\geq -C$. In particular, $X\cdot\psi_{s}>-C$ for all $s\in[0,\,1]$.
\end{prop}

\begin{proof}
In order to prove that $X\cdot\vartheta_{s}$ is uniformly bounded from below, first note that since $X\cdot\Phi_{s}$ is bounded and
$X\cdot\vartheta_{s}$ tends to zero at infinity, $f_{\sigma_{s}}:=f+\frac{X}{2}\cdot\Phi_{s}+\frac{X}{2}\cdot\vartheta_{s}$ is a proper function bounded from below
by virtue of the fact that $f$ is by Lemma \ref{warm}. Then since $X=\nabla^{h_{s}}f_{\sigma_{s}}$, $f_{\sigma_{s}}$ must attain its global minimum at a point lying in the
zero set of $X$ and hence must coincide with the global minimum of $f$ on this set; that is to say,
$$f_{\sigma_{s}}\geq\min_{\{X\,=\,0\}}f_{\sigma_{s}}=\min_{\{X\,=\,0\}}f\geq -C.$$ The lower bound on
$X\cdot\vartheta_{s}$ then follows from the previous localisation of the minimum of this function given by Lemma \ref{lemma-loc-crit-pts-rad-der}.
\end{proof}

\subsection{A priori $C^{0}$-estimate}\label{sec-a-priori-energy}

We proceed with the a priori estimate on the $C^{0}$-norm of $(\vartheta_{s})_{0\,\leq \,s\,\leq\, 1}$
which is uniform in $s\in[0,\,1]$.
We begin with two crucial observations, the first a localisation result for the global extrema of $\vartheta_{s}$.

\begin{lemma}[Localising the supremum and infimum of a solution of \eqref{starstar-s}]\label{lemma-loc-crit-pts}
Let $(\vartheta_s)_{0\,\leq\, s\,\leq\, 1}$  be a path of solutions in $\mathbb{R}\oplus C^{\infty}_{X,\,\beta}(M)$ to \eqref{starstar-s}. Then
$\sup_M\vartheta_{s}=\max_{W}\vartheta_{s}$ (resp.~$\inf_M\vartheta_{s}=\min_{W}\vartheta_{s}$).
\end{lemma}

\begin{proof}
We prove the assertions of Lemma \ref{lemma-loc-crit-pts} that concern the supremum of a solution $\vartheta_{s}$ only. The statements
on the infimum of $\vartheta_{s}$ can be proved in a similar manner.

Observe from \eqref{starstar-s} {and the basic inequality $\log(1+x)\leq x$ for all $x>-1$} that $\vartheta_{s}$ is a subsolution of the following differential inequality:
\begin{equation*}
\Delta_{\omega_{s}}\vartheta_{s}-\frac{X}{2}\cdot\vartheta_{s}\geq G_{s},
\end{equation*}
{where recall that $G_{s}$ is compactly supported.} Let $\varepsilon>0$ and consider any smooth function $u_{\varepsilon}$ on $M$ identically equal to $2\varepsilon \log (r)$ on $M\setminus W$ such that $\lim_{\varepsilon\rightarrow 0} u_{\varepsilon}=0$ uniformly on compact sets of $M$. This function will serve as a barrier function. Indeed, {since $\log(r)$ is pluriharmonic}, one has that on $M\setminus W$,
\begin{equation}\label{max-ppe-vartheta2}
\Delta_{\omega_{s}}\left(\vartheta_{s}-2\varepsilon\log(r)\right)-\frac{X}{2}\cdot\left(\vartheta_{s}-2\varepsilon\log(r)\right)\geq \varepsilon>0.
\end{equation}
Now $\vartheta_{s}$ being bounded on $M$ implies that the function $\vartheta_{s}-2\varepsilon\log(r)$ tends to $-\infty$ as $r\to+\infty$. In particular, this latter function must attain its maximum on $M$. The maximum principle applied to \eqref{max-ppe-vartheta2} then ensures that it must be attained in $W$, i.e., $\max_M(\vartheta_{s}-u_{\varepsilon})=\max_W(\vartheta_{s}-u_{\varepsilon}).$ In conclusion, we have that
\begin{equation*}
\vartheta_s(x)\leq u_{\varepsilon}(x)+\max_W(\vartheta_{s}-u_{\varepsilon}),\qquad x\in M,
\end{equation*}
which leads to the bound $\vartheta_s(x)\leq \max_W \vartheta_s$ by letting $\varepsilon\to 0$ and making use of the assumption on $u_{\varepsilon}$.
Since this holds true for any $x\in M$, the desired estimate follows.
\end{proof}

\subsubsection{Aubin-Tian-Zhu's functionals}

We now introduce two functionals that have been defined and used by
Aubin \cite{Aub-red-Cas-Pos}, Bando and Mabuchi \cite{Ban-Mab-Uni}, and Tian \cite[Chapter $6$]{Tian-Can-Met-Boo}
in the study of Fano manifolds, and by Tian and Zhu \cite{Tian-Zhu-I} in the study of shrinking gradient K\"ahler-Ricci solitons
on compact K\"ahler manifolds.
\begin{definition}\label{IJ}
Let $(\varphi_t)_{0\,\leq\, t\,\leq\, 1}$ be a $C^1$-path in $\mathcal{M}^{\infty}_{X,\,\beta}(M)$ from $\varphi_{0}=0$ to $\varphi_{1}=\varphi$.
We define the following two generalised weighted energies:
\begin{equation*}
\begin{split}
I_{\omega,\,X}(\varphi)&:=\int_M\varphi\left(e^{-f}\omega^n-e^{-f-\frac{X}{2}\cdot\varphi}\omega_{\varphi}^n\right),\\
J_{\omega,\,X}(\varphi)&:=\int_0^1\int_M\dot{\varphi_s}\left(e^{-f}\omega^n-e^{-f-\frac{X}{2}\cdot\varphi_s}\omega_{\varphi_s}^n\right)\wedge ds.
\end{split}
\end{equation*}
\end{definition}

At first sight, these two functionals resemble relative weighted mean values of a potential $\varphi$ in
$\mathcal{M}^{\infty}_{X,\,\beta}(M)$ or of
a path $(\varphi_t)_{0\,\leq\, t\,\leq\, 1}$ in $\mathcal{M}^{\infty}_{X,\,\beta}(M)$ respectively. When $X\equiv 0$ and $(M,\,\omega)$ is a compact K\"ahler manifold,
an integration by parts together with some algebraic manipulations (see Aubin's seminal paper \cite{Aub-red-Cas-Pos} or Tian's book \cite[Chapter $6$]{Tian-Can-Met-Boo}) show that
\begin{equation}
\begin{split}\label{formulae--fct-I-J-ein}
I_{\omega,\,0}(\varphi)&=\sum_{k\,=\,0}^{n-1}\int_{M}i\partial\varphi\wedge\bar{\partial}\varphi\wedge\omega^k\wedge\omega_{\varphi}^{n-1-k},\\
J_{\omega,\,0}(\varphi)&=\sum_{k\,=\,0}^{n-1}\frac{k+1}{n+1}\int_{M}i\partial\varphi\wedge\bar{\partial}\varphi\wedge\omega^{k}\wedge\omega_{\varphi}^{n-1-k}.
\end{split}
\end{equation}
This justifies the description of $I_{\omega,\,0}(\varphi)$ and $J_{\omega,\,0}(\varphi)$ as modified energies.
Moreover, it demonstrates that on a compact K\"ahler manifold $J_{\omega,\,0}$ is a true functional, that is to say, it does not depend on the choice of path.

Such formulae \eqref{formulae--fct-I-J-ein} for $I_{\omega,\,X}$ and $J_{\omega,\,X}$
for a non-vanishing vector field $X$ and a non-compact K\"ahler manifold $(M,\,\omega)$
do not seem to be readily available for a good reason: the exponential function is not algebraic.
However, following Tian and Zhu's work \cite{Tian-Zhu-I}, one can
prove that the essential properties shared by both
$I_{\omega,\,0}$ and $J_{\omega,\,0}$ hold true for a non-vanishing vector field $X$ in a non-compact setting. The proof follows exactly as in \cite[Theorem 7.5]{conlon33}.

\begin{theorem}\label{main-thm-I-J}
$I_{\omega,\,X}(\varphi)$
and $J_{\omega,\,X}(\varphi)$ are well-defined for $\varphi\in\mathcal{M}^{\infty}_{X,\,\beta}(M)$.
Moreover, $J_{\omega,\,X}$ does not depend on the choice of a $C^1$ path $(\varphi_t)_{0\,\leq\, t\,\leq\, 1}$ in
$\mathcal{M}^{\infty}_{X,\,\beta}(M)$ from $\varphi_{0}=0$ to $\varphi_{1}=\varphi$, hence defines
a functional on $\mathcal{M}^{\infty}_{X,\,\beta}(M)$. Finally, the first variation of the difference $(I_{\omega,\,X}-J_{\omega,\,X})$ is
given by
\begin{equation*}
\frac{d}{dt}\left(I_{\omega,\,X}-J_{\omega,\,X}\right)(\varphi_t)=-\int_M\varphi_t\left(\Delta_{\omega_{\varphi_t}}\dot{\varphi_t}-
\frac{X}{2}\cdot\dot{\varphi_t}\right)\,e^{-f_{\varphi_t}}\omega_{\varphi_t}^n,
\end{equation*}
where $f_{\varphi_t}:=f+\frac{X}{2}\cdot\varphi_t$ satisfies
$X=\nabla^{\omega_{\varphi_t}}f_{\varphi_t}$
and where $(\varphi_t)_{0\,\leq\, t\,\leq\, 1}$ is any $C^1$-path in $\mathcal{M}^{\infty}_{X,\,\beta}(M)$ from $\varphi_{0}=0$ to $\varphi_{1}=\varphi$.
\end{theorem}

Recall that the equation we wish to solve is \eqref{star-s}, namely
\begin{equation*}
e^{-f_{\psi_{s}}}\omega_{\psi_{s}}^{n}=e^{F_s-f}\omega^{n}.
\end{equation*}

\begin{prop}[A priori energy estimates]\label{prop-a-priori-ene-est}
Let $(\psi_s)_{0\,\leq\, s\,\leq\, 1}$ be a path of solutions in $\mathcal{M}^{\infty}_{X,\,\beta}(M)$ to \eqref{star-s}. Then for $p\in(1,2)$, there exists a positive constant $C=C\left(n,p,\omega,\sup_{s\in[0,1]}\|F_s\|_{C^0}\right)$ such that
\begin{equation*}
\sup_{0\,\leq\, s\,\leq\, 1}\int_M|\psi_s-\overline{\psi}_{s}|^p\,e^{-f}\omega^n\leq C,
\end{equation*}
where $\overline{\psi}_{s}:=\int_{M}\psi_{s}\,e^{-f}\omega^{n}$. In particular, if $\overline{\psi}_{s}=0$, then
\begin{equation*}
\sup_{0\,\leq\, s\,\leq\, 1}\int_M|\vartheta_{s}|^p\,e^{-f}\omega^n\leq C.
\end{equation*}
\end{prop}

\begin{proof}
As a consequence of Theorem \ref{main-thm-I-J}, we can use any $C^1$-path $(\varphi_t)_{0\,\leq\, t\,\leq\, 1}$
in $\mathcal{M}^{\infty}_{X,\,\beta}(M)$ from $\varphi_{0}=0$ to $\varphi_{1}=\varphi\in\mathcal{M}^{\infty}_{X,\,\beta}(M)$
to compute $J_{\omega,\,X}(\varphi)$. As in \cite{Tian-Zhu-I}, we choose two different paths
to compute $J_{\omega,\,X}(\psi)$, the first being the linear path
defined by $\varphi_t:=t\psi$, $t\in[0,\,1]$, for $\psi\in\mathcal{M}^{\infty}_{X,\,\beta}(M)$ a solution to \eqref{star-s}. For this path, Theorem \ref{main-thm-I-J} asserts that
\begin{equation*}
\left(I_{\omega,\,X}-J_{\omega,\,X}\right)(\psi)=-\int_0^1\int_Mt\psi\left(\Delta_{\omega_{t\psi}}\psi-\frac{X}{2}\cdot\psi\right)\,e^{-f-t\frac{X}{2}\cdot\psi}\omega_{t\psi}^n\wedge dt.
\end{equation*}
Integration by parts with respect to the weighted volume form $e^{-f-t\frac{X}{2}\cdot\psi}\omega_{t\psi}^n$ then leads to
\begin{equation}
\begin{split}\label{bded-below-I-J}
\left(I_{\omega,\,X}-J_{\omega,\,X}\right)(\psi)&=n\int_0^1\int_Mt\, i\partial\psi\wedge\bar{\partial}\psi\wedge\,\left(e^{-f-t\frac{X}{2}\cdot\psi}\omega_{t\psi}^{n-1}\right)\wedge dt\\
&=n\int_0^1\int_M t\,i\partial\psi\wedge\bar{\partial}\psi\wedge\,\left(e^{-f-t\frac{X}{2}\cdot\psi}\left((1-t)\omega+t\omega_{\psi}\right)^{n-1}\right)\wedge dt\\
&= n\sum_{k\,=\,0}^{n-1}{{n-1}\choose{k}}\left(\int_0^1t^{k+1}(1-t)^{n-1-k}\int_M i\partial\psi\wedge\bar{\partial}\psi\wedge\left(e^{-f-t\frac{X}{2}\cdot\psi}\omega^{n-1-k}
\wedge\omega_{\psi}^k\right)\right)\wedge dt\\
&\geq n \int_0^1t(1-t)^{n-1}\int_M i\partial\psi\wedge\bar{\partial}\psi\wedge\left(e^{-f-t\frac{X}{2}\cdot\psi}\omega^{n-1}\right)\wedge dt\\
&=n\int_{M}\left(\int_0^1t(1-t)^{n-1}e^{-t\frac{X}{2}\cdot\psi}\,dt\right)i\partial\psi\wedge\bar{\partial}\psi\wedge e^{-f}\omega^{n-1}.
\end{split}
\end{equation}
From this, the following claim will allow us to obtain a lower bound.

\begin{claim}\label{claim-est-bded-below}
There exists positive uniform constants $A,\,c,$ such that
\begin{equation*}
\int_0^1t(1-t)^{n-1}e^{-t\frac{X}{2}\cdot\psi}\,dt\geq \frac{c}{\left(\frac{X}{2}\cdot\psi+A\right)^{2}}.
\end{equation*}
\end{claim}

\begin{proof}[Proof of Claim \ref{claim-est-bded-below}]
For $k\geq k_n:=2n(n-1)$, we find using integration by parts and a change of variable that
\begin{equation*}
\begin{split}
\int_0^1t(1-t)^{n-1}e^{-kt}\,dt&=\int_0^1(1-s)s^{n-1}e^{-k(1-s)}\,ds=e^{-k}\left\{\left(1+\frac{n}{k}\right)\int_0^1s^{n-1}e^{ks}\,ds-\frac{e^k}{k}\right\}\\
&=e^{-k}\left\{\left(1+\frac{n}{k}\right)\left(\frac{e^k}{k}-\frac{(n-1)}{k}\int_0^1s^{n-2}e^{ks}\,ds\right)-\frac{e^k}{k}\right\}\\
&\geq \left(1+\frac{n}{k}\right)\left(\frac{1}{k}-\frac{(n-1)}{k^2}(1-e^{-k})\right)-\frac{1}{k}\\
&=\frac{k-n(n-1)}{k^3}+e^{-k}\frac{(n+k)(n-1)}{k^3}\\
&\geq \frac{1}{2k^2}.
\end{split}
\end{equation*}
Here we have bounded $s^{n-2}$ from above by $1$ in the fourth inequality.

Set $A:=k_n-\inf_{M}\frac{X}{2}\cdot\psi$ and let $k=\frac{X}{2}\cdot\psi+A$. Then $k\geq k_n$, {$A$ is uniformly bounded from above by Proposition \ref{lowerbound}},
and it follows from what we have just derived that
\begin{equation*}
\int_0^1 t(1-t)^{n-1}e^{-t\left(\frac{X}{2}\cdot\psi+A\right)}\,dt\geq\frac{1}{2\left(\frac{X}{2}\cdot\psi+A\right)^{2}},
\end{equation*}
resulting in the desired bound.
\end{proof}

Applying Claim \ref{claim-est-bded-below} to \eqref{bded-below-I-J} yields the lower bound
\begin{equation}\label{inequ-bded-bel-I-J-fin}
(I_{\omega,\,X}-J_{\omega,\,X})(\psi)\geq c\int_M i\partial\psi\wedge\bar{\partial}\psi\wedge \frac{e^{-f}\omega^{n-1}}{\left(\frac{X}{2}\cdot\psi+A\right)^{2}}
\geq c\int_M \frac{|\nabla^{g}\psi|^{2}_{g}}{\left(\frac{X}{2}\cdot\psi+A\right)^{2}}\,e^{-f}\omega^{n}
\end{equation}
for some positive constant $c$. We also require an upper bound on
$(I_{\omega,\,X}-J_{\omega,\,X})(\psi)$ to complete the proof of the proposition.

To this end,
we consider the continuity path of solutions
$\varphi_s:=\psi_{s}$, $s\in[0,\,1]$, to \eqref{star-s}
to compute $(I_{\omega,\,X}-J_{\omega,\,X})(\psi)$. First observe that the
first variations $(\dot{\psi_s})_{0\,\leq\, s\,\leq\, 1}$ satisfy the following PDE
obtained from \eqref{star-s} by differentiating with respect to the parameter $s$:
\begin{equation*}
\begin{split}
\Delta_{\omega_{\psi_s}}\dot{\psi_s}-\frac{X}{2}\cdot\dot{\psi_s}=\dot{F}_{s},\qquad 0\leq s\leq 1.
\end{split}
\end{equation*}
Combined with \eqref{star-s} and Theorem \ref{main-thm-I-J}, we see that
\begin{equation*}
\begin{split}
(I_{\omega,\,X}-J_{\omega,\,X})(\psi)&=\int_0^1\int_M\psi_t\cdot(-\dot{F}_{t})\,e^{-f_{\psi_t}}\omega_{\psi_t}^n\wedge dt\\
&=\int_0^1\int_M\psi_t\cdot(-\dot{F}_{t})\,e^{F_{t}-f}\omega^n\wedge dt\\
\end{split}
\end{equation*}
so that, from \eqref{inequ-bded-bel-I-J-fin}, for some $c>0$,
\begin{equation}\label{star}
\int_0^1\int_M\psi_t\cdot(-\dot{F}_{t})\,e^{F_{t}-f}\omega^n\wedge dt
\geq c\int_M \frac{|\nabla^{g}\psi|^{2}_{g}}{\left(\frac{X}{2}\cdot\psi+A\right)^{2}}\,e^{-f}\omega^{n}.
\end{equation}
Now, as
$$\frac{d}{ds}\left(\int_{M}e^{-f_{\psi_{s}}}\omega_{\psi_{s}}^{n}\right)=0$$
by Lemma \ref{lemma-preserved-int}(i) with $G\equiv 1$, we derive from \eqref{star-s} that
$$\int_{M}\dot{F}_{t}e^{F_{t}-f}\omega^{n}=0.$$
This allows us to rewrite \eqref{star} as
\begin{equation*}
\int_0^1\int_M(\psi_t-\overline{\psi}_{t})\cdot(-\dot{F}_{t})\,e^{F_{t}-f}\omega^n\wedge dt
\geq c\int_M \frac{|\nabla^{g}\psi|^{2}_{g}}{\left(\frac{X}{2}\cdot\psi+A\right)^{2}}\,e^{-f}\omega^{n},
\end{equation*}
with $\overline{\psi}_{t}$ as in the statement of the proposition. Applying the Poincar\'e inequality of Proposition \ref{poincare}, we then
see that for any $p\in(1,\,2)$ and $\frac{1}{q}=1-\frac{1}{p}$,
\begin{equation}\label{hithere}
\begin{split}
\Biggl(\int_M|\psi&-\overline{\psi}|^{p}\,e^{-f}\omega^n\Biggl)^{\frac{2}{p}}\leq
C\left(\int_M|\nabla^{g}\psi|_{g}^{p}\,e^{-f}\omega^n\right)^{\frac{2}{p}}\\
&\leq C\left(\int_M\frac{|\nabla^{g}\psi|_{g}^{2}}{\left(\frac{X}{2}\cdot\psi+A\right)^{2}}\,e^{-f}\omega^n\right)
\left(\int_M\left(\frac{X}{2}\cdot\psi+A\right)^{\frac{2p}{2-p}}\,e^{-f}\omega^n\right)^{\frac{2-p}{p}}\\
&\leq C\left(\int_0^1\int_M|\psi_t-\overline{\psi}_{t}||\dot{F}_{t}|\,e^{F_{t}-f}\omega^n\wedge dt\right)
\left(\int_M\left(\frac{X}{2}\cdot\psi+A\right)^{\frac{2p}{2-p}}\,e^{-f}\omega^n\right)^{\frac{2-p}{p}}\\
&\leq C\int_{0}^{1}\left(\int_M|\psi_t-\overline{\psi}_{t}|^{p}\,e^{-f}\omega^n\right)^{\frac{1}{p}}
\left(\int_M|\dot{F}_{t}|^{q}\,e^{qF_{t}}\,e^{-f}\omega^n\right)^{\frac{1}{q}}\,dt\left(\int_M\left(\frac{X}{2}\cdot\psi+A\right)^{\frac{2p}{2-p}}\,e^{-f}\omega^n\right)^{\frac{2-p}{p}}\\
&\leq C\int_{0}^{1}\left(\int_M|\psi_t-\overline{\psi}_{t}|^{p}\,e^{-f}\omega^n\right)^{\frac{1}{p}}\,dt
\left(\int_M\left(\frac{X}{2}\cdot\psi+A\right)^{\frac{2p}{2-p}}\,e^{-f}\omega^n\right)^{\frac{2-p}{p}}.
\end{split}
\end{equation}
Here we have used H\"older's inequality in the second and fourth lines with respect to the weighted measure $e^{-f}\omega^n$.

Next, observe from Lemma \ref{lemma-preserved-int}(i) that for all $r\in\mathbb{N}$,
$$c\int_{M}(f_{\psi_{s}}+A)^{r}\,e^{-f}\omega^{n}\leq\int_{M}(f_{\psi_{s}}+A)^{r}e^{F_s}\,e^{-f}\omega^{n}=\int_{M}(f_{\psi_{s}}+A)^{r}\,e^{-f_{\psi_{s}}}\omega_{\psi_{s}}^{n}=\int_{M}(f+A)^{r}
\,e^{-f}\omega^{n}\leq C(r).$$
By induction {on $r$}, using the fact that $\frac{X}{2}\cdot\psi+A\geq0$
and that $A\leq C$ by Proposition \ref{lowerbound}, one can prove directly from this that
$$\int_M\left(\frac{X}{2}\cdot\psi+A\right)^{r}\,e^{-f}\omega^n\leq C(r)\qquad\textrm{for all $r\in\mathbb{N}$}.$$
It then follows from H\"older's inequality that this statement holds true for all $r\geq1$.
Applying this to \eqref{hithere}, we arrive at the fact that for all $p\in(1,\,2)$,
$$\left(\int_M|\psi-\overline{\psi}|^{p}\,e^{-f}\omega^n\right)^{\frac{2}{p}}\leq
C(p)\int_{0}^{1}\left(\int_M|\psi_{t}-\overline{\psi_{t}}|^{p}\,e^{-f}\omega^n\right)^{\frac{1}{p}}\,dt,$$
i.e.,
 \begin{equation*}
\|\psi-\overline{\psi}\|^2_{L^p(e^{-f}\omega^n)}\leq C(p)\int_0^1\|\psi_t-\overline{\psi}_{t}\|_{L^p(e^{-f}\omega^n)}\,dt\qquad\textrm{for any $p\in(1,\,2)$}.
\end{equation*}
This last inequality applies to any truncated path of the one-parameter family of solutions $(\psi_s)_{0\,\leq\, s\,\leq\, 1}$ of \eqref{star-s}. Thus,
 \begin{equation}
\begin{split}\label{crucial-a-priori-ineq-bis}
\|\psi_s-\overline{\psi}_{s}\|^2_{L^p(e^{-f}\omega^n)}&\leq C\int_0^1\|\psi_{st}-\overline{\psi}_{st}\|_{L^p(e^{-f}\omega^n)}\,dt\\
&=\frac{C}{s}\int_0^s\|\psi_{t}-\overline{\psi}_{t}\|_{L^p(e^{-f}\omega^n)}\,dt.
\end{split}
\end{equation}
This is a Gr\"onwall-type differential inequality and can be integrated as follows.
Let $$H(s):=\int_0^s\|\psi_{t}-\overline{\psi}_{t}\|_{L^p(e^{-f}\omega^n)}\,dt$$ and observe that \eqref{crucial-a-priori-ineq-bis} can be rewritten as
\begin{equation*}
\begin{split}
H'(s)\leq \frac{C}{s^{\frac{1}{2}}}(H(s))^{\frac{1}{2}},\quad s\in(0,1].
\end{split}
\end{equation*}
Integrating then implies that $H(s)\leq C\left(n,\omega,\sup_{s\in[0,1]}\|F_s\|_{C^0}\right)\cdot s$
for all $s\in[0,\,1]$ which, after applying \eqref{crucial-a-priori-ineq-bis} once more, yields the desired upper bound.
\end{proof}

\subsubsection{A priori estimate on $\sup_M\vartheta_{s}$}
Let $\vartheta_s$ be a solution to \eqref{starstar-s} for some fixed value of the parameter $s\in[0,\,1]$.
We next obtain an upper bound for $\sup_M\vartheta_s$ uniform in $s$.
To obtain such a bound, it suffices by Lemma \ref{lemma-loc-crit-pts} to only bound $\max_{W}\vartheta_s$ from above.
We do this by implementing a local Nash-Moser iteration using the fact that $\vartheta_{s}$ is a super-solution of the
linearised complex Monge-Amp\`ere equation of which the drift Laplacian with respect to the known metric $\omega_{s}$
forms a part.

\begin{prop}[A priori upper bound on $\sup_M\vartheta$]\label{prop-bd-abo-uni-psi}
Let $(\vartheta_s)_{0\,\leq\, s\,\leq\, 1}$ be a path of solutions in $\mathbb{R}\oplus C^{\infty}_{X,\,\beta}(M)$ to \eqref{starstar-s}. Then there exists a positive constant $C=C\left(n,\omega,\sup_{s\in[0,1]}\|G_s\|_{C^0}\right)$ such that
\begin{equation*}
\sup_{0\,\leq\, s\,\leq\, 1}\sup_{W}\vartheta_s\leq C.
\end{equation*}
\end{prop}

\begin{proof}
Let $s\in[0,\,1]$ and let
$(\vartheta_{s})_+:=\max\{\vartheta_{s},\,0\}$. This is a non-negative Lipschitz function. The strategy of proof is standard:
we use a Nash-Moser iteration to obtain an a priori upper bound on $\sup_{W}(\vartheta_{s})_+$ in terms of the
(weighted) energy of $(\vartheta_{s})_+$ on a tubular neighbourhood of $W$. The result then follows by invoking Proposition \ref{prop-a-priori-ene-est}.

To this end, notice that since $\log(1+x)\leq x$ for all $x>-1$ and since $\vartheta_{s}$ is a solution to \eqref{starstar-s}, $\vartheta_{s}$ satisfies the differential inequality
\begin{equation}\label{sub-diff-inequ-psi}
\Delta_{\omega_{s}}\vartheta_{s}-\frac{X}{2}\cdot \vartheta_{s}\geq -|G_{s}|\qquad\text{on $M$.}
\end{equation}
Let $g_{s}$ denote the K\"ahler metric associated to $\omega_{s}$ and let $f_{\omega_{s}}:=f+\frac{X}{2}\cdot\Phi_{s}$. Then these metrics are all equivalent to $g$ uniformly in $s$
and $-\omega_{s}\lrcorner X=df_{\omega_{s}}$. Let $x\in\{f\,<\,R\}$ and $\varepsilon>0$ be such that $B_{g_{s}}(x,\,\varepsilon)\Subset \{f\,<\,R\}$ and multiply \eqref{sub-diff-inequ-psi} across
by $\eta_{t,\,t'}^2(\vartheta_{s})_+|(\vartheta_{s})_+|^{2(p-1)}$ with $p\geq 1$, where $\eta_{t,\,t'}$, with $0<t+t'<\varepsilon$ and $t,\,t'>0$, is a Lipschitz cut-off function with compact support in $B_{g_{s}}(x,\,t+t')$ equal to $1$ on $B_{g_{s}}(x,\,t)$ and with $|\nabla^{g_{s}}\eta_{t,\,t'}|_{g_{s}}\leq\frac{1}{t'}$ almost everywhere. Next, integrate by parts and use a local Sobolev inequality for the pair $(\omega_{s},\,f_{\omega_{s}})$ to obtain a so-called ``reversed H\"older inequality'' which, after iteration, leads to the following bound for $p\in(1,2)$:
\begin{equation*}
\begin{split}\label{first-a-priori-c-0-est-cpct-part-non-lin}
\sup_{B_{g_{s}}(x,\,\frac{\varepsilon}{2})}(\vartheta_{s})_+&\leq C(n,p,\omega,\varepsilon)\left(\|(\vartheta_{s})_+\|_{L^p(B_{g_{s}}(x,\,\varepsilon),\,e^{-f_{\omega_{s}}}\omega_{s}^{n})}^p+\|G_{s}\|^p_{C^0}\right)^{\frac{1}{p}}\\
&\leq C(n,p,\omega,\varepsilon)\left(\int_{\{f\,<\,R\}}(\vartheta_{s})_+^p\,e^{-f_{\omega_{s}}}\omega_{s}^n+\|G_{s}\|^p_{C^0}\right)^{\frac{1}{p}}\\
&\leq C(n,p,\omega,\varepsilon)\left(\int_{\{f\,<\,R\}}|\vartheta_{s}|^p\,e^{-f}\omega^n+\|G_{s}\|^p_{C^0}\right)^{\frac{1}{p}}\\
&\leq C\left(n,p,\omega,\varepsilon,\sup_{s\,\in\,[0,\,1]}\|G_s\|_{C^0}\right).
\end{split}
\end{equation*}
Here, we have made use of Proposition \ref{prop-a-priori-ene-est} in the last line.
\end{proof}

\subsubsection{A priori estimate on $\inf_M\vartheta_{s}$}

Recall that the equation we wish to solve is \eqref{star-s}, that is,
\begin{equation*}
e^{-f_{\psi_{s}}}\omega^{n}_{\psi_{s}}=e^{F_s-f}\omega^{n},
\end{equation*}
where $\omega_{\psi_{s}}:=\omega+i\partial\bar{\partial}\psi_{s}>0$
and $f_{\psi_{s}}:=f+\frac{X}{2}\cdot\psi_{s}$. This pair satisfies $-\omega_{\psi_{s}}\lrcorner X=df_{\psi_{s}}$.
We work under the assumption that $\int_{M}\psi_{s}\,e^{-f}\omega^{n}=0$.

\subsubsection*{An upper bound on the $I_{\omega,\,X}$-functional}
We first show that the $I_{\omega,\,X}$-functional is bounded along the continuity path.
\begin{lemma}\label{I-bounded}
$\sup_{s\,\in\,[0,\,1]}I_{\omega,\,X}(\psi_{s})\leq C(\sup_M(\vartheta_{s})_+).$
\end{lemma}

\begin{proof}
By assumption, $\int_{M}\psi_{s}\,e^{-f}\omega^{n}=0$ so that $\int_{\{\psi_{s}\,\geq\,0\}}\psi_{s}\,e^{-f}\omega^{n}=-\int_{\{\psi_{s}\,\leq\,0\}}\psi_{s}\,e^{-f}\omega^{n}.$
We therefore have that
\begin{equation*}
\begin{split}
I_{\omega,\,X}(\psi_{s})&=\int_M\psi_{s}\left(e^{-f}\omega^n-e^{-f_{\psi_{s}}}\omega_{\psi_{s}}^n\right)
=-\int_M\psi_{s}\,e^{-f_{\psi_{s}}}\omega_{\psi_{s}}^n\\
&=-\int_{\{\psi_{s}\,\geq\,0\}}\psi_{s}\,e^{-f_{\psi_{s}}}\omega_{\psi_{s}}^n+\int_{\{\psi_{s}\,\leq\,0\}}(-\psi_{s})\,e^{-f_{\psi_{s}}}\omega_{\psi_{s}}^n
\leq\int_{\{\psi_{s}\,\leq\,0\}}(-\psi_{s})\,e^{-f_{\psi_{s}}}\omega_{\psi_{s}}^n\\
&=\int_{\{\psi_{s}\,\leq\,0\}}(-\psi_{s})\,e^{F_s}\,e^{-f}\omega^n\leq C\int_{\{\psi_{s}\,\leq\,0\}}(-\psi_{s})\,e^{-f}\omega^n=C\int_{\{\psi_{s}\,\geq\,0\}}\psi_{s}\,e^{-f}\omega^n\\
&=C\left(\int_{\{\vartheta_{s}\,\geq\,-\Phi_{s}\}}(\vartheta_{s}+\Phi_{s})\,e^{-f}\omega^n\right)\leq C\Biggl(\underbrace{\int_{M}|\Phi_{s}|\,e^{-f}\omega^{n}}_{\textrm{bounded}}+\int_{\{\vartheta_{s}\,\geq\,-\Phi_{s}\}}\vartheta_{s}\,e^{-f}\omega^n\Biggr)\\
&\leq C+C\sup_M\vartheta_s^+\int_{\{\vartheta_{s}\,\geq\,-\Phi_{s}\}}\,e^{-f}\omega^n\leq C+C\sup_M(\vartheta_{s})_+\int_{M}\,e^{-f}\omega^n\\
&\leq C(1+\sup_M(\vartheta_{s})_+).
\end{split}
\end{equation*}
From this, the result follows.
\end{proof}

\subsubsection*{An upper bound on the weighted $L^{p}$-norm of the gradient of the Legendre transform}

Recall the continuity path \eqref{star-s}:
\begin{equation*}
(\omega+i\partial\bar{\partial}\psi_{s})^{n}=e^{F_{s}+\frac{X}{2}\cdot\psi_{s}}\omega^{n},\qquad s\in[0,\,1],
\end{equation*}
where $$F_{s}:=\log\left(se^{F}+(1-s)\right)\qquad\textrm{and}\qquad i\partial\bar{\partial}F=\rho_{\omega}+\frac{1}{2}\mathcal{L}_{X}\omega-\omega.$$
Here, $\rho_{\omega}$ denotes the Ricci form of $\omega$ and $F\in C^{\infty}(M)$ is bounded.
On $\mathfrak{t}\simeq\mathbb{R}^{n}$ we have coordinates $\xi:=(\xi_{1},\ldots,\xi_{n})$, induced coordinates $x=(x_{1},\ldots,x_{n})$ on $\mathfrak{t}^{*}$ which contains the image
of the moment map, and we can write $\omega=2i\partial\bar{\partial}\phi_{0}$ for a convex function $\phi_{0}$ on $\mathbb{R}^{n}\simeq\mathfrak{t}$
up to the addition of a linear function (cf.~Section \ref{toric-geom}). Let $b_{X}\in\mathbb{R}^{n}$ denote the vector field $JX\in\mathfrak{t}$ as in \eqref{eqnY4},
write $\nabla$ for the Levi-Civita connection of the flat metric on $\mathbb{R}^{n}$, and $\langle\cdot\,,\,\cdot\rangle$ for the corresponding inner product.
As in \eqref{norm}, we normalise $\phi_{0}$ so that
\begin{equation*}
F=-\log\det(\phi_{0,\,ij})+\langle\nabla\phi_0,\,b_X \rangle-2\phi_0.
\end{equation*}
Set $\phi_s:=\phi_{0}+\frac{1}{2}\psi_{s}$. Then in the coordinates $\xi$ on $\mathbb{R}^{n}$, equation \eqref{star-s} becomes
$$\det(\phi_{s,\,ij})=\left(se^{F} + (1-s)\right)e^{\langle \nabla \phi_s,\,b_X \rangle - \langle \nabla \phi_0,\,b_X \rangle}\det(\phi_{0,\,ij}),\qquad s\in[0,\,1].$$
Plugging in the definition of $F$, this becomes
\begin{equation*}
\begin{split}
\det(\phi_{s,\,ij})&=\left(se^{-2\phi_0 - \log\det(\phi_{0,\,ij})} + (1-s)e^{-\langle\nabla\phi_0, b_X\rangle}\right)e^{\langle \nabla \phi_s,\,b_X \rangle}\det(\phi_{0,\,ij}) \\
				&= \left(se^{-2\phi_0} + (1-s)e^{-\langle\nabla \phi_0,\,b_X \rangle}\det(\phi_{0,\,ij}) \right)e^{\langle\nabla\phi_s,\,b_X \rangle},\qquad s\in[0,\,1],
\end{split}
\end{equation*}
or equivalently,
\begin{equation}\label{caddo}
e^{-\langle \nabla \phi_s, \,b_X \rangle}\det(\phi_{s,\,ij}) = se^{-2\phi_0}+(1-s)e^{-\langle\nabla \phi_0,\,b_X \rangle}\det(\phi_{0,\,ij}),\qquad s\in[0,\,1].
\end{equation}
Let $u_{s}=L(\phi_{s})$. Then we have the following uniform integral bound on $|\nabla u_{s}|^{p},\,p\geq1$.

\begin{lemma}\label{Lpnorm}
For all $p\geq1$,
$$\sup_{s\,\in\,[0,\,1]}\int_{P_{-K_{M}}}|\nabla u_{s}|^{p}e^{-\langle b_{X},\,x\rangle}\,dx\leq C.$$
\end{lemma}

\begin{proof}
First note that
$$\int_{\mathbb{R}^{n}}|\xi|^{p}e^{-\langle b_{X},\,\nabla\phi_{0}\rangle}\det(\phi_{0,ij})\,d\xi\leq C.$$
Indeed, since $F$ is equal to a constant $c_0$ off of a compact subset of $M$, we see that $F$ is globally bounded on $M$. This means that
$$ \sup_{\mathbb{R}^n}\left| -\log\det(\phi_{0,\,ij}) + \langle \nabla \phi_0, b_X \rangle - 2\phi_0 \right|\leq C, $$
resulting in the fact that
$$\int_{\mathbb{R}^{n}}|\xi|^{p}e^{-\langle b_{X},\,\nabla\phi_{0}\rangle}\det(\phi_{0,\,ij})\,d\xi\leq C \int_{\mathbb{R}^{n}}|\xi|^{p}e^{-2\phi_0}\,d\xi\leq C, $$
where we have used Lemma \ref{growth} in the last inequality. Therefore, using Lemma \ref{growth} once again and \eqref{caddo}, we find that
\begin{equation*}
\begin{split}
\int_{P_{-K_{M}}}|\nabla u_{s}|^{p}e^{-\langle b_{X},\,x\rangle}\,dx&=\int_{\mathbb{R}^{n}}|\xi|^{p}e^{-\langle b_{X},\,\nabla\phi_{s}\rangle}\det(\phi_{s,ij})\,d\xi \\
 &= s\int_{\mathbb{R}^{n}}|\xi|^{p}e^{-2\phi_{0}}\,d\xi   + (1-s) \int_{\mathbb{R}^{n}}|\xi|^{p}e^{-\langle b_{X},\,\nabla\phi_{0}\rangle}\det(\phi_{0,ij})\,d\xi  \\
 &\leq C,
\end{split}
\end{equation*}
as desired.
\end{proof}

\subsubsection*{An upper bound on the $\hat{F}$-functional}

Now, our background metric $\omega$ satisfies the two bullet points above
Lemma \ref{expression} as demonstrated in the already proved Theorem \ref{mainthm}(ii)--(iv).
As a consequence, it is clear from Lemma \ref{expression}(i) that condition (a) of Definition \ref{fhat} holds true.
The hypothesis of Lemma \ref{converge2} as well as condition (b) of Definition \ref{fhat} via Lemma \ref{converge1} also
hold true thanks to Lemma \ref{lemma-preserved-int}(ii). Thus, the $\hat{F}$-functional from Definition \ref{fhat} is finite and therefore well-defined along the continuity path \eqref{star-s}
and moreover, by Lemma \ref{converge2}, may be expressed along in terms of the $J_{\omega,\,X}$-functional as $$\hat{F}(\psi_{s})=J_{\omega,\,X}(\psi_{s})-\int_{M}\psi_{s}\,e^{-f}\omega^{n}.$$

We next show that $\hat{F}$ is bounded above along the continuity path \eqref{star-s} using Lemma \ref{I-bounded}. This will in turn provide
an a priori estimate on the weighted integral of the Legendre transform $u_{s}:=L(\phi_{s})$ of $\phi_{s}$. From this,
we derive an a priori estimate on the weighted $L^{1}$-norm of $u_{s}$. Via the Sobolev inequality, we then obtain local
control on $u_{s}$, and as a result, on $\psi_{s}$. This eventually leads to the desired uniform lower bound on $\inf_{M}\vartheta_{s}$.

\begin{lemma}\label{boundedd}
$\hat{F}(\psi_{s})\leq C(\sup_M(\vartheta_{s})_+).$
\end{lemma}

\begin{proof}
By assumption we have that $\int_{M}\psi_{s}\,e^{-f}\omega^{n}=0$ so that
$\hat{F}(\psi_{s})=J_{\omega,\,X}(\psi_{s})$. Moreover, from \eqref{inequ-bded-bel-I-J-fin} we read that $(I_{\omega,\,X}-J_{\omega,\,X})(\psi_{s})\geq0$. Thus, Lemma \ref{I-bounded} implies that
\begin{equation*}
\begin{split}
\hat{F}(\psi_{s})&=J_{\omega,\,X}(\psi_{s})=I_{\omega,\,X}(\psi_{s})-(I_{\omega,\,X}-J_{\omega,\,X})(\psi_{s})\leq I_{\omega,\,X}(\psi_{s})+0\leq C\left(\sup_{M}(\vartheta_{s})_+\right),
\end{split}
\end{equation*}
as claimed.
\end{proof}

\subsubsection*{An upper bound on the weighted integral of the Legendre transform}
We know that
\begin{equation*}
\begin{split}
\int_{P_{-K_{M}}}|u_{s}|\,e^{-\langle b_{X},\,x\rangle}dx&\leq
\int_{P_{-K_{M}}}|u_{s}-u_{0}|\,e^{-\langle b_{X},\,x\rangle}dx+\int_{P_{-K_{M}}}|u_{0}|\,e^{-\langle b_{X},\,x\rangle}dx\\
&\leq\int_{P_{-K_{M}}}\left(\int_{0}^{1}|\dot{u}_{st}|\,dt\right)\,e^{-\langle b_{X},\,x\rangle}dx+\int_{P_{-K_{M}}}|u_{0}|\,e^{-\langle b_{X},\,x\rangle}dx,
\end{split}
\end{equation*}
and these last two integrals are finite by Lemma \ref{lemma-preserved-int}(ii) via Lemma \ref{converge1}, and Lemma \ref{expression}(ii), respectively.
By definition, the $\hat{F}$-functional along \eqref{star-s} is given by
\begin{equation}\label{fhat2}
\hat{F}(\psi_s)=2\int_{P_{-K_{M}}}(u_{s}-u_{0})\,e^{-\langle b_{X},\,x\rangle}dx.
\end{equation}
Therefore with $\int_{P_{-K_{M}}}|u_{0}|\,e^{-\langle b_{X},\,x\rangle}dx$
and $\int_{P_{-K_{M}}}|u_{1}|\,e^{-\langle b_{X},\,x\rangle}dx$ convergent, we can split
the integral in \eqref{fhat2}. Together with the integral bound given in Lemma \ref{expression}(ii), this leads to the following consequence of Lemma \ref{boundedd}.

\begin{corollary}\label{boundd}
$$\sup_{s\,\in\,[0,\,1]}\int_{P_{-K_{M}}}u_{s}\,e^{-\langle b_{X},\,x\rangle}\,dx\leq C.$$
\end{corollary}

\subsubsection*{An upper bound on the weighted $L^{1}$-norm of the Legendre transform}

We now use Corollary \ref{boundd} to derive a uniform weighted $L^{1}$-norm on $u_{s}$.
Notice that we must make use of the already obtained uniform upper bound on $\vartheta_{s}$.

\begin{lemma}\label{L1norm}
$$\sup_{s\,\in\,[0,\,1]}\int_{P_{-K_{M}}}|u_s|\,e^{-\langle b_X,x\rangle}\,dx\leq C.$$
\end{lemma}

\begin{proof}
Recall from the definition of the Legendre transform that for all $x\in P_{-K_{M}}$,
\begin{equation*}
\begin{split}
u_s(x)-u_0(x)&=\sup_{\xi\,\in\,\mathbb{R}^n}\left(\langle x,\,\xi\rangle-\phi_s(\xi)\right)-u_0(x)\\
&\geq\langle x,\,\nabla u_0(x)\rangle-\phi_s(\nabla u_0(x))-u_0(x)\\
&= \phi_0(\nabla u_0(x))-\phi_s(\nabla u_0(x))\\
&= -\frac{1}{2}\psi_s(\nabla u_0(x))\\
&=-\frac{1}{2}\Phi_s(\nabla u_0(x))-\frac{1}{2}\vartheta_s(\nabla u_0(x))\\
&\geq-\frac{1}{2}\Phi_s(\nabla u_0(x))-C\\
\end{split}
\end{equation*}
for some uniform positive constant $C$. Here we have used the a priori upper bound on $\vartheta_s$ given by Proposition \ref{prop-bd-abo-uni-psi} in the last line.
With this, we estimate that
\begin{equation}\label{finished}
\begin{split}
\int_{P_{-K_{M}}}&|u_s|\,e^{-\langle b_X,x\rangle}dx\leq\int_{P_{-K_{M}}}\left(u_s-u_0+\frac{1}{2}\Phi_s(\nabla u_0(x))+C\right)\,e^{-\langle b_X,\,x\rangle}dx\\
&\qquad+\int_{P_{-K_{M}}}\left|u_0-\frac{1}{2}\Phi_s(\nabla u_0(x))-C\right|\,e^{-\langle b_X,\,x\rangle}dx\\
&\leq\int_{P_{-K_{M}}}u_s\,e^{-\langle b_X,\,x\rangle}dx+2\int_{P_{-K_{M}}}|u_0|\,e^{-\langle b_X,\,x\rangle}dx+2C\int_{P_{-K_{M}}}e^{-\langle b_X,\,x\rangle}dx\\
&\qquad+\int_{P_{-K_{M}}}|\Phi_s(\nabla u_0(x))|\,e^{-\langle b_X,\,x\rangle}dx\\
&\leq C'+\int_{P_{-K_{M}}}|\Phi_s(\nabla u_0(x))|\,e^{-\langle b_X,\,x\rangle}dx
\end{split}
\end{equation}
for a uniform positive constant $C'$. Here we have used Corollary \ref{boundd}, Lemma \ref{expression}(ii), and the fact that
$$\int_{P_{-K_{M}}}e^{-\langle b_X,\,x\rangle}dx=(2\pi)^{n}\int_{M}e^{-f}\omega^{n}<\infty$$
to bound each of the terms in the third line respectively. The final integral we bound in the following way.

Choose a compact subset $U\subset M$ strictly containing $W$ and $f^{-1}((-\infty,\,1])$. This we can do because $f$ is proper and bounded below.
Next, choose $R>0$ sufficiently large so that $(\nabla\phi_{0})(U)\subset B_{R}(0)$. Then in particular,
$(\nabla\phi_{0})(W)\subset B_{R}(0)$ and $\langle b_{X},\,x\rangle>1$ for all $x\in P_{-K_{M}}\setminus(B_{R}(0)\cap P_{-K_{M}})$, the
latter being true because $\langle b_{X},\,x\rangle=f(\nabla u_{0}(x))$ for all $x\in P_{-K_{M}}$. Then recalling that $\Phi_{s}=-c_{s}\log(2(f+1))$ on $M\setminus W$, which in particular holds on
$P_{-K_{M}}\setminus(B_{R}(0)\cap P_{-K_{M}})$, and using the fact that $0<\log(x)<x$ for all $x>1$, we estimate that
\begin{equation*}
\begin{split}
\int_{P_{-K_{M}}}|\Phi_s(\nabla u_0(x))|\,e^{-\langle b_X,\,x\rangle}dx&=\int_{B_{R}(0)\cap P_{-K_{M}}}|\Phi_s(\nabla u_0(x))|\,e^{-\langle b_X,\,x\rangle}dx\\
&\qquad+\int_{P_{-K_{M}}\setminus(B_{R}(0)\cap P_{-K_{M}})}|\Phi_s(\nabla u_0(x))|\,e^{-\langle b_X,\,x\rangle}dx\\
&\leq C\left(1+\int_{P_{-K_{M}}\setminus(B_{R}(0)\cap P_{-K_{M}})}|\log(2(f(\nabla u_0(x))+1))|\,e^{-\langle b_X,\,x\rangle}dx\right)\\
&=C\left(1+\int_{P_{-K_{M}}\setminus(B_{R}(0)\cap P_{-K_{M}})}\log(2(\langle b_{X},\,x\rangle+1))\,e^{-\langle b_X,\,x\rangle}dx\right)\\
&\leq C\left(1+\int_{P_{-K_{M}}\setminus(B_{R}(0)\cap P_{-K_{M}})}(1+\langle b_{X},\,x\rangle)\,e^{-\langle b_X,\,x\rangle}dx\right)\\
&\leq C'
\end{split}
\end{equation*}
for a uniform positive constant $C'$. Combined with \eqref{finished}, this yields the desired bound.
\end{proof}

\subsubsection*{Local control on $u_{s}$}

Lemmas \ref{Lpnorm} and \ref{L1norm}, combined with an application of the Sobolev inequality, now give us local control on $u_{s}$.
\begin{prop}\label{u-bound}
There exists $C>0$ such that for all $x\in P_{-K_{M}}$ and $s\in[0,\,1]$, $$|u_{s}(x)-u_{0}(x)|\leq Ce^{\langle b_{X},\,x\rangle}.$$
\end{prop}

\begin{proof}
From the first paragraph of the proof of Lemma \ref{useful}, we know that outside a compact subset,
$P_{-K_{M}}$ coincides with the Cartesian product of the half line and $P_{D}$, the polytope associated
to $D$. More precisely, in light of \eqref{product-poly}, $P_{-K_{M}}$ coincides with
$[a,\,\infty)\times P_{D}\subseteq\mathbb{R}\times\mathbb{R}^{n-1}$ for some $a\in\mathbb{R}$ outside a convex compact subset.
Suppose that $x\in P_{-K_{M}}$ lies in the region $[a+1,\,\infty)\times P_{D}$.
Then there exists $b\in[a+1,\,\infty)$ such that $x\in\{b\}\times P_{D}$.
Let $\Omega:=[b-1,\,b+1]\times P_{D}\subseteq[a,\,\infty)\times P_{D}\subseteq\mathbb{R}\times\mathbb{R}^{n-1}$.
Set $U_{s}:=u_{s}-u_{0}$ and let $q>n$. Then since $U_{s}$ is smooth up to $\partial P_{-K_{M}}$ by Lemma
\ref{boundaryy}(i), we can apply the Sobolev inequality from \cite[Theorem 3.4]{polysob} (which in particular states that the Sobolev constant depends only on the Euclidean diameter and measure of $\Omega$), together with Lemmas \ref{Lpnorm} and \ref{L1norm}, to determine that for a uniform constant $C>0$,
\begin{equation*}
\begin{split}
|U_{s}(x)|&\leq\norm{U_{s}}_{C^{0}(\Omega)}\leq\left\lVert U_{s}-\frac{1}{|\Omega|}\int_{\Omega}U_{s}\,dx\right\rVert_{C^{0}(\Omega)}+\frac{1}{|\Omega|}\int_{\Omega}|U_{s}|\,dx
\\
&\leq C\norm{\nabla U_{s}}_{L^{q}(\Omega)}+\frac{1}{|\Omega|}\int_{\Omega}|U_{s}|\,dx\\
&\leq C\left(\sup_{y\,\in\,\Omega}e^{\langle b_{X},\,y\rangle}+\left(\sup_{y\,\in\,\Omega}e^{\langle b_{X},\,y\rangle}\right)^{\frac{1}{q}}\right).\\
&\leq C\sup_{y\,\in\,\Omega}e^{\langle b_{X},\,y\rangle},\\
\end{split}
\end{equation*}
because $0<\frac{1}{q}<1$. Continuing, we find that
$$|U_{s}(x)|\leq C\sup_{y\,\in\,\Omega}e^{\langle b_{X},\,y\rangle}=Ce^{\langle b_{X},\,x\rangle}\cdot\sup_{y\,\in\,\Omega}e^{\langle b_{X},\,y-x\rangle}\leq Ce^{\langle b_{X},\,x\rangle}.$$

A slight modification of this argument also shows that
$|U_{s}(x)|\leq Ce^{\langle b_{X},\,x\rangle}$ for all $x\in \linebreak P_{-K_{M}}\setminus([a+1,\,\infty)\times P_{D})$ which as noted above, is a compact
convex subset of $\mathbb{R}^{n}$. In sum, we arrive at the bound
$$|U_{s}(x)|\leq Ce^{\langle b_{X},\,x\rangle}\qquad\textrm{for all $x\in P_{-K_{M}}$},$$
as required.
\end{proof}

\subsubsection*{Local control on $\psi_{s}$}

The previous proposition can be reformulated to give local control on $\psi_{s}$.

\begin{prop}\label{lowerr}
There exists $C>0$ such that for all $x\in M$ and $s\in[0,\,1]$,
$$\psi_{s}(x)\geq-Ce^{f(x)}.$$
\end{prop}

\begin{proof}
The definition of the Legendre transform and Proposition \ref{u-bound} gives us that for all $\xi\in\mathbb{R}^{n}$ and $s\in[0,\,1]$,
\begin{equation*}
\begin{split}
\psi_{s}(\xi)&=2(\phi_{s}(\xi)-\phi_{0}(\xi))\\
&=2\left(\sup_{x\,\in\,P_{-K_{M}}} \left\{ \langle \xi, x \rangle  - u_{s}(x) \right\}-\phi_{0}(\xi)\right)\\
&\geq2\left( \langle \xi,\,\nabla\phi_{0}(\xi) \rangle  - u_{s}(\nabla\phi_{0}(\xi)) -\phi_{0}(\xi)\right)\\
&=2\left( u_{0}(\nabla\phi_{0}(\xi))- u_{s}(\nabla\phi_{0}(\xi))\right)\\
&\geq -Ce^{\langle b_{X},\,\nabla\phi_{0}(\xi)\rangle}\\
&=-Ce^{f(\xi)},
\end{split}
\end{equation*}
for some uniform $C>0$, as claimed.
\end{proof}

\subsubsection*{A priori lower bound on $\inf_M\vartheta_{s}$}
This brings us to the concluding bound of this section. Proposition \ref{lowerr} yields a uniform lower bound on $\min_{W}\psi_{s}$. By Lemma \ref{lemma-loc-crit-pts}, this results in a uniform lower bound on $\inf_{M}\vartheta_{s}$. This is demonstrated in the following proposition.

\begin{prop}[A priori lower bound on $\inf_M\vartheta_{s}$]\label{prop-bd-bel-uni-psi}
Let $(\vartheta_s)_{0\,\leq\, s\,\leq\, 1}$ be a path of solutions in $\mathbb{R}\oplus C^{\infty}_{X,\,\beta}(M)$ to \eqref{starstar-s}. Then there exists a uniform constant $C>0$ such that
\begin{equation*}
\inf_{0\,\leq\, s\,\leq\, 1}\inf_{M}\vartheta_s\geq-C.
\end{equation*}
\end{prop}

\begin{proof}
Combining Lemma \ref{lemma-loc-crit-pts} and Proposition \ref{lowerr}, we find that for all $s\in[0,\,1]$,
$$\inf_{M}\vartheta_s=\min_{W}\vartheta_{s}=\min_{W}\left(\psi_{s}-\Phi_{s}\right)\geq\min_{W}\left(-Ce^{f}-\Phi_{s}\right)\geq-C.$$
\end{proof}

\subsection{A priori upper bound on the radial derivative}\label{sec-upp-bd-rad-der}

The $C^{0}$-bound on $\vartheta_{s}$ allows us to derive an a priori upper bound on $X\cdot\vartheta_{s}$.
\begin{prop}\label{prop-bd-uni-X-psi}
Let $(\vartheta_s)_{0\,\leq\, s\,\leq\, 1}$ be a path of solutions in $\mathbb{R}\oplus C^{\infty}_{X,\,\beta}(M)$ to \eqref{starstar-s}. Then there exists a positive constant $C=C\left(n,\omega,\sup_{s\in[0,1]}\|G_s\|_{C^0}\right)$ such that
\begin{equation*}
\sup_{0\,\leq\, s\,\leq\, 1}\sup_M X\cdot\vartheta_s\leq C.
\end{equation*}
In particular, $X\cdot\vartheta_{s}<C$ for all $s\in[0,\,1]$.
\end{prop}

\begin{proof}
Our proof is based on that of Siepmann in the case of an expanding gradient K\"ahler-Ricci soliton; see \cite[Lemma 5.4.14]{siepmann}.
We adapt his proof here to our particular setting.

We begin with Claim \ref{commute} which gives
\begin{equation}\label{est-sec-der-vec-fiel}
X\cdot X\cdot\vartheta_{s}=2i\partial\bar{\partial}\vartheta_{s}(X,\,JX)=2\left(\sigma_{s}(X,\,JX)-\omega_{s}(X,\,JX)\right)\geq-2\omega_{s}(X,\,JX)=-2|X|^{2}_{g_{s}}.
\end{equation}
To get an upper bound for $X\cdot\vartheta_{s}$, we introduce the flow $(\varphi^{X}_{t})_{t\in\mathbb{R}}$ generated by the vector field $\frac{X}{2}$.
This flow is complete since $X$ grows linearly at infinity. Define $\vartheta^{s}_x(t):=\vartheta_{s}(\varphi^{X}_{t}(x))$ for $(x,\,t)\in M\times\mathbb{R}.$ Then for any cut-off function $\eta:\mathbb{R}_+\rightarrow[0,\,1]$ such that $\eta(0)=1$ and $\eta'(0)=0$ we have that
\begin{eqnarray*}
\int_0^{+\infty}\eta''(t)\vartheta^{s}_x(t)dt&=&-\int_0^{+\infty}\eta'(t)(\vartheta^{s}_x)'(t)dt\\
&=&(\vartheta^{s}_x)'(0)+\int_0^{+\infty}\eta(t)(\vartheta^{s}_x)''(t)dt.
\end{eqnarray*}
Using \eqref{est-sec-der-vec-fiel}, it then follows that
\begin{equation*}
\begin{split}
\frac{X}{2}\cdot\vartheta_{s}(x)&=(\vartheta^{s}_x)'(0)\leq-\int_{\supp(\eta)}\frac{X}{2}\cdot \left(\frac{X}{2}\cdot \vartheta_{s}\right)(\varphi^{X}_{t}(x))\,dt+\sup_{t\,\in\,\supp(\eta'')}\arrowvert\vartheta^{s}_x(t)\arrowvert\int_{\supp(\eta'')}\arrowvert\eta''(t)\arrowvert\,dt\\
&\leq\frac{1}{2}\int_{\supp(\eta)}\arrowvert X\arrowvert^2_{g_{s}}(\varphi^{X}_{t}(x))\,dt+\sup_{t\,\in\,\supp(\eta'')}\arrowvert\vartheta_{s}(\varphi^{X}_{t}(x)\arrowvert\int_{\supp(\eta'')}\arrowvert\eta''(t)\arrowvert \,dt.\\
\end{split}
\end{equation*}

Choose $\eta$ such that $\supp(\eta)\subset [0,\,1]$ and let $x$ now be the point where $X\cdot\vartheta_{s}$ attains its maximum value. By
Lemma \ref{lemma-loc-crit-pts-rad-der}(i), we know that $x$ is contained in $W$. Hence, we deduce from the above that
\begin{equation*}
\frac{X}{2}\cdot\vartheta_{s}(x)\leq C\left(\sup_{s\,\in\,[0,\,1]}\left(\sup_{\cup_{t\in[0,\,1]}\varphi^{X}_{t}(W)}|X|^{2}_{g_{s}}\right)+\|\vartheta_{s}\|_{C^{0}}\right).
\end{equation*}
The result now follows from the uniform upper bound on $\|\vartheta_{s}\|_{C^{0}}$.
\end{proof}

\subsection{A priori estimates on higher derivatives}\label{sec-high-der}
We next derive a priori global bounds on higher derivatives of solutions to the complex Monge-Amp\`ere equation \eqref{starstar-s}, beginning with
the $C^{2}$-estimate. The a priori bounds we derive hold everywhere on the manifold $M$, not just on a given fixed compact subset.
The unboundedness of the vector field $X$ prevents us from applying standard local estimates to higher derivatives of solutions to \eqref{starstar-s}.

\subsubsection{$C^2$ a priori estimate}

\begin{prop}[A priori $C^2$-estimate]\label{prop-C^2-est}
Let $(\vartheta_s)_{0\,\leq\, s\,\leq\, 1}$ be a path of solutions in $\mathbb{R}\oplus C^{\infty}_{X,\,\beta}(M)$ to \eqref{starstar-s}. Then there exists a positive constant $C=C\left(n,\omega,\sup_{s\in[0,1]}\|G_s\|_{C^2}\right)$ such that the following $C^2$ a priori estimate holds true:
\begin{equation*}
\sup_{0\,\leq\, s\,\leq\, 1}\|i\partial\bar{\partial}\vartheta_s\|_{C^0}\leq C.
\end{equation*}
In particular,
\begin{equation*}
\sup_{0\,\leq\, s\,\leq\, 1}\|i\partial\bar{\partial}\psi_s\|_{C^0}\leq C.
\end{equation*}
\end{prop}

\begin{proof}
Following closely \cite[Proposition 6.6]{con-der} where the
approach taken is based on standard computations performed in Yau's seminal paper \cite[pp.347--351]{Calabiconj} (see also \cite[Lemma 5.4.16]{siepmann}
and \cite[pp.52--55]{Tian-Can-Met-Boo}), we let $\Delta_{s}$ denote the Laplacian with respect to $\sigma_{s}$ and first estimate the drift Laplacian $\Delta_{s}-\frac{X}{2}\cdot$ of $\tr_{\omega_s}\sigma_s$ to obtain
 \begin{equation}
 \begin{split}\label{C2-est-yau}
\left(\Delta_{s}-\frac{X}{2}\cdot\right) \tr_{\omega_s}\sigma_s&\geq \frac{(\vartheta_s)_{i\bar{\jmath}k}(\vartheta_s)_{\bar{\imath}j\bar{k}}}{(1+(\vartheta_s)_{i\bar{\imath}})(1+(\vartheta_s)_{k\bar{k}})}+\Delta_{s}G_{s}-C\tr_{\omega_s}\sigma_{s}\cdot \tr_{\sigma_s}\omega_{s}\cdot (1+\inf_M \Rm(g_s))\\
&\qquad-C(n,\omega).
\end{split}
\end{equation}

Let $u_{s}:=e^{-\lambda\vartheta_{s}}(n+\Delta_{s}\vartheta_{s})$, where $\lambda>0$ will be specified later. Then one estimates the drift Laplacian $\Delta_{s}-\frac{X}{2}\cdot$ of $u_{s}$ with respect to $\sigma_{s}$ in the following way using the fact that $\vartheta_{s}$ satisfies \eqref{starstar-s}:
 \begin{equation*}
 \begin{split}
\left(\Delta_{s}-\frac{X}{2}\cdot\right) u_{s}
&\geq e^{-\lambda\vartheta_{s}}\Delta_{s}G_{s}+e^{-\lambda\vartheta_{s}}g_{s}\left(\nabla^{s}\left(\frac{X}{2}\right),\,i\partial\bar\partial\vartheta_{s}\right)
-C_{s}n^{2}e^{-\lambda\vartheta_{s}}+\lambda\left(\frac{X}{2}\cdot\vartheta_{s}\right)u_{s}
-\lambda n u_{s}\\
&\qquad+(\lambda+C_{s})e^{\frac{\lambda\vartheta_{s}-G_{s}-\frac{X}{2}\cdot\vartheta_{s}}{n-1}}u_{s}^{\frac{n}{n-1}},
\end{split}
\end{equation*}
where $\nabla^{s}$ is the Levi-Civita connection of $g_{s}$ and $C_{s}:=\inf_{i\,\neq\,k}\operatorname{Rm}^{s}_{i\bar{\imath}k\bar{k}}$,
$\operatorname{Rm}^{s}$ here denoting the complex linear extension of the curvature operator of the metric $g_{s}$. As $C_{s}$ is uniformly bounded
below in $s$ by a constant $A$ (which we may assume is $\leq1$), we may choose $\lambda>0$ sufficiently large so that $\lambda+A=1$. Moreover, as
$$\left|g_{s}\left(\nabla^{s}\left(\frac{X}{2}\right),\,i\partial\bar\partial\vartheta_{s}\right)\right|\leq C\|\nabla^{s}X\|_{C^{0}}(1+u)$$
for some generic constant $C>0$, we deduce that $u$ satisfies the following differential inequality:
\begin{equation*}
\left(\Delta_{s}-\frac{X}{2}\cdot\right) u_{s}\geq-C_{1}(1+u_{s})+C_{2}u_{s}^{\frac{n}{n-1}},
\end{equation*}
where $C_{1}$ and $C_{2}$ depend only on $n$, $A$, $\sup_{s\,\in\,[0,\,1]}\|\vartheta_{s}\|_{C^{0}}$, $\sup_{s\,\in\,[0,\,1]}\|X\cdot\vartheta_{s}\|_{C^{0}}$,
$\sup_{s\,\in\,[0,\,1]}\|G_{s}\|_{C^{2}}$, and $\sup_{s\,\in\,[0,\,1]}\|\nabla^{s}X\|_{C^{0}}$. The combination of Propositions \ref{lowerbound}, \ref{prop-bd-abo-uni-psi}, \ref{prop-bd-bel-uni-psi}, and \ref{prop-bd-uni-X-psi} shows that $C_{1}$ and $C_{2}$ depend only on $n$, $A$ and $\sup_{s\,\in\,[0,\,1]}\|G_{s}\|_{C^{2}}$.

Since $u_{s}$ is non-negative and converges to $n$ at infinity as $\vartheta_{s}\in \mathbb{R}\oplus C^{\infty}_{X,\,\beta}(M)$,
an application of the maximum principle to an exhausting sequence of domains of $M$ yields an upper bound for $n+\Delta_{s}\vartheta_{s}$
and consequently, the desired bound on $i\partial\bar{\partial}\vartheta_{s}$.
\end{proof}

A useful consequence of Proposition \ref{prop-C^2-est} is that the K\"ahler metrics
induced by $\sigma_{s}$ and $\omega_{s}$ are uniformly equivalent.
\begin{corollary}\label{coro-equiv-metrics-0}
Let $(\vartheta_s)_{0\,\leq\, s\,\leq\, 1}$ be a path of solutions in $\mathbb{R}\oplus C^{\infty}_{X,\,\beta}(M)$ to \eqref{starstar-s} and
for $s\in[0,\,1]$, let $g_{s},\,h_{s}$ denote the K\"ahler metrics induced by $\omega_{s},\,\sigma_{s}$ respectively. Then the tensors
$g_{s}^{-1}h_{s}$ and $h_{s}^{-1}g_{s}$ satisfy the following uniform estimate:
\begin{equation*}
\sup_{0\,\leq\, t\,\leq\, 1}\|g_{s}^{-1}h_{s}\|_{C^{0}}+\sup_{0\,\leq\, t\,\leq\, 1}\|h_{s}^{-1}g_{s}\|_{C^{0}}\leq C
\end{equation*}
for some positive constant $C=C\left(n,\omega,\sup_{s\in[0,1]}\|G_s\|_{C^2}\right)$.
In particular, the metrics $g$ and $(h_{s})_{0\,\leq\, s\,\leq\, 1}$ are uniformly equivalent.
\end{corollary}

\begin{proof}
The estimate follows as in \cite[Corollary 7.15]{conlon33} using {Propositions \ref{lowerbound}, \ref{prop-bd-uni-X-psi}, and \ref{prop-C^2-est}}. The fact that $\omega$ and $\sigma_{s}$ differ by
a $(1,\,1)$-form whose norm is controlled uniformly in $s$ yields the last claim of the corollary.
\end{proof}

\subsubsection{$C^3$ a priori estimate}
We now present the $C^{3}$-estimate.
\begin{prop}[A priori $C^3$-estimate]\label{prop-C^3-est}
Let $(\vartheta_s)_{0\,\leq\, s\,\leq\, 1}$ be a path of solutions in $\mathbb{R}\oplus C^{\infty}_{X,\,\beta}(M)$ to \eqref{starstar-s} and let
$g_{s}$ be the K\"ahler metric induced by $\omega_{s}$ with Levi-Civita connection $\nabla^{g_{s}}$. Then
\begin{equation*}
\sup_{0\,\leq\, s\,\leq\, 1}\|\nabla^{g_{s}}\partial\bar{\partial}\vartheta_s\|_{C^0}\leq C\left(n,\omega,\sup_{s\in[0,1]}\| G_s\|_{C^3}\right).
\end{equation*}
In particular,
\begin{equation}\label{a-priori-nabla-rad-der}
\sup_{0\,\leq\, s\,\leq\, 1}\|\nabla^{g_{s}}\left(X\cdot\vartheta_s\right)\|_{C^0}\leq C\left(n,\omega,\sup_{s\in[0,1]}\| G_s\|_{C^3}\right).
\end{equation}
\end{prop}

\begin{proof}
We follow closely the proof given in \cite[Proposition 6.9]{con-der} which itself is based on \cite{Pho-Ses-Stu}.

Set
$$S(h_{s},\,g_{s}):=\arrowvert\nabla^{g_{s}}h_{s}\arrowvert^2_{h_{s}}.$$
Then from the definition of $S$, we see that
\begin{equation*}
\begin{split}
S(h_{s},\,g_{s})=&h_{s}^{i\bar{\jmath}}h_{s}^{k\bar{l}}h_{s}^{p\bar{q}}\nabla^{g_{s}}_i(h_{s})_{kp}\overline{\nabla^{g_{s}}_{j}(h_{s})_{lq}}\\
=&|\Psi|_{h_{s}}^2,
\end{split}
\end{equation*}
where
\begin{equation*}
\begin{split}
\Psi_{ij}^k(h_{s},\,g_{s})&:=\Gamma(h_{s})_{ij}^k-\Gamma(g_{s})_{ij}^k\\
&=h_{s}^{k\bar{l}}\nabla^{g_{s}}_i(h_{s})_{j\bar{l}}.
\end{split}
\end{equation*}
Now, since $\vartheta_{s}$ solves \eqref{starstar-s}, $(M,\,h_{s},\,X)$
is an ``approximate'' steady gradient K\"ahler-Ricci soliton in the following precise sense:
if $h_{s}(t):=(\varphi^{X}_{t})^*h_{s}$ and $g_{s}(t):=(\varphi^{X}_{t})^*g_{s}$,
where $(\varphi^{X}_{t})_{t\,\in\,\R}$ is the
one-parameter family of diffeomorphisms generated by $\frac{X}{2}$,
then $(h_{s}(t))_{t\,\in\,\R}$
is a solution of the following perturbed K\"ahler-Ricci flow with initial condition $h_{s}$:
\begin{equation*}
\begin{split}
\partial_{t}h_{s}(t)&=-\Ric(h_{s}(t))+(\varphi^{X}_{t})^*\left(\mathcal{L}_{\frac{X}{2}}g_{s}+\Ric(g_{s})+\nabla^{g_{s}}\bar{\nabla}^{g_{s}}G_{s}\right),\qquad t\in\R,\\
h_{s}(0)&=h_{s}.
 \end{split}
\end{equation*}
In particular, $\partial_{t}h_{s}=-\Ric(h_{s})
+(\varphi^{X}_{t})^*\Lambda$, where $\Lambda:=
\mathcal{L}_{\frac{X}{2}}g_{s}+\Ric(g_{s})+\nabla^{g_{s}}\bar{\nabla}^{g_{s}}G_{s}$ has uniformly controlled $C^1$-norm
as $g_{s}$ is isometric to $g$ and $G_{s}$ is equal to zero, all outside a compact set independent of $s$.

Define $S(t):=S(h_{s}(s),\,g_{s}(t))$ and correspondingly set $\Psi(t):=\Psi(h_{s}(t),\,g_{s}(t))$. We adapt \cite[Proposition 3.2.8]{Bou-Eys-Gue} to our setting. By a
brute force computation, we have that
\begin{equation*}
\begin{split}
\Delta_{\sigma_{s}}S&=2\Re\left(h_{s}^{i\bar{\jmath}}h_{s}^{p\bar{q}}(h_{s})_{k\bar{l}}\left(\Delta_{\sigma_{s},\,1/2}\Psi_{ip}^k\right)
\overline{\Psi_{jq}^l}\right)+|\nabla^{h_{s}} \Psi|^2_{h_{s}}+|\overline{\nabla}^{h_{s}}\Psi|_{h_{s}}^2\\
&\qquad+\Ric(h_{s})^{i\bar{\jmath}}h_{s}^{p\bar{q}}(h_{s})_{k\bar{l}}\Psi_{ip}^k\overline{\Psi_{jq}^l}
+h_{s}^{i\bar{\jmath}}\Ric(h_{s})^{p\bar{q}}(h_{s})_{k\bar{l}}\Psi_{ip}^k\overline{\Psi_{jq}^l}-h_{s}^{i\bar{\jmath}}h_{s}^{p\bar{q}}\Ric(h_{s})_{k\bar{l}}\Psi_{ip}^k\overline{\Psi_{jq}^l},
\end{split}
\end{equation*}
where
\begin{equation*}
\begin{split}
&\Delta_{\sigma_{s},\,1/2}:=h_{s}^{i\bar{\jmath}}\nabla^{h_{s}}_i\nabla^{h_{s}}_{\bar{\jmath}},\label{def-lap-half}\\
&T^{i\bar{\jmath}}:=h_{s}^{i\bar{k}}h_{s}^{l\bar{\jmath}}T_{k\bar{l}},
\end{split}
\end{equation*}
for $T_{k\bar{l}}\in\Lambda^{1,\,0}M\otimes\Lambda^{0,\,1}M$. We also have that
\begin{equation*}
\begin{split}
\partial_{u}\Psi(u)_{ip}^k|_{u\,=\,0}&=\partial_{u}|_{u\,=\,0}(\Gamma(h_{s}(u))-\Gamma(g_{s}(u)))_{ip}^k\\
&=\nabla^{h_{s}}_i(-\Ric(h_{s})_p^k+\Lambda_p^k)-\nabla^{g_{s}}_i(\mathcal{L}_{\frac{X}{2}}(g_{s})_{p}^{k}),\\
\partial_{u}h_{s}^{i\bar{\jmath}}|_{u\,=\,0}&=\Ric(h_{s})^{i\bar{\jmath}}-\Lambda^{i\bar{\jmath}}.
\end{split}
\end{equation*}
Finally, using the second Bianchi identity, we compute that
\begin{equation*}
\Delta_{\sigma_{s},\,1/2}\Psi_{ip}^k=h_{s}^{a\bar{b}}\nabla_a^{h_{s}}\Rm(g_{s})_{i\bar{b}p}^k-\nabla^{h_{s}}_i\Ric(h_{s})_p^k,
\end{equation*}
which in turn implies that the following evolution equation is satisfied by $\Psi$:
\begin{equation*}
{\partial_{u}\Psi_{ip}^k(u)|_{u\,=\,0}}=\Delta_{\sigma_{s},\,1/2}\Psi_{ip}^k+T_{ip}^k,
\end{equation*}
for a tensor $T$ of the form
\begin{equation*}
\begin{split}
T&=h_{s}^{-1}\ast\nabla^{h_{s}}\Rm(g_{s})+\nabla^{h_{s}}\Lambda-\nabla^{g_{s}}(\mathcal{L}_{\frac{X}{2}}g_{s})\\
&=h_{s}^{-1}\ast\nabla^{g_{s}}\Rm(g_{s})+h_{s}^{-1}\ast h_{s}^{-1}\ast\Rm(g_{s})\ast\Psi+h_{s}^{-1}\ast\Psi\ast \Lambda+\nabla^{g_{s}}(\Lambda-\mathcal{L}_{\frac{X}{2}}g_{s}).
\end{split}
\end{equation*}
Notice the simplification here regarding the ``bad'' term $-\nabla^{h_{s}}\Ric(h_{s})$.
Since this flow is evolving only by diffeomorphism, we know that
\begin{equation*}
\begin{split}
S(t)&=(\varphi^{X}_{t})^*S(h_{s},\,g_{s}),\\
\partial_{u}S|_{u\,=\,0}&=\frac{X}{2}\cdot S(h_{s},\,g_{s}).
\end{split}
\end{equation*}
Hence Young's inequality, together with
the boundedness of $\|h_{s}^{-1}g_{s}\|_{C^0}$ and $\|h_{s}g_{s}^{-1}\|_{C^0}$ ensured by Corollary \ref{coro-equiv-metrics-0}
and the boundedness of the covariant derivatives of the tensors $\Rm(g_{s})$ and $\Lambda$, imply that
\begin{equation*}
\Delta_{\sigma_{s}}S-\frac{X}{2}\cdot S\geq -C(S+1)
\end{equation*}
for some positive uniform constant $C$.

We use as a barrier function the trace $\tr_{\omega_{s}}\sigma_{s}$ which, by \eqref{C2-est-yau} and the uniform equivalence of the metrics $g_{s}$ and $h_{s}$ provided by Corollary \ref{coro-equiv-metrics-0}, satisfies
\begin{equation*}
\Delta_{\sigma_{s}}\tr_{\omega_{s}}\sigma_{s}-\frac{X}{2}\cdot \tr_{\omega_{s}}\sigma_{s}\geq C^{-1}S-C,
\end{equation*}
where $C$ is a uniform positive constant that may vary from line to line. By applying the maximum principle to $\varepsilon S+\tr_{\omega_{s}}\sigma_{s}$ for some sufficiently small $\varepsilon>0$, one arrives at the desired a priori estimate.

The proof of \eqref{a-priori-nabla-rad-der} is a consequence of the previously proved a priori bound on $\nabla^{g_s}\partial\overline{\partial}\vartheta_s$, once we differentiate \eqref{starstar-s}.
\end{proof}

We next establish H\"older regularity of $g_{s}^{-1}h_{s}$ and $h^{-1}_{s}g_{s}$,
an improvement on Corollary \ref{coro-equiv-metrics-0}.

\begin{corollary}\label{coro-equiv-metrics}
Let $(\vartheta_{s})_{0\,\leq\, s\,\leq\, 1}$ be a path of solutions in $\mathbb{R}\oplus C^{\infty}_{X,\,\beta}(M)$ to \eqref{starstar-s}
and for $s\in[0,\,1]$, let $h_{s}$ be the K\"ahler metric induced by $\sigma_{s}$.
Then for any $\alpha\in\left(0,\,\frac{1}{2}\right)$, the tensors $g_{s}^{-1}h_{s}$ and $h_{s}^{-1}g_{s}$ satisfy the following uniform estimate:
 \begin{equation*}
\sup_{0\,\leq\, s\,\leq\, 1}\left(\|g_{s}^{-1}h_{s}\|_{C_{\operatorname{\operatorname{\operatorname{loc}}}}^{0,\,2\alpha}}+\|h_{s}^{-1}g_{s}
\|_{C_{\operatorname{\operatorname{\operatorname{loc}}}}^{0,\,2\alpha}}\right)\leq C\left(n,\alpha,\omega,\sup_{s\in[0,1]}\| G_s\|_{C^{3}}\right).
\end{equation*}
\end{corollary}

 \begin{proof}
By standard local interpolation inequalities applied to Propositions \ref{prop-C^2-est} and \ref{prop-C^3-est}, we see that
\begin{equation*}
\|g_{s}^{-1}h_{s}\|_{C_{\operatorname{\operatorname{\operatorname{loc}}}}^{0,\,2\alpha}}\leq C\left(n,\alpha,\omega,\sup_{s\in[0,1]}\| G_s\|_{C^{3}}\right).
\end{equation*}
Combining the previous estimate with Corollary \ref{coro-equiv-metrics-0}, it suffices to prove a uniform bound on the local $2\alpha$-H\"older norm of $h_{s}^{-1}g_{s}$.
We conclude with the following observation: if $u$ is a positive function on $M$ in $C_{\operatorname{\operatorname{\operatorname{loc}}}}^{0,2\alpha}(M)$ uniformly bounded from below by a positive constant, then $[u^{-1}]_{2\alpha}\leq [u]_{2\alpha}(\inf_Mu)^{-2}$. By invoking Corollary \ref{coro-equiv-metrics-0} once more, this last remark applied to $h_{s}^{-1}g_{s}$ implies that
\begin{equation*}
\|h_{s}^{-1}g_{s}\|_{C_{\operatorname{\operatorname{\operatorname{loc}}}}^{0,\,2\alpha}}\leq C\left(n,\alpha,\omega,\sup_{s\in[0,1]}\| G_s\|_{C^{3}}\right)
\end{equation*}
as well.
\end{proof}

\subsubsection{Local bootstrapping}
We now improve the local regularity of our continuity path of solutions to \eqref{starstar-s}. This estimate will be used
in deriving the subsequent weighted a priori estimates.

\begin{prop}\label{prop-loc-holder-C-3}
Let $(\vartheta_s)_{0\,\leq\, s\,\leq\, 1}$ be a path of solutions in $\mathbb{R}\oplus C^{\infty}_{X,\,\beta}(M)$
to \eqref{starstar-s}. Then for any $\alpha\in\left(0,\,\frac{1}{2}\right)$ and for any compact subset $K\subset M$,
\begin{equation*}
\sup_{0\,\leq\, s\,\leq\, 1}\|\vartheta_{s}\|_{C^{3,\,2\alpha}(K)}\leq C\left(n,\alpha,\omega, \sup_{s\in[0,1]}\| G_s\|_{C^{3}},K\right).
\end{equation*}
\end{prop}

\begin{proof}
From the standard computations involved in the proof of the a priori $C^2$-estimate, we derive that
\begin{equation}\label{weloveele}
\begin{split}
\Delta_{\sigma_{s}}\left(\Delta_{\omega_{s}}\vartheta_{s}-\frac{X}{2}\cdot\vartheta_{s}\right) =& \Delta_{\sigma_{s}}G_{s}+h_{s}^{-1}\ast g_{s}^{-1}\ast\Rm(g_{s})+\Rm(g_{s})\ast\nabla^{h_{s}}\bar{\nabla}^{h_{s}}\vartheta_{s}\ast h_{s}^{-1}\\
&\quad+g_{s}^{-1}\ast g_{s}^{-1}\ast\Rm(g_{s})+g_{s}^{-1}\ast h_{s}^{-1}\ast h_{s}^{-1}\ast \bar{\nabla}^{h_{s}}\nabla^{h_{s}}\bar{\nabla}^{h_{s}}\vartheta_{s}\ast\nabla^{h_{s}}\bar{\nabla}^{h_{s}}\nabla^{h_{s}} \vartheta_{s}\\
&\quad-\left(\Delta_{\sigma_{s}}-\Delta_{\omega_{s}}\right)\left(\frac{X\cdot \vartheta_{s}}{2}\right),
\end{split}
\end{equation}
where $\ast$ denotes the ordinary contraction of two tensors. Now, since $X$ is real holomorphic and $\vartheta_s$ being $JX$-invariant,
we see that
\begin{equation}\label{lie-der-cov-der-X}
i\partial\overline{\partial}(X\cdot\vartheta_s)=\mathcal{L}_X(i\partial\overline{\partial}\vartheta_s)=\nabla^{g_s}_X(i\partial\overline{\partial}\vartheta_s)+i\partial\overline{\partial}\vartheta_s\ast\nabla^{g_s}X.
\end{equation}
Therefore, thanks to \eqref{lie-der-cov-der-X}, we have the following pointwise estimate:
\begin{equation}
\begin{split}\label{easy-obs-diff-lap}
\left|\left(\Delta_{\sigma_{s}}-\Delta_{\omega_{s}}\right)(X\cdot\vartheta_{s})\right|&
=\left|h_{s}^{-1}\ast i\partial\bar{\partial}\vartheta_{s}\ast i\partial\bar{\partial}(X\cdot\vartheta_{s})\right|_{g_{s}}\\
&\leq |h_{s}^{-1}g_{s}|_{g_{s}}\cdot|i\partial\bar{\partial}\vartheta_{s}|_{g_{s}}\cdot \left(
|i\partial\bar{\partial}\vartheta_{s}|_{g_{s}}|\nabla^{g_{s}}X|_{g_{s}}+|\nabla^{g_s}i\partial\bar{\partial}\vartheta_{s}|_{g_{s}}|X|_{g_{s}}\right).
\end{split}
\end{equation}

By Propositions \ref{prop-C^2-est} and \ref{prop-C^3-est} together with \eqref{easy-obs-diff-lap},
the $C^0$-norm of the right-hand side of \eqref{weloveele} is uniformly bounded on compact subsets and, thanks to Corollary \ref{coro-equiv-metrics}, so too are the coefficients of $\Delta_{\sigma_{s}}$ in the $C^{0,\,2\alpha}_{\operatorname{\operatorname{\operatorname{loc}}}}$-sense. As a result, by applying the Morrey-Schauder $C^{1,\,2\alpha}$-estimates, we see that for any $x\in M$ and for $\delta<\inj_{g_{s}}(M)$,
\begin{equation*}
\left\|\Delta_{\omega_{s}}\vartheta_{s}-\frac{X}{2}\cdot\vartheta_{s}\right\|_{C^{1,\,2\alpha}(B_{g_{s}}(x,\,\delta))}\leq C(x,\,\delta,\,\alpha).
\end{equation*}
Finally, applying standard interior Schauder estimates for elliptic equations once again with respect to $\Delta_{\omega_{s},\,X}$, we deduce that
\begin{equation*}
\begin{split}
\|\vartheta_{s}\|_{C^{3,\,2\alpha}(B_{g_{s}}(x,\,\frac{\delta}{2}))}&\leq C(x,\,\delta,\,\alpha)\left(\left\|\Delta_{\omega_{s}}\vartheta_{s}-\frac{X}{2}\cdot\vartheta_{s}\right\|_{C^{1,\,2\alpha}(B_{g_{s}}(x,\,\delta))}
+\|\vartheta_{s}\|_{C^{1,\,2\alpha}(B_{g_{s}}(x,\,\delta))}\right)\\
&\leq C(x,\,\delta,\,\alpha).
\end{split}
\end{equation*}
\end{proof}

We next establish the following well-known local regularity result for solutions to \eqref{starstar-s}.
\begin{prop}\label{prop-loc-reg}
Let $G_{s}\in C^{k,\,\alpha}_{\operatorname{\operatorname{\operatorname{loc}}}}(M)$ for some $k\geq1$ and $\alpha\in(0,\,1)$ and let $\vartheta_{s}\in C^{3,\,\alpha}_{\operatorname{\operatorname{\operatorname{loc}}}}(M)$ be a solution to \eqref{starstar-s} with data $G_{s}$. Then $\vartheta_{s}\in C^{k+2,\alpha}_{\operatorname{\operatorname{\operatorname{loc}}}}(M)$. Moreover, for all $k\geq 1$, $\alpha\in(0,\,1)$, and compact subset $K\subset M$,
\begin{equation*}
\begin{split}
\|\vartheta_{s}\|_{C^{k+2,\alpha}(K)}\leq C\left(n,\alpha,\omega, \sup_{s\in[0,1]}\| G_s\|_{C^{\max\{k,3\},\alpha}},K\right).
\end{split}
\end{equation*}
\end{prop}

\begin{proof}
We prove this proposition by induction on $k\geq 1$. The case $k=1$ is true by Proposition \ref{prop-loc-holder-C-3}, so let $G_{s}\in C^{k+1,\,\alpha}_{\operatorname{\operatorname{\operatorname{loc}}}}(M)$ and let $\vartheta_{s}\in C^{3,\,\alpha}_{\operatorname{\operatorname{\operatorname{loc}}}}(M)$ be a solution of \eqref{starstar-s}. Then by induction, $\vartheta_{s}\in C^{k+2,\alpha}_{\operatorname{\operatorname{\operatorname{loc}}}}(M)$.
Let $x\in M$ and choose local holomorphic coordinates defined on $B_{g_{s}}(x,\,\delta)$ for some $0<\delta<\inj_{g_{s}}(M)$. Then since $\vartheta_{s}$ satisfies
\begin{equation*}
G_{s}=\log\left(\frac{\sigma_{s}^n}{\omega_{s}^n}\right)-\frac{X}{2}\cdot\vartheta_{s},
\end{equation*}
we know that for $j=1,...,2n$, the derivative $\partial_j\vartheta_{s}$ satisfies
\begin{eqnarray*}
\Delta_{\sigma_{s}}\left(\partial_j\vartheta_{s}\right)=\partial_j\left(G_{s}+\frac{X}{2}\cdot\vartheta_{s}\right)\in C^{k,\,\alpha}_{\operatorname{\operatorname{\operatorname{loc}}}}(M).
\end{eqnarray*}
As the coefficients of $\Delta_{\sigma_{s}}$ are in $C^{k,\,\alpha}_{\operatorname{\operatorname{\operatorname{loc}}}}(M)$,
an application of the standard interior Schauder estimates for elliptic equations now gives us the
desired local regularity result, namely $\partial_j\vartheta_{s}\in C^{k+2,\alpha}_{\operatorname{\operatorname{\operatorname{loc}}}}(M)$ for all $j=1,...,2n$, or equivalently, $\vartheta_{s}\in C^{k+3,\alpha}_{\operatorname{\operatorname{\operatorname{loc}}}}(M)$ together with the expected estimate.
\end{proof}

\subsection{Weighted a priori estimates}\label{sec-wei-bd}

Our first proposition establishes an a priori decay estimate on the gradient of the $X$-derivative of solutions to \eqref{starstar-s}. Its proof uses the Bochner formula in an essential way.

\begin{prop}\label{decay-first-der-rad-der}
Let $(\vartheta_s)_{0\,\leq\, s\,\leq\, 1}$ be a path of solutions in $\mathbb{R}\oplus C^{\infty}_{X,\,\beta}(M)$ to \eqref{starstar-s}.
Then there exist positive constants $C$, $R_0$, and $\varepsilon>0$ such that for all $s\in[0,\,1]$,
\begin{equation*}
\left|\nabla^{g}\left(X\cdot \vartheta_s\right)\right|_{g}\leq \frac{C}{f^{\varepsilon}},\qquad f\geq R_0.
\end{equation*}
\end{prop}

\begin{proof}
Let $u:=X\cdot\vartheta_s$, write $\Delta_{h_{s},\,X}:=\Delta_{h_{s}}-X\cdot $ where $\Delta_{h_{s}}$ denotes the Riemannian Laplacian with respect to $h_{s}$, and recall from \eqref{lovely-eqn-der-rad} the differential equation satisfied by $u$ outside a sufficiently large compact set $W$ of $M$:
\begin{equation}\label{lovely}
\frac{1}{2}\Delta_{h_{s},\,X}u=2e^{-\frac{X\cdot\vartheta_s}{2}}\frac{(\omega_{D}+i\partial\bar{\partial}\vartheta_{s})^{n}}{\omega^{n}}.
\end{equation}
Applying the Bochner formula for the drift Laplacian to the function $u$, we obtain
\begin{equation*}
\begin{split}
\frac{1}{2}\Delta_{h_{s},\,X}|\nabla^{h_{s}}u|_{h_{s}}^{2}&=|\operatorname{Hess}_{h_{s}}(u)|_{h_{s}}^{2}+\operatorname{Ric}(h_{s})(\nabla^{h_{s}}u,\,\nabla^{h_{s}}u)+\operatorname{Hess}_{h_{s}}(f_{\sigma_{s}})(\nabla^{h_{s}}u,\,\nabla^{h_{s}}u)\\
&\qquad+\langle\nabla^{h_{s}}\Delta_{h_{s},\,X}u,\,\nabla^{h_{s}}u\rangle_{h_{s}}\\
&=|\operatorname{Hess}_{h_{s}}(u)|_{h_{s}}^{2}+\operatorname{Ric}(g_{s})(\nabla^{h_{s}}u,\,\nabla^{h_{s}}u)+\operatorname{Hess}_{g_{s}}(f_{\omega_{s}})(\nabla^{h_{s}}u,\,\nabla^{h_{s}}u)\\
&\qquad-i\partial\bar{\partial}G_{s}(\nabla^{h_{s}}u,\,\nabla^{h_{s}}u)+4\left\langle\nabla^{h_{s}}\left(e^{-\frac{X\cdot\vartheta_s}{2}}\frac{(\omega_{D}+i\partial\bar{\partial}\vartheta_{s})^{n}}{\omega^{n}}\right),\,\nabla^{h_{s}}u\right\rangle_{h_{s}},\\
\end{split}
\end{equation*}
where we have used \eqref{wtf} and \eqref{lovely} in the second equality. As $G_{s}$ is supported in $W$ and $g_{s}$ is isometric to $g$ on $M\setminus W$, on this latter set this equation reads as
\begin{equation*}
\begin{split}
\frac{1}{2}\Delta_{h_{s},\,X}|\nabla^{h_{s}}u|_{h_{s}}^{2}&=|\operatorname{Hess}_{h_{s}}(u)|_{h_{s}}^{2}+\operatorname{Ric}(g)(\nabla^{h_{s}}u,\,\nabla^{h_{s}}u)+\operatorname{Hess}_{g}(f)(\nabla^{h_{s}}u,\,\nabla^{h_{s}}u)\\
&\qquad+4\left\langle\nabla^{h_{s}}\left(e^{-\frac{X\cdot\vartheta_s}{2}}\frac{(\omega_{D}+i\partial\bar{\partial}\vartheta_{s})^{n}}{\omega^{n}}\right),\,\nabla^{h_{s}}u\right\rangle_{h_{s}}
\end{split}
\end{equation*}
which, using the properties of $g$, then becomes
\begin{equation}\label{bochnerr}
\Delta_{h_{s},\,X}|\nabla^{h_{s}}u|_{h_{s}}^{2}=2|\operatorname{Hess}_{h_{s}}(u)|_{h_{s}}^{2}+2|\nabla^{h_{s}}u|^{2}_{g}+8\left\langle\nabla^{h_{s}}\left(e^{-\frac{X\cdot\vartheta_s}{2}}\frac{(\omega_{D}+i\partial\bar{\partial}\vartheta_{s})^{n}}{\omega^{n}}\right),\,\nabla^{h_{s}}u\right\rangle_{h_{s}}
\end{equation}
on $M\setminus W$. Henceforth working on $M\setminus W$, we analyse the last term of this equation in the following claim.

\begin{claim}\label{claim-bd-boc-rad-der}
On $M\setminus W$, we have that
\begin{equation*}
\left|\left\langle\nabla^{h_{s}}\left(e^{-\frac{X\cdot\vartheta_s}{2}}\frac{(\omega_{D}+i\partial\bar{\partial}\vartheta_{s})^{n}}{\omega^{n}}\right),\,\nabla^{h_{s}}u\right\rangle_{h_{s}}\right|\leq \frac{C}{r}\left(|\operatorname{Hess}_{h_{s}}(u)|_{h_{s}}+|\nabla^{h_{s}}u|_{h_{s}}\right)|\nabla^{h_{s}}u|_{h_{s}}.
\end{equation*}
\end{claim}

\begin{proof}[Proof of Claim \ref{claim-bd-boc-rad-der}]
By the pointwise Cauchy-Schwarz inequality together with the a priori $C^2$ estimate from Proposition \ref{prop-C^2-est}, it suffices to prove that on $M\setminus W$,
\begin{equation*}
\left|\nabla^{g}\left(e^{-\frac{X\cdot\vartheta_s}{2}}\frac{(\omega_{D}+i\partial\bar{\partial}\vartheta_{s})^{n}}{\omega^{n}}\right)\right|_{g}\leq \frac{C}{r}\left(|\operatorname{Hess}_{h_{s}}(u)|_{h_{s}}+|\nabla^{h_{s}}u|_{h_{s}}\right).
\end{equation*}
Now, thanks to \eqref{lovely-eqn-der-rad-bis}, the a priori bounds on $X\cdot \vartheta_s$ (Propositions \ref{lowerbound} and \ref{prop-bd-uni-X-psi}) and its gradient  (Proposition \ref{prop-C^3-est}), one gets schematically:
\begin{equation*}
\begin{split}
\left|\nabla^{g}\left(e^{-\frac{X\cdot\vartheta_s}{2}}\frac{(\omega_{D}+i\partial\bar{\partial}\vartheta_{s})^{n}}{\omega^{n}}\right)\right|_{g}&\leq C\left(\frac{1}{r}|\nabla^gu|_g+\frac{1}{r^2}|\nabla^{g}u|^2_g+\frac{1}{r}|\operatorname{Hess}_{g}(u)|_{g}\right)\\
&\leq \frac{C}{r}\left(|\nabla^gu|_g+|\operatorname{Hess}_{g}(u)|_{g}\right),
\end{split}
\end{equation*}
where we have used implicitly the a priori $C^3$ bound (Proposition \ref{prop-C^3-est}). In order to conclude, it suffices to observe that
\begin{equation*}
\begin{split}
\left|\operatorname{Hess}_{h_s}(u)-\operatorname{Hess}_{g}(u)\right|_g&\leq C|\nabla^g\partial\overline{\partial}\vartheta_s|_g|\nabla^gu|_g\\
&\leq C|\nabla^gu|_g,
\end{split}
\end{equation*}
where $C$ is a positive constant independent of $s\in[0,1]$ that may vary from line to line. Here we have used Proposition \ref{prop-C^3-est} again in the last line.
\end{proof}

Combining \eqref{bochnerr} with Claim \ref{claim-bd-boc-rad-der} and
using Proposition \ref{coro-equiv-metrics-0} to deal with the term $|\nabla^{h_s}u|^{2}_{g}$ of \eqref{bochnerr}, all in all
we end up with the following differential inequality satisfied by $|\nabla^{h_{s}}u|^2_{h_s}$:
\begin{equation*}
\begin{split}
\Delta_{h_{s},\,X}|\nabla^{h_{s}}u|_{h_{s}}^{2}&\geq2|\operatorname{Hess}_{h_{s}}(u)|_{h_{s}}^{2}+C^{-1}|\nabla^{h_{s}}u|^{2}_{h_{s}}-\frac{C}{r}\left(|\operatorname{Hess}_{h_{s}}(u)|_{h_{s}}+|\nabla^{h_{s}}u|_{h_{s}}\right)|\nabla^{h_{s}}u|_{h_{s}}.
\end{split}
\end{equation*}
Next, upon applying Young's inequality, we derive that on the set $\{r>R\}$ for some $R>0$ with $W\subset\{r\leq R\}$ chosen sufficiently large,
\begin{equation}\label{inequ-nabla-u-ultimate}
\begin{split}
\Delta_{h_{s},\,X}|\nabla^{h_{s}}u|_{h_{s}}^{2}&\geq \frac{1}{2}C^{-1}|\nabla^{h_{s}}u|^{2}_{h_{s}}.\\
\end{split}
\end{equation}
Now, Lemma \ref{lemma-tr-star-star} ensures that $f_{\sigma_s}^{-\beta}$ for $\beta>0$ satisfies outside a sufficiently large uniform compact set of $M$ the differential inequality
\begin{equation*}
\begin{split}
\Delta_{h_s,X}f_{\sigma_s}^{-\beta}&=-\beta f_{\sigma_s}^{-\beta-1}\left(\Delta_{h_s,X}f_{\sigma_s}-(\beta+1)|X|^2_{h_s}f_{\sigma_s}^{-1}\right)\\
&=\beta\left(2f_{\sigma_s}-X\cdot\vartheta_{s}+(\beta+1)|X|^2_{h_s}f_{\sigma_s}^{-1}\right)f_{\sigma_s}^{-\beta-1}\\
&\leq 2\beta\left(1+Cf_{\sigma_s}^{-1}\right)f_{\sigma_s}^{-\beta}\leq 3\beta f_{\sigma_s}^{-\beta}
\end{split}
\end{equation*}
for some uniform positive constant $C$. Here we have used Proposition \ref{lowerbound} in the last line to bound $-X\cdot\vartheta_s$ uniformly from above. We have also used \eqref{a-priori-nabla-rad-der} from Proposition \ref{prop-C^3-est} to bound $|X|^2_{h_s}$ from above, since $2|X|^2_{h_s}=2X\cdot f_{\sigma_{s}}=2X\cdot f+X\cdot X\cdot \vartheta_s=r^2+O(r)$ where $O(\cdot)$ is uniform in $s\in[0,1]$. Recalling \eqref{inequ-nabla-u-ultimate}, one can then use $f_{\sigma_s}^{-\beta}$ for some $\beta>0$ to be specified as a barrier function. Indeed, if $A>0$, then outside a sufficiently large compact subset of $M$ we have that
\begin{equation}
\begin{split}\label{inequ-nabla-u-ultimate-bis}
\Delta_{h_s,X}\left(|\nabla^{h_s}u|^2_{h_s}-Af_{\sigma_s}^{-\beta}\right)&\geq \frac{1}{2}C^{-1}\left(|\nabla^{h_s}u|^2_{h_s}-Af_{\sigma_s}^{-\beta}\right)
\end{split}
\end{equation}
whenever $6\beta\leq C^{-1}.$ The maximum principle applied to \eqref{inequ-nabla-u-ultimate-bis} now yields the desired estimate.
\end{proof}

This leads to the following weighted estimate.

\begin{corollary}\label{coro-dec-met}
Let $(\vartheta_s)_{0\,\leq\, s\,\leq\, 1}$ be a path of solutions in $\mathbb{R}\oplus C^{\infty}_{X,\,\beta}(M)$ to \eqref{starstar-s}
and let $C$, $R_0$, and $\varepsilon>0$ be as in Proposition \ref{decay-first-der-rad-der}. Then
for all $s\in[0,\,1]$, there exists $\vartheta_{s}^{\infty}\in\R$ such that
\begin{equation*}
|\vartheta_s-\vartheta_s^{\infty}|+|X\cdot\vartheta_s|+\left|\nabla^{g}\vartheta_s\right|_{g}\leq \frac{C}{f^{\frac{\varepsilon}{2}}},\qquad f\geq R_0.
\end{equation*}
\end{corollary}

\begin{proof}
First observe that since $X=\nabla^{g}f$, for any vector field $Y$ on $M$ we have that
\begin{equation*}
\begin{split}
g(\nabla^{g}(X\cdot \vartheta_s),Y)&=\operatorname{Hess}_{g}(f)(\nabla^{g}\vartheta_{s},\,Y)+\operatorname{Hess}_{g}(\vartheta_{s})(X,\,Y)\\
&=\frac{1}{2}(\mathcal{L}_{X}g)(\nabla^g\vartheta_s,Y)+\operatorname{Hess}_{g}(\vartheta_{s})(X,\,Y).\\
\end{split}
\end{equation*}
In particular, upon setting $Y:=\nabla^g\vartheta_s$, using the $JX$-invariance of $\vartheta_{s}$ and
the fact that $\frac{X}{2}\cdot|\nabla^{g}\vartheta_{s}|^{2}_{g}=\operatorname{Hess}_{g}(\vartheta_{s})(X,\,\nabla^{g}\vartheta_{s})$
and $\frac{1}{2}\mathcal{L}_{X}g=g_{C}$ on $M\setminus W$, we see that on this set,
\begin{equation*}
\begin{split}
g(\nabla^{g}(X\cdot \vartheta_s),\nabla^g\vartheta_s)&=|\nabla^{C}\vartheta_s|_{g_{C}}^2+\frac{X}{2}\cdot|\nabla^g\vartheta_s|^2_g\\
&=r^{-2}\underbrace{|X\cdot\vartheta_s|^{2}}_{\leq\,C}+r^{-2}\underbrace{|JX\cdot\vartheta_{s}|^2}_{=\,0}+\frac{X}{2}\cdot|\nabla^g\vartheta_s|^2_g\\
&\leq\frac{C}{r^2}+\frac{X}{2}\cdot|\nabla^g\vartheta_s|^2_g,\\
\end{split}
\end{equation*}
where we have also used the boundedness of $|X\cdot\vartheta_s|$ given by Propositions \ref{lowerbound} and \ref{prop-bd-uni-X-psi} in the last line.
Therefore by Young's inequality together with Proposition \ref{decay-first-der-rad-der}, we find that
\begin{equation*}
\begin{split}
\frac{X}{2}\cdot|\nabla^g\vartheta_s|^2_g&\geq
-|\nabla^{g}(X\cdot \vartheta_s)|_{g}|\nabla^g\vartheta_s|_{g}-\frac{C}{r^2}\\
&\geq -\frac{C}{r^{2\varepsilon}}|\nabla^g\vartheta_s|_{g}-\frac{C}{r^2}\\
&\geq -\frac{C}{r^{2\varepsilon}}|\nabla^g\vartheta_s|_{g}^2-\frac{C}{r^{\min\{2\varepsilon,\,2\}}},
\end{split}
\end{equation*}
where $C$ is a positive constant that may vary from line to line. The previous differential inequality can be reformulated as follows:
\begin{equation*}
\partial_r\left(e^{-Cr^{-2\varepsilon}}|\nabla^g\vartheta_s|_{g}^2\right)\geq -\frac{Ce^{-Cr^{-2\varepsilon}}}{r^{1+\min\{2\varepsilon,\,2\}}}.
\end{equation*}
Integrating from $r$ to $r=+\infty$ and using the assumption that the covariant derivatives of $\vartheta_s$ decay to $0$ at infinity, we subsequently deduce that
\begin{equation*}
0\leq e^{-Cr^{-2\varepsilon}}|\nabla^g\vartheta_s|_{g}^2\leq C\int^{+\infty}_{r}s^{-1-\min\{2\varepsilon,\,2\}}e^{-Cs^{-2\varepsilon}}\,ds
\end{equation*}
so that
\begin{equation*}
0\leq|\nabla^g\vartheta_s|_{g}^2\leq Ce^{Cr^{-2\varepsilon}}\int^{+\infty}_{r}s^{-1-\min\{2\varepsilon,\,2\}}\underbrace{e^{-Cs^{-2\varepsilon}}}_{\leq\,1}\,ds
\leq Cr^{-\min\{2\varepsilon,\,2\}}e^{Cr^{-2\varepsilon}}.
\end{equation*}
As $e^{Cr^{-2\varepsilon}}$ is bounded at infinity, we arrive at the estimate $|\nabla^g\vartheta_s|_{g}\leq Cr^{-\min\{\varepsilon,\,1\}}$.

Next note from the mean value theorem on $D$ that at height $r$,
\begin{equation}
\begin{split}\label{osc-est-vartheta}
\left|\vartheta_s(r,\,\cdot)-\fint_{D}\vartheta_s(r,\,\cdot)\,\omega_{D}^{n-1}\right|&\leq \sup_{D\times\{r\}}|\nabla^{g}\vartheta_s|_{g}\diam_{g}D\leq \frac{C}{r^{\varepsilon}},
\end{split}
\end{equation}
and thanks to Proposition \ref{decay-first-der-rad-der} that
\begin{equation}
\begin{split}\label{osc-est-xvartheta}
\left|X\cdot\vartheta_s(r,\,\cdot)-\fint_{D}X\cdot\vartheta_s(r,\,\cdot)\,\omega_{D}^{n-1}\right|&\leq\frac{C}{r^{\varepsilon}}.
\end{split}
\end{equation}
These inequalities we will make use of later.

Linearising \eqref{starstar-s} around the background metric $g$ on $M\setminus W$, we can write
\begin{equation}\label{lin-CMA-infinity}
\Delta_{g,\,X}\vartheta_s=\int_0^1\int_0^{u}|\partial\bar{\partial}\vartheta_s|^2_{h_{s,\tau}}\,d\tau du,\qquad h_{s,\tau}:=(1-\tau)g+\tau h_s.
\end{equation}
Integrating over $D\times\{r\}$ then yields the equation
\begin{equation*}
\Delta_{C,X}\overline{\vartheta_s}(r)=\int_{D}\int_0^1\int_0^{u}|\partial\bar{\partial}\vartheta_s|^2_{h_{s,\tau}}\,d\tau du\,\omega_{D}^{n-1},
\end{equation*}
where recall that
$$\overline{\vartheta_s}(r):=\fint_{D\times\{r\}}\vartheta_{s}(r,\,\cdot)\,\omega_{D}^{n-1}.$$
By Corollary \ref{coro-equiv-metrics-0}, we therefore have that
\begin{equation}\label{inequ-mean-value-variant}
0\leq\Delta_{C,X}\overline{\vartheta_s}(r)\leq C\int_{D}|i\partial\bar{\partial}\vartheta_s|^2_{g}\,\omega_{D}^{n-1}
\end{equation}
for some uniform constant $C>0$.

Now, since $\nabla^gX=\nabla^{g,2}f=g_{C}$, one gets the following pointwise estimate obtained by considering an orthonormal frame of the form $(r^{-1}X,r^{-1}JX, (e_i,Je_i)_{1\leq i\leq n-1})$, where $(e_i,Je_i)_{1\leq i\leq n-1}$ is an orthonormal frame with respect to $g_D$:
\begin{equation*}
\begin{split}
|i\partial\bar{\partial}\vartheta_s|^2_{g}&\leq C|\nabla^{g,2}\vartheta_s|^2_g\\
&\leq C\left(r^{-2}|\nabla^g(X\cdot\vartheta_s)|^2_g+r^{-2}|\nabla^g\vartheta_s|^2_g+|\nabla^{g_D,2}\vartheta_s|_{g_D}^2\right)
\end{split}
\end{equation*}
for some uniform positive constant $C$. Integrating over $D$, using integration by parts together with Proposition \ref{decay-first-der-rad-der},
we next derive that
\begin{equation}\label{intermez-hessian-deldel}
\int_{D}|i\partial\bar{\partial}\vartheta_s|^2_{g}\,\omega_{D}^{n-1}\leq\frac{C}{r^{4\varepsilon+2}}+\int_{D}|\nabla^{g_D,2}\vartheta_s|^2_{g_{D}}\,\omega_{D}^{n-1}.
\end{equation}
Now, by Bochner formula applied to $(D,\,g_D)$ and the function $\vartheta_s$, we have that
\begin{equation}
\begin{split}\label{boc-for-D}
\Delta_{D}|\nabla^{g_D}\vartheta_s|^2_{g_D}&=2|\nabla^{g_D,2}\vartheta_s|^2_{g_D}+2\Ric(g_D)(\nabla^{g_D}\vartheta_s,\nabla^{g_D}\vartheta_s)+2g_D\left(\nabla^{g_D}\Delta_{D}\vartheta_s,\nabla^{g_D}\vartheta_s\right)\\
&\geq 2|\nabla^{g_D,2}\vartheta_s|^2_{g_D}+2g_D\left(\nabla^{g_D}\Delta_{D}\vartheta_s,\nabla^{g_D}\vartheta_s\right),
\end{split}
\end{equation}
where we have used the fact that $g_D$ has nonnegative Ricci curvature. (Ricci curvature bounded from below would be enough to complete the argument thanks to the decay on the gradient of $\vartheta_s$ that we have just proved above.) Integrating \eqref{boc-for-D} on $D$ and noticing that $\Delta_{D}\vartheta_s=2\tr_{\omega_D}(i\partial\bar{\partial}\vartheta_s)$ then leads to the bound
\begin{equation}
\begin{split}\label{est-hessian-pt-int}
\int_D|\nabla^{g_D,2}\vartheta_s|^2_{g_D}\,\omega_D^{n-1}&\leq \int_D|\nabla^{g_D}\Delta_{D}\vartheta_s|_{g_D}|\nabla^{g_D}\vartheta_s|_{g_D}\,\omega_D^{n-1}\\
 &\leq C\sup_{D\times\{r\}}|\nabla^{g_D}(i\partial\bar{\partial}\vartheta_s)|_{g_D}|\nabla^{g_D}\vartheta_s|_{g_D}\\
 &\leq \frac{C}{r^{\varepsilon}},
\end{split}
\end{equation}
where $C$ denotes a uniform positive constant that may vary from line to line. Here we have used Proposition \ref{prop-C^3-est} and the decay on the gradient of $\vartheta_s$ previously proved in the last line. Combining \eqref{inequ-mean-value-variant}, \eqref{intermez-hessian-deldel}, and \eqref{est-hessian-pt-int}, we can now infer that
\begin{equation*}
\begin{split}
0\leq\Delta_{C,X}\overline{\vartheta_s}(r)\leq \frac{C}{r^{4\varepsilon+2}}+\frac{C}{r^{\varepsilon}}.
\end{split}
\end{equation*}
We then have that
$$0\leq \frac{\partial}{\partial r}\left(e^{-\frac{r^{2}}{2}}X\cdot\overline{\vartheta_{s}}\right)\leq Cr^{1-\varepsilon}e^{-\frac{r^{2}}{2}}.$$
After integrating this differential inequality from $r$ to $r=+\infty$, we find that
\begin{equation*}
-C\int_r^{+\infty}s^{1-\varepsilon}e^{-\frac{s^{2}}{2}}\,ds\leq e^{-\frac{r^{2}}{2}}X\cdot\overline{\vartheta_{s}}(r)\leq 0.
\end{equation*}
Now, $\int_r^{+\infty}s^{1-\varepsilon}e^{-\frac{s^{2}}{2}}\,ds\leq C r^{-\varepsilon} e^{-\frac{r^{2}}{2}}$ for $r$ large enough which can be proved using integration by parts. In particular,
we have that
\begin{equation*}
-Cr^{-\varepsilon}\leq X\cdot\overline{\vartheta_{s}}(r)\leq 0.
\end{equation*}
Integrating once more yields the existence of a constant $\vartheta_{s}^{\infty}\in\R$ such that  $\vartheta_{s}^{\infty}\leq\overline{\vartheta_s}(r)\leq \vartheta_{s}^{\infty}+Cr^{-\varepsilon}$. The triangle inequality
applied to the oscillation estimates \eqref{osc-est-vartheta} and \eqref{osc-est-xvartheta}
then imply the desired estimates for $\vartheta_{s}$ and $X\cdot\vartheta_s$, respectively.
\end{proof}

As an intermediate step, we obtain a first rough decay estimate of the difference between the background metric and the metric resulting from the solution to \eqref{starstar-s}. More precisely, we have:

\begin{corollary}\label{coro-dec-met-bis}
Let $(\vartheta_s)_{0\,\leq\, s\,\leq\, 1}$ be a path of solutions in $\mathbb{R}\oplus C^{\infty}_{X,\,\beta}(M)$ to \eqref{starstar-s}. If $\alpha\in\left(0,\frac{1}{2}\right)$,
then there exists $C>0$ and $\varepsilon>0$ such that for all $s\in[0,\,1]$,
\begin{equation*}
\|f^{\frac{\varepsilon}{2}}\cdot i\partial\bar{\partial}\vartheta_s\|_{C^{0,2\alpha}_{\operatorname{\operatorname{loc}}}}\leq C.
\end{equation*}
\end{corollary}

\begin{proof}
It suffices to prove this estimate outside a compact set $W$ such that $\omega_s=\omega$ on $M\setminus W$. To this end, let $x\in M\setminus W$ and choose normal holomorphic coordinates in a ball $B_{g}(x,\iota)$ for some $\iota>0$ uniform in $x\in M$.
Let $g_{\tau\vartheta_s}^{i\bar{\jmath}}$ denote the components of the inverse of the K\"ahler metric associated to $\omega+i\partial\bar{\partial}(\tau\vartheta_s)$ in these coordinates and set $$a^{i\bar{\jmath}}:=\int_{0}^{1}
g_{\tau\vartheta_s}^{i\bar{\jmath}}\,d\tau.$$ Then we have that
 \begin{eqnarray*}
0&=&\log\left(\frac{\sigma_s^n}{\omega^n}\right)-\frac{X}{2}\cdot\vartheta_s\\
&=&\int_0^1\frac{d}{d\tau}\log\left(\frac{\omega_{\tau\vartheta_s}^n}{\omega^n}\right)\,d\tau-\frac{X}{2}\cdot\vartheta_s\\
&=&\left(\int_0^1 g_{\tau\vartheta_s}^{i\bar{\jmath}}\,d\tau\right)\partial_i\partial_{\bar{\jmath}}\vartheta_s-\frac{X}{2}\cdot\vartheta_s\\
&=&a^{i\bar{\jmath}}\partial_i\partial_{\bar{\jmath}}\vartheta_s-\frac{X}{2}\cdot\vartheta_s.
\end{eqnarray*}
Now, by Corollary \ref{coro-equiv-metrics}, $\|a^{i\bar{\jmath}}\|_{C_{\operatorname{\operatorname{loc}}}^{0,2\alpha}}$ is uniformly bounded from above and $a^{i\bar{\jmath}}\geq \Lambda^{-1}\delta^{i\bar{\jmath}}$ on $B_{g}(x,\iota)$ for some uniform constant $\Lambda>0$. Therefore, by considering $\frac{X}{2}\cdot \vartheta_s$ as a source term, the Schauder estimates imply that
\begin{equation*}
\begin{split}
\|\vartheta_s-\vartheta_{s}^{\infty}\|_{C^{2,\,2\alpha}(B_g(x,\iota/2))}&\leq C\Bigg(\left\|X\cdot\vartheta_s\right\|_{C^{0,2\alpha}(B_g(x,\iota))}+\|\vartheta_s-\vartheta_{s}^{\infty}\|_{C^{0}(B_g(x,\iota))}\Bigg)\\
&\leq Cf(x)^{-\frac{\varepsilon}{2}}
\end{split}
\end{equation*}
for some uniform positive constant $C=C\left(n,\alpha,\omega\right)$. Here we have used Proposition \ref{decay-first-der-rad-der} and Corollary \ref{coro-dec-met} in the last line. The desired rough a priori decay estimate on $i\partial\bar{\partial}\vartheta_s$ and its H\"older semi-norm now follow.
\end{proof}

The next result proves a sharp decay at infinity on the $C^0$-norm of the difference between a solution to \eqref{starstar-s} and its limit at infinity.

\begin{theorem}\label{theo-a-priori-wei-est}
 Let $(\vartheta_s)_{0\,\leq\, s\,\leq\, 1}$ be a path of solutions in $\mathbb{R}\oplus C^{\infty}_{X,\,\beta}(M)$ to \eqref{starstar-s}. Then there exist $R_0>0$ and $C>0$ such that for $s\in[0,\,1]$,
\begin{equation*}
| \vartheta_s-\vartheta_{s}^{\infty}|\leq \frac{C}{f^{\frac{\beta}{2}}},\qquad f\geq R_0,
\end{equation*}
where $\vartheta_{s}^{\infty}\in\R$ is as in Corollary \ref{coro-dec-met} {and $\beta$ is as in Theorem \ref{mainthm}(v)}. Moreover, there exists $C>0$ such that $\|\vartheta_s\|_{\mathcal{D}^{2,\,2\alpha}_{X,\beta}}\leq C$.
\end{theorem}

\begin{proof}
Linearising \eqref{starstar-s} around $g$ outside a compact set to obtain \eqref{lin-CMA-infinity} and using the uniform equivalence of the metrics $h_{s}$ and $g$ given by Corollary \ref{coro-equiv-metrics-0} together with the bounds of Corollary \ref{coro-dec-met-bis}, we obtain the improved estimate
\begin{equation*}
0\leq\Delta_{g,\,X}\vartheta_s\leq Cr^{-2\varepsilon}.
\end{equation*}
Akin to the proof of Claims \ref{cla-c-0-bd-mean-val} and \ref{cla-c-0-orth-mean-val}, one estimates $X\cdot\overline{\vartheta_s}$ and $\vartheta_s-\overline{\vartheta_s}$ separately. Estimating the former can be reduced to an ODE which gives $X\cdot\overline{\vartheta_s}=O(r^{-2\varepsilon})$ uniformly in $s\in[0,\,1]$, and by integrating from $r$ to $r=+\infty$, we obtain $\overline{\vartheta_s}-\vartheta_{s}^{\infty}=O(r^{-2\varepsilon})$. The latter estimate uses the Poincar\'e inequality on $D$ endowed with its metric $g_D$. By assumption, $\lambda^D>\beta>0$ is the first non-zero eigenvalue of the spectrum of the Laplacian on $D$, and so one has that $\vartheta_s-\overline{\vartheta_s}=O(r^{-\min\{\beta,2\varepsilon\}}).$ Combining these two estimates, one arrives at the fact that $\vartheta_s-\vartheta_{s}^{\infty}=O(r^{-\min\{\beta,2\varepsilon\}})$ which is a strict improvement of Corollary \ref{coro-dec-met}, {provided that $\varepsilon<\beta$}.

Next, invoking local parabolic Schauder estimates established in [\eqref{est-sch-loc-para}, Claim \ref{claim-a-priori-rough-bd-hih-der}] with $k=0$ applied to the linearisation of \eqref{starstar-s} around the background metric $g$ outside a compact set as in \eqref{lin-CMA-infinity} yields the existence of a positive constant $C$ such that for $R\geq R_0$,
\begin{equation*}
\begin{split}
\|\vartheta_s-\vartheta_{s}^{\infty}\|_{C^{2,\,2\alpha}_{X,\min\{\beta,2\varepsilon\}}}&\leq C\left(\|\vartheta_s-\vartheta_{s}^{\infty}\|_{C^{0}_{X,\min\{\beta,2\varepsilon\}}}+\|i\partial\overline{\partial}\vartheta_s\|_{C^{0,2\alpha}_{X,\min\{\beta,2\varepsilon\}}}\|i\partial\overline{\partial}\vartheta_s\|_{C^0(r\geq R)}\right)+C(R)\\
&\leq C\|\vartheta_s-\vartheta_{s}^{\infty}\|_{C^{0}_{X,\min\{\beta,2\varepsilon\}}}+C\|\vartheta_s-\vartheta_{s}^{\infty}\|_{C^{2,\,2\alpha}_{X,\min\{\beta,2\varepsilon\}}}R^{-\min\{\beta,2\varepsilon\}}+C(R),
\end{split}
\end{equation*}
where we have invoked local uniform estimates given by Propositions \ref{prop-C^2-est} and \ref{prop-C^3-est}. By choosing $R$ large enough and absorbing the relevant terms, one finds in particular that $\|\vartheta_s-\vartheta_{s}^{\infty}\|_{C^{2,\,2\alpha}_{X,\min\{\beta,2\varepsilon\}}}\leq C$ for some uniform positive constant $C$. This implies that $|i\partial\overline{\partial}\vartheta_s|_g=O(r^{-\min\{\beta,2\varepsilon\}})$.

By iterating the previous steps a finite number of times, the decay on $\vartheta_s$ is multiplied by $2$ with each iteration until it eventually reaches the threshold decay $r^{-\beta}$.
\end{proof}

We now present the weighted $C^{4}$-estimate.

\begin{prop}[Weighted $C^4$ a priori estimate]\label{prop-C4-est}
Let $(\vartheta_s)_{0\,\leq\, s\,\leq\, 1}$ be a path of solutions in $\mathbb{R}\oplus C^{\infty}_{X,\,\beta}(M)$ to \eqref{starstar-s}. If $\alpha\in\left(0,\frac{1}{2}\right)$, then there exists $C>0$ such that for all $s\in[0,\,1]$,
\begin{equation}
\|\vartheta_s-\vartheta_{s}^{\infty}\|_{C^{4,\,2\alpha}_{X,\,\beta}}\leq C.\label{a-priori-wei-3-alp}
\end{equation}
\end{prop}

\begin{proof}
In order to prove the a priori bound on the $C^{4,\,2\alpha}_{X,\,2}$-norm of $\vartheta_s-\vartheta_{s}^{\infty}$, we first establish the following uniform decay on the third derivatives of $\vartheta_s-\vartheta_{s}^{\infty}$.

\begin{claim}\label{dec-C3-inf}
There exists $C>0$ such that for all $s\in[0,1]$,
\begin{equation*}
\|\nabla^g \vartheta_s\|_{C^{2,\,2\alpha}_{X,\beta}}\leq C.
\end{equation*}
In particular,
\begin{equation*}
|\nabla^g\partial\overline{\partial}\vartheta_s|_g\leq \frac{C}{r^{\beta}}.
\end{equation*}
\end{claim}

\begin{proof}[Proof of Claim \ref{dec-C3-inf}]
We differentiate the linearisation of \eqref{starstar-s} around the background metric $g$ outside a compact set as given in \eqref{lin-CMA-infinity} to get schematically on $\{r\geq R\}$ with $R$ sufficiently large:
\begin{equation}
\begin{split}\label{lin-first-der-dec}
\Delta_{g,X}\left(\nabla^g\vartheta_s\right)&=\nabla^g\vartheta_s+Q(\partial\overline{\partial}\vartheta_s,\nabla^g\partial\overline{\partial}\vartheta_s),\\
\|Q(\partial\overline{\partial}\vartheta_s,\nabla^g\partial\overline{\partial}\vartheta_s)\|_{C^{0,2\alpha}_{X,\beta}}&\leq C \|\nabla^g\partial\overline{\partial}\vartheta_s\|_{C^{0,2\alpha}_{X,\beta}}\|\partial\overline{\partial}\vartheta_s\|_{C^{0,2\alpha}(r> R)}\leq \frac{C}{R^\beta}\|\nabla^g\partial\overline{\partial}\vartheta_s\|_{C^{0,2\alpha}_{X,\beta}}.
\end{split}
\end{equation}
Here we have used Theorem \ref{theo-a-priori-wei-est} in the last inequality. In particular, as in the proof of Theorem \ref{theo-a-priori-wei-est},
by choosing $R$ large enough and absorbing the non-linear term on the right-hand side of \eqref{lin-first-der-dec}, thanks to Proposition \ref{prop-loc-reg} together with Theorem \ref{theo-a-priori-wei-est}, one is led to the bound
\begin{equation*}
\|\nabla^g \vartheta_s\|_{C^{2,\,2\alpha}_{X,\beta}}\leq C.
\end{equation*}
In particular, the desired decay on $|\nabla^g\partial\overline{\partial}\vartheta_s|_g$ holds true.
\end{proof}

By Proposition \ref{prop-loc-reg}, in order to establish \eqref{a-priori-wei-3-alp}
it suffices to estimate the $C^{2,\,2\alpha}_{X,2}$-norm of the right-hand side of the linearisation of \eqref{starstar-s} around the background metric $g$ as given in \eqref{lin-CMA-infinity} once it is localized on $\{r> R\}$ for $R$ sufficiently large. As in the proof of Claim \ref{dec-C3-inf}, the linearisation of \eqref{starstar-s}
around the background metric $g$ outside a compact set as given in \eqref{lin-CMA-infinity} 
gives schematically on $\{r> R\}$:
\begin{equation*}
\begin{split}
\Delta_{g,\,X}\vartheta_s&=Q(\partial\overline{\partial}\vartheta_s),\\
\|Q(\partial\overline{\partial}\vartheta_s)\|_{C^{2,\,2\alpha}_{X,\beta}}&\leq C\Biggl(\|\vartheta_s-\vartheta_{s}^{\infty}\|^2_{C^{2,\,2\alpha}_{X,\beta}}+\|\partial\overline{\partial}\vartheta_s\|_{C^{2,\,2\alpha}_{X,\beta}}\|\partial\overline{\partial}\vartheta_s\|_{C^{0,2\alpha}(r> R)}\\
&\qquad+\|\nabla^g\partial\overline{\partial}\vartheta_s\|_{C^{0,2\alpha}_{X,\beta}}\|\nabla^g\partial\overline{\partial}\vartheta_s\|_{C^{0}(r> R)}\Biggr)\\
&\leq C\left(1+R^{-\beta}\|\vartheta_s-\vartheta_{s}^{\infty}\|_{C^{4,\,2\alpha}_{X,\beta}}+\|\vartheta_s-\vartheta_{s}^{\infty}\|_{C^{4,\,2\alpha}_{X,\beta}}\|\nabla^g\partial\overline{\partial}\vartheta_s\|_{C^{0}(r> R)}\right)\\
&\leq C\left(1+R^{-\beta}\|\vartheta_s-\vartheta_{s}^{\infty}\|_{C^{4,\,2\alpha}_{X,\beta}}\right)
\end{split}
\end{equation*}
for some positive uniform constant that may vary from line to line. Here we have used Theorem \ref{theo-a-priori-wei-est} in the second and third inequalities together with Claim \ref{dec-C3-inf}
in the last inequality. In particular, Theorem \ref{iso-sch-Laplacian-pol} applied to $\vartheta_s-\vartheta_{s}^{\infty}$ and $k=2$ and $\alpha\in\left(0,\frac{1}{2}\right)$ gives for some constant $C$ independent of $R$ the following bound:
\begin{equation*}
\|\vartheta_s-\vartheta_{s}^{\infty}\|_{C^{4,\,2\alpha}_{X,\beta}}\leq C(R)+CR^{-\beta}\|\vartheta_s-\vartheta_{s}^{\infty}\|_{C^{4,\,2\alpha}_{X,\beta}}.
\end{equation*}
This yields the expected a priori estimate after absorbing the last term on the right-hand side of the previous estimates into the left-hand side.
\end{proof}

The next proposition gives the a priori higher order weighted estimates. Since its proof is along the same lines as that of Proposition \ref{prop-C4-est}, we omit it.

\begin{prop}[Higher order weighted estimates]\label{high-order-est-prop}
Let $(\vartheta_s)_{0\,\leq\, s\,\leq\, 1}$ be a path of solutions in $\mathbb{R}\oplus C^{2k+2,\,2\alpha}_{X,\,\beta}(M)$ to \eqref{starstar-s} for $k\geq 1$. If $\alpha\in\left(0,\frac{1}{2}\right)$ and  if there exists $C_{k,\,\alpha}>0$ such that for all $s\in[0,\,1]$, $\|\vartheta_s\|_{\mathcal{D}^{2k+2,2\alpha}_{X,\,\beta}}\leq C_{k,\,\alpha}$, then there exists $C_{k+1,\,\alpha}>0$ such that for all $s\in[0,\,1]$,
\begin{equation*}
\|\vartheta_s\|_{\mathcal{D}^{2(k+1)+2,2\alpha}_{X,\,\beta}}\leq C_{k+1,\,\alpha}.\label{a-priori-wei-high-alp}
\end{equation*}
\end{prop}

\subsection{Completion of the proof of Theorem \ref{mainthm}(v)}\label{sec-proof-main-thm}

We finally prove Theorem \ref{mainthm}(v).
Set
\begin{equation*}
\begin{split}
S&:=\left\{s\in[0,\,1]\,|\,\textrm{there exists $\psi_s\in\mathcal{M}_{X,\,\beta}^{\infty}(M)$
satisfying \eqref{star-s}}\right\}.
\end{split}
\end{equation*}
Note that $S\neq\emptyset$ since $0\in S$ (take $\psi_{0}=0$).

We first claim that $S$ is open. Indeed, this follows from Theorem \ref{Imp-Def-Kah-Ste}: if $s_0\in S$, then by Theorem \ref{Imp-Def-Kah-Ste}, there exists $\epsilon_{0}>0$ such that for all $s\in(s_{0}-\epsilon_{0},\,s_{0}+\epsilon_{0})$, there exists a solution $\psi_{s}\in \mathcal{M}^{4,\,2\alpha}_{X,\,\beta}(M)$
to $\eqref{star-s}$ with data $F_{s}\in\left(\mathcal{C}^{2,\,2\alpha}_{X,\,\beta}(M)\right)_{\omega,\,0}$. Since the data $F_{s}$ lies in $\mathcal{C}^{\infty}_{X,\,\beta}(M)$,
Theorem \ref{Imp-Def-Kah-Ste} ensures that for each
$s$ in this interval, $\psi_{s}\in \mathcal{M}^{\infty}_{X,\,\beta}(M)$.
It follows that $(s_{0}-\epsilon_{0},\,s_{0}+\epsilon_{0})\cap[0,\,1]\subseteq S$.

We next claim that $S$ is closed. To see this, take a sequence
$(s_k)_{k\,\geq\,0}$ in $S$ converging to some $s_{\infty}\in S$.
Then for $F_k:=F_{s_{k}}$, $k\geq 0$, the corresponding solutions $\psi_{s_k}=:\psi_k$,
$k\geq 0$, of \eqref{star-s} satisfy
 \begin{equation}
(\omega+i\partial\bar{\partial}\psi_k)^n=
e^{F_{k}+\frac{X}{2}\cdot\psi_k}\omega^{n},\qquad k\geq 0.\label{MA-seq}
\end{equation}
It is straightforward to check that the sequence $(F_{{k}})_{k\,\geq\,0}$
 is uniformly bounded in $\mathcal{C}^{2,\,2\alpha}_{X,\,\beta}(M)$.
 As a consequence, the sequence $(\psi_k)_{k\,\geq\,0}$ is
 uniformly bounded in $\mathcal{M}^{4,\,2\alpha}_{X,\,\beta}(M)$
by Proposition \ref{prop-C4-est}. Indeed, recall the correspondence between solutions of \eqref{star-s} and \eqref{starstar-s}: $\psi_k$ is a solution to \eqref{star-s} if and only if $\vartheta_{s_k}=\psi_{s_k}-\Phi_{s_k}$ is a solution to \eqref{starstar-s}.  The
Arzel\`a-Ascoli theorem therefore allows us to pull out a subsequence of
 $(\psi_k)_{k\,\geq\,0}$ that converges to some $\psi_{\infty}\in C^{4,\,2\alpha'}_{\operatorname{\operatorname{loc}}}(M)$,
 $\alpha'\in(0,\alpha)$. As $(\psi_k)_{k\,\geq\,0}$ is uniformly bounded in
  $\mathcal{M}^{4,\,2\alpha}_{X,\,\beta}(M)$,
  $\psi_{\infty}$ will also lie in $\mathcal{M}^{4,\,2\alpha}_{X,\,\beta}(M)$.
  We need to show that\linebreak $(\omega+i\partial\bar{\partial}\psi_{\infty})(x)>0$
at every point $x\in M$. For this, it
 suffices to show that
  $(\omega+i\partial\bar{\partial}\psi_{\infty})^n(x)>0$
  for every $x\in M$. This is seen to hold true by letting $k$ tend to $+\infty$
   (up to a subsequence) in \eqref{MA-seq}. The fact that $\psi_{\infty}\in\mathcal{M}^{\infty}_{X,\,\beta}(M)$
   follows from Proposition \ref{high-order-est-prop}.

Finally, as an open and closed non-empty subset of $[0,\,1]$, connectedness of
$[0,\,1]$ implies that $S=[0,\,1]$. This completes the proof of the Theorem \ref{mainthm}(v).

\bibliographystyle{amsalpha}

\bibliography{ref2}

\end{document}